\documentclass[11pt]{amsart}
\reversemarginpar  
\usepackage{eucal}
\usepackage{amsthm}
\usepackage{amsfonts}
\usepackage{amsmath}
\usepackage{amssymb}
\usepackage{mathrsfs}
\usepackage{epsfig}
\usepackage{cases}
\newtheorem{theorem}{\bf Theorem}[section]
\newtheorem{proposition}[theorem]{\bf Proposition}
\newtheorem{corollary}[theorem]{\bf Corollary}
\newtheorem{lemma}[theorem]{\bf Lemma}
\newtheorem{notation}[theorem]{\bf Notation}

\newtheorem{remark}[theorem]{\bf Remark}
\newtheorem{definition}[theorem]{\bf Definition}

\newcommand{\overbar}[1]{\mkern 1.5mu\overline{\mkern-1.5mu#1\mkern-1.5mu}\mkern 1.5mu}

\newcommand{\z}{\mathbb Z_+}

\newcommand{\hl }{\mathcal{H}}
\newcommand{{\ran}}{\mbox{\rm ran}~}

\newcommand{\hla}{(\mathcal{H}, K^\alpha)\otimes(\mathcal{H}, K^\beta)}

\newcommand{\pa}{\beta \bar{\partial}_iK^\alpha(\cdot,w)\otimes K^\beta(\cdot,w)-\alpha K^\alpha(\cdot,w)\otimes\bar{\partial}_iK^\beta(\cdot,w)}

\newcommand{\ktilab}{K^{\alpha+\beta}(z,w)\partial_i\bar{\partial}_j\log K(z,w)}

\newcommand {\D} {\mathbb{D}}
\newcommand\quotient[2]{
        \mathchoice
            {
                \text{\raise1ex\hbox{$#1$}\Big/\lower1ex\hbox{$#2$}}%
            }{
                #1\,/\,#2
            }
            {
                #1\,/\,#2
            }
            {
                #1\,/\,#2
            }
    }

\usepackage{mathrsfs}

\numberwithin{equation}{section}

\begin{document}

\title [Decomposition of the tensor product of two Hilbert modules ]{Decomposition of the tensor product of two Hilbert modules}
\author[S. Ghara ]{Soumitra Ghara}
\address[S. Ghara]{Department of Mathematics, Indian Institute of Science,
Bangalore 560012, India} \email{ghara90@gmail.com}

\author[G. Misra]{Gadadhar Misra}
\address[G. Misra]{Department of Mathematics, Indian Institute of Science,
Bangalore 560012, India} \email{gm@math.iisc.ac.in}

\thanks{The work of the first named author  was supported by CSIR SPM Fellowship  (Ref. No. SPM-07/079(0242)/2016-EMR-I). The work  of  the  second named  author was supported  by the J  C Bose Fellowship of the DST and CAS II of the UGC. Many of the results in  this paper are from the PhD  thesis of the first named author submitted  to the Indian Institute   of Science in the year 2018.}

\date{}
\begin{abstract}
Given a pair of positive real numbers $\alpha, \beta$ and a sesqui-analytic  function $K$ on a bounded domain $\Omega \subset \mathbb C^m$, in this paper, we investigate the properties of the sesqui-analytic function $\mathbb K^{(\alpha, \beta)}:= K^{\alpha+\beta}\big(\partial_i\bar{\partial}_j\log K\big )_{i,j=1}^ m,$
taking values in $m\times m$ matrices. One of the key findings is that 
$\mathbb K^{(\alpha, \beta)}$ is non-negative definite whenever $K^\alpha$ and $K^\beta$ are non-negative definite.
In this case, a realization of the Hilbert module determined by the kernel $\mathbb K^{(\alpha,\beta)}$  is obtained. 
Let $\mathcal M_i$, $i=1,2,$ be two Hilbert modules over the polynomial ring $\mathbb C[z_1, \ldots, z_m]$. Then $\mathbb C[z_1, \ldots, z_{2m}]$  acts naturally on the tensor product $\mathcal M_1\otimes \mathcal M_2$. The restriction of this action to the polynomial ring $\mathbb C[z_1, \ldots, z_m]$ 
obtained using the restriction map $p \mapsto p_{|\Delta}$ leads to a natural decomposition of the tensor product $\mathcal M_1\otimes \mathcal M_2$, which is investigated. 
Two of the initial pieces in this decomposition are identified. 
\end{abstract}

 \renewcommand{\thefootnote}{}
%
%
%

 \maketitle

\section{Introduction}

\subsection{Hilbert Module} We will find it useful to state many of our results in the language of Hilbert modules. The notion of a Hilbert module was introduced by R. G. Douglas (cf. \cite{Douglasmodule}), which we recall below. We point out that in the original definition, the module multiplication was assumed to be continuous in both the variables. However, for our purposes, it would be convenient to assume that   
it is continuous only in the second variable. 

\begin{definition}[Hilbert module]
A Hilbert module $\mathcal M$ over a unital, complex algebra $\mathbb A$ consists of a complex Hilbert space $\mathcal M$ and 
a map $(a,h)\mapsto a \cdot h$, $a\in \mathbb A, h\in \mathcal M$, such that
\begin{itemize}
\item[\rm (i)]
$1\cdot h=h$
\item[\rm (ii)]
$(ab)\cdot h=a\cdot (b\cdot h)$
\item[\rm (iii)]
$(a+b)\cdot h=a\cdot h+b\cdot h$
\item [\rm(iv)]
for each $a$ in $\mathbb A$, the map $\mathbf m_{a}:\mathcal M \to \mathcal M$, defined by $\mathbf m_a(h)=a\cdot h,\;h\in \mathcal M,$
is a bounded linear operator on $\mathcal M$.
\end{itemize}
\end{definition}

A closed subspace $\mathcal S$ of $\mathcal M$ is said to be a submodule of $\mathcal M$ if $\mathbf m_a h\in \mathcal S$ for all $ h\in \mathcal S$ and $a\in \mathbb A$.  
The quotient module $\mathcal Q :=\quotient{\mathcal H}{\mathcal S}$ is  the Hilbert space $\mathcal S^\perp$, where the module multiplication is defined to be the compression of the module multiplication on $\mathcal H$ to the subspace $\mathcal S^\perp$, that is, the module action on $\mathcal Q$ is given by  $\mathbf m_a (h) = P_{\mathcal S^\perp} (\mathbf m_a h)$, $h\in \mathcal S^\perp$. Two Hilbert modules  $\mathcal M_1$ and $\mathcal M_2$ over $\mathbb A$ are said to be isomorphic if there exists a unitary operator $U:\mathcal M_1\to \mathcal M_2$ such that $U(a\cdot h)=a\cdot Uh$, $a\in \mathbb A$, $h\in \mathcal M_1$.

Let $K:\Omega\times \Omega\to \mathcal M_k(\mathbb C)$
be a ses-qui analytic (that is holomorphic in first $m$-variables and anti-holomorphic in the second set of $m$-variables) non-negative definite kernel on a bounded domain $\Omega\subset \mathbb C^m$. It uniquely determines a Hilbert space $(\mathcal H,K)$ consisting of holomorphic functions on $\Omega$ taking values in
$\mathbb C^k$ possessing the following properties. For $w \in \Omega$, 
\begin{itemize}
\item[$\rm (i)$]  the vector valued function 
 $K(\cdot,w)\zeta$, $\zeta \in \mathbb C^k$, belongs to the Hilbert space $\mathcal H$
\item[$\rm (ii)$] 
$\left\langle f, K(\cdot,w)\zeta \right\rangle_\mathcal H 
= \left\langle f(w), \zeta \right\rangle_{{\mathbb C}^k},$ $f\in (\mathcal H, K)$.
\end{itemize}
Assume that the  operator of multiplication $M_{z_i}$ by the $i$th coordinate function $z_i$ is bounded on the Hilbert space $(\mathcal H, K)$ for $i=1,\ldots,m$. Then $(\mathcal H, K)$ may be realized as a Hilbert module over the polynomial ring 
$\mathbb C[z_1,\ldots,z_m]$ with the module action given by the point-wise multiplication:
$$\mathbf m_p(h) = p h,\; h\in (\mathcal H, K),~ p\in \mathbb C[z_1,\ldots,z_m].$$

 Let $K_1$ and $K_2$ be two scalar valued non-negative definite kernels defined on $\Omega\times\Omega$. 
It turns out that $(\hl, K_1)\otimes (\hl, K_2)$ is  the reproducing kernel Hilbert space with the reproducing kernel $K_1\otimes K_2$, where $K_1\otimes K_2:(\Omega\times \Omega)\times(\Omega\times \Omega)\to \mathbb C$ is given by 
$$(K_1\otimes K_2)(z,\zeta;w,\rho)=K_1(z,w)K_2(\zeta,\rho),\; \; z,\zeta,w,\rho\in \Omega.$$  
Assume that the multiplication operators 
$M_{z_i}$, $i=1,\ldots,m$, are bounded on $(\mathcal H, K_1)$ as well as on $(\mathcal H, K_2).$ 
Then $(\mathcal H, K_1) \otimes (\mathcal H, K_2)$ 
may be realized as a Hilbert module over $\mathbb C[z_1,\ldots,z_{2m}]$ with the module action
defined by  
$$\mathbf m_p(h) = p h,\: h\in (\mathcal H, K_1) \otimes (\mathcal H, K_2),\: p\in \mathbb C[z_1,\ldots,z_{2m}].$$
The module $(\mathcal H, K_1) \otimes (\mathcal H, K_2)$ admits a natural direct sum decomposition as follows. 

For a non-negative integer $k$, 
let $\mathcal A_k$ be the subspace of 
$(\mathcal H ,K_1)\otimes (\mathcal H, K_2)$ defined by
\begin{equation}\label{eqA_k}
\mathcal A_k:=\big\{f\in (\mathcal H ,K_1)\otimes (\mathcal H ,K_2)
:\big(\big(\tfrac{\partial}{\partial \zeta}\big)^{\boldsymbol i} f(z,\zeta)\big)_{|\Delta}=0,
\;|\boldsymbol i|\leq k\big\},
\end{equation}
where $\boldsymbol i \in {\mathbb Z}_+^m$, $|\boldsymbol i|=i_1+\cdots+i_m$, 
$\big(\tfrac{\partial}{\partial \zeta}\big)^{\boldsymbol i}=\frac{\partial^{|\boldsymbol i|}}{\partial \zeta_1^{i_1}\cdots\partial \zeta_m^{i_m}}$, and 
$\big(\big(\tfrac{\partial}{\partial \zeta}\big)^{\boldsymbol i} f(z,\zeta)\big)_{|\Delta}$ is 
the restriction of $\big(\tfrac{\partial}{\partial \zeta}\big)^{\boldsymbol i} f(z,\zeta)$
to the diagonal set $\Delta:=\{(z,z):z\in \Omega\}$.
It is easily verified that each of the subspaces $\mathcal A_k$ is closed and invariant under multiplication by any polynomial in $\mathbb C[z_1,\ldots, z_{2m}]$ and therefore they are sub-modules of $(\mathcal H ,K_1)\otimes (\mathcal H, K_2)$.
Setting $\mathcal S_0= \mathcal A_0^\perp$, $\mathcal S_k:= \mathcal A_{k-1} \ominus \mathcal A_{k}$, $k = 1, 2, \ldots$, we obtain  a direct sum decomposition of the Hilbert space 
$$
(\mathcal H ,K_1)\otimes (\mathcal H, K_2) = \bigoplus _{k=0}^\infty \mathcal S_k.$$
In this decomposition, the subspaces $\mathcal S_k \subseteq (\mathcal H ,K_1)\otimes (\mathcal H, K_2)$ are not necessarily sub-modules. Indeed, one may say they are semi-invariant modules following the  terminology commonly used in Sz.-Nagy--Foias model theory for contractions. We  study the compression of the module action to these subspaces analogous to the ones studied in operator theory.
Also, such a decomposition is similar to the Clebsch-Gordan formula, which describes the decomposition of the tensor product of two irreducible representations, say $\varrho_1$ and $\varrho_2$ of a group $G$ when restricted to the diagonal subgroup in $G\times G$:
$$\varrho_1(g) \otimes \varrho_2(g) = \bigoplus_k d_k\pi_k(g),$$
where $\pi_k,$ $k\in \z,$ are irreducible representation of the group $G$ and $d_k$, $k\in \z$, are natural numbers. However, the decomposition of the tensor product of two Hilbert modules cannot be expressed as the direct sum of submodules. 
Noting that $\mathcal S_0$ is a quotient module, describing all the semi-invariant modules $\mathcal S_k$, $k \geq 1$, would appear to be a natural question.  
To describe the equivalence classes of $\mathcal S_0$, $\mathcal S_1,\ldots$ etc., it would be useful to recall the notion of the push-forward of a module. 

Let  $\iota: \Omega \to \Omega\times \Omega$ be the map $\iota(z) = (z,z)$,  $z\in\Omega$. Any Hilbert module $\mathcal  M$ over the polynomial ring $\mathbb C[z_1,\ldots,z_{m}]$ may be thought of as a module $\iota_\star\mathcal M$ over the ring $\mathbb C[z_1,\ldots,z_{2m}]$ by re-defining the multiplication: $\mathbf m_p(h) = (p \circ \iota) h$, $h\in \mathcal M$ and $p \in  \mathbb C[z_1,\ldots,z_{2m}]$.
The module $\iota_\star\mathcal M$ over $\mathbb C[z_1,\ldots,z_{2m}]$ is defined to be the push-forward of the module $\mathcal M$ over $\mathbb C[z_1,\ldots,z_m]$  under the inclusion map $\iota$.

In \cite{Aro}, Aronszajn proved that the Hilbert space $(\mathcal H, K_1 K_2)$ corresponding to the point-wise product $K_1 K_2$ of two non-negative definite kernels $K_1$ and $K_2$ is obtained by the restriction of the functions in the tensor product $(\mathcal H, K_1) \otimes (\mathcal H, K_2)$ to the diagonal set $\Delta$. Building on his work, it was shown in \cite{D-M-V} that the restriction map is isometric  on the subspace $\mathcal S_0$ onto $(\mathcal H, K_1 K_2)$ intertwining the module actions on $\iota_\star(\mathcal H, K_1 K_2)$ and $\mathcal S_0$.  However, using the jet construction given below, it is possible to describe the quotient modules $\mathcal A_k^\perp$, $k\geq 0$. We reiterate that one of the main questions we address is  that of of describing the semi-invariant modules, namely, $ \mathcal S_1, \mathcal S_2, \ldots$. We have  succeed in describing only $\mathcal S_1$ only after assuming that the  pair of kernels is of the form $K^\alpha$, $K^\beta$, $\alpha, \beta   >0$, where the real power of a non-negative definite kernel is defined below. 

Let $\Omega\subset \mathbb C^m$ be a bounded domain and  $K:\Omega\times \Omega\to \mathbb C$ be a non-zero sesqui-analytic function. Let $t$ be a  real number. 
The function $K^t$ is defined in the usual manner,
namely $K^t(z,w) = \exp(t \log K(z,w))$, $z,w\in \Omega$, assuming that a continuous branch of the logarithm of $K$ exists on $\Omega \times \Omega$.
Clearly, $K^t$ is also sesqui-analytic. However, if $K$ is non-negative definite, then $K^t$ need not be non-negative definite unless $t$ is a natural number.  
A direct computation, assuming the existence of a continuous branch of  logarithm of $K$ on $\Omega\times \Omega$, shows that for $1\leq i, j\leq m$,
$$\partial_i \bar{\partial}_j \log K(z,w)=\frac{K(z,w)\partial_i 
\bar{\partial}_j K(z,w)-\partial_i K(z,w)\bar{\partial}_jK(z,w)} {K(z,w)^2},\;\;z,w\in \Omega, $$
where $\partial_i$ and $\bar{\partial}_j$ denote $\frac{\partial}{\partial z_i}$ and $\frac{\partial}{\partial \bar{w}_j}$, 
respectively.


For a sesqui-analytic function $K:\Omega\times\Omega\to \mathbb C$
satisfying $K(z,z)>0$, an alternative interpretation of $K(z,w)^t$ (resp. $\log K(z,w)$) is possible using the notion of polarization. The real analytic function  $K(z,z)^t$ (resp. $\log K(z,z)$) defined  on $\Omega$ extends to a unique sesqui-analytic function in some neighbourhood $U$ of the diagonal set $\{(z,z):z\in \Omega\}$ in $\Omega\times \Omega$.  If the principal branch of logarithm of $K$ exists on $\Omega\times \Omega$, then it is easy to verify that these two definitions of $K(z,w)^t$ (resp. $\log K(z,w)$) agree on the open set $U$.

In the particular case, when $K_1={(1-z\bar{w})^{-\alpha}}$ and $K_2={(1-z\bar{w})^{-\beta}}$, $\alpha,\beta>0$, the description of the semi-invariant modules $\mathcal S_k$, $k\geq 0$, is obtained from somewhat more general results  of  Ferguson and Rochberg.
\begin{theorem}[Ferguson-Rochberg,\cite{Ferguson-Rochberg}]
If $K_1(z,w)=\frac{1}{(1-z\bar{w})^\alpha}$ and $K_2(z,w)=\frac{1}{(1-z\bar{w})^{\beta}}$ on $\mathbb D\times\mathbb D$ for some $\alpha,\beta>0$, then the Hilbert modules $\mathcal S_k$ and $\iota_\star (\hl, {(1-z\bar{w})^{-(\alpha+\beta+2k)}})$ are isomorphic.
\end{theorem} 

In this paper, first we show that if $K^{\alpha}$ and $K^{\beta}$, $\alpha, \beta >0$, are two non-negative definite kernels on $\Omega$, then  function $\mathbb K^{(\alpha, \beta)}:\Omega \times \Omega \to \mathcal M_m(\mathbb C)$ defined by  
$$\mathbb K^{(\alpha, \beta)}(z,w)=K^{\alpha+\beta}(z,w)\Big(\,\big(\partial_i\bar{\partial}_j\log K\big)(z,w) \,\Big)_{i,j=1}^ m, \,\, z,w \in \Omega,$$ 	
is also a non-negative definite kernel. In this case, a description of the Hilbert module $\mathcal S_1$ is obtained. Indeed, it is shown that 
the Hilbert modules $\mathcal S_1$ and $\iota_\star \big(\mathcal H, \mathbb{K}^{(\alpha,\beta)}\big)$ are isomorphic.

\subsection{The jet construction}
For a bounded domain $\Omega\subset \mathbb C^m$, let $K_1$ and $K_2$ be two scalar valued non-negative kernels defined on $\Omega\times\Omega$. Assume that the multiplication operators $M_{z_i},$ $i=1,\ldots,m,$ are bounded on $(\mathcal H, K_1)$ as well as on $(\mathcal H, K_2).$ For a non-negative integer $k$, let $\mathcal A_k$ be the subspace defined in \eqref{eqA_k}.

Let $d$ be the cardinality of the set $\{\boldsymbol i\in \mathbb Z_+^m,|\boldsymbol i|\leq k\}$, which is $\binom{m+k}{m}$. Define the linear map 
$ J_k:(\mathcal H ,K_1)\otimes (\mathcal H, K_2)\to 
\rm Hol(\Omega\times\Omega, \mathbb{C}^d)$ by 
\begin{equation}\label{eqnjetm_ap}
(J_kf)(z,\zeta)=\sum_{|\boldsymbol i|\leq k}\big(\tfrac{\partial}{\partial \zeta}\big)^{\boldsymbol i} f(z,\zeta)\otimes e_{\boldsymbol i},~f\in (\mathcal H ,K_1)\otimes (\mathcal H, K_2),
\end{equation}
where $\big \{e_{\boldsymbol i}\big \}_{{\boldsymbol i} \in \mathbb Z_+^m,|\boldsymbol i|\leq k}$
is the standard orthonormal basis of $\mathbb C^d$. 
Let $R:\ran J_k\to \rm Hol(\Omega, \mathbb{C}^d)$ be the restriction map, that is, $R(\mathbf{h})=\mathbf{h}_{|\Delta},~\mathbf {h}\in \ran J_k$. Clearly, $\ker R J_k=\mathcal A_k$. Hence the map $R J_k: \mathcal A_k^\perp \to \rm Hol(\Omega, \mathbb{C}^d)$ is one to one. Therefore we can give a natural inner product on $ \ran {R J_k} $, namely, 
$$\langle RJ_k(f),R J_k(g)\rangle=\langle P_{\mathcal A_k^\perp}f,P_{\mathcal A_k^\perp}g\rangle,~f,g\in (\hl,K_1)\otimes (\hl,K_2).$$
In what follows, we think of $\ran {R J_k}$ as a Hilbert space equipped with this inner product.  The theorem stated below is a straightforward generalization of one of the main results from \cite{D-M-V}. 
\begin{theorem}\rm({\cite [Proposition 2.3]{D-M-V}})\label{thmjetcons}
Let $K_1,K_2:\Omega\times\Omega\to \mathbb C$ be two non-negative definite kernels. Then $\ran RJ_k$ is a reproducing kernel Hilbert space and its  reproducing kernel $J_k(K_1,K_2)_{|\rm res \, \Delta}$ is given by the formula 
$$J_k(K_1,K_2)_{|\rm res \, \Delta}(z,w):=\big(K_1(z,w)\partial^{\boldsymbol i}\bar{\partial}^{\boldsymbol j} K_2(z,w)\big)_{|\boldsymbol i|,|\boldsymbol j|=0}^k,\;\;z,w\in \Omega.$$
\end{theorem} 
Now for any polynomial $p$ in $z,\zeta$, define the operator $\mathcal T_p$ on $\ran RJ_k$ as 
$$(\mathcal Tp)(RJ_kf) = \sum_{|\boldsymbol l|\leq k}
\Big(\sum_{\boldsymbol q\leq \boldsymbol l}{\tbinom{\boldsymbol l}{\boldsymbol q}}\Big(\big(\tfrac{\partial}{\partial \zeta}\big)^{\boldsymbol q} p(z,\zeta)\Big)_{|{\Delta}}
\Big(\big(\tfrac{\partial}{\partial \zeta}\big)^{\boldsymbol l-\boldsymbol q} 
f(z,\zeta)\Big)_{|{\Delta}}\Big)\otimes e_{\boldsymbol l}, f\in (\hl, K_1)\otimes (\hl, K_2),$$
where $\boldsymbol l=(l_1,\ldots,l_m),
\boldsymbol q=(q_1,\ldots,q_m)\in \mathbb Z^m_+$, and $\boldsymbol q\leq \boldsymbol l$ means $q_i\leq l_i,~i=1,\ldots,m$ and $\binom{\boldsymbol l}{\boldsymbol q}=\binom{l_1}{q_1}\cdots\binom{l_m}{q_m}$.  The proof of the Proposition below follows from a straightforward computation using the Leibniz rule, the details are on page 378 - 379 of \cite{D-M-V}.    
\begin{proposition}\label{prop:jetq}
For any polynomial $p$ in $\mathbb C[z_1,\ldots,z_{2m}]$, the operator $P_{\mathcal A_k^\perp}{M_p}_{|{\mathcal A_k^\perp}}$ is unitarily equivalent to the operator $\mathcal T_p$
on $(\ran RJ_k)$.
\end{proposition} 
In section 4, we prove a  generalization of the theorem of Salinas for  all kernels of the form $J_k(K_1,K_2)_{|\rm res \, \Delta}$. In particular, we show that if $K_1,K_2:\Omega\times \Omega\to
\mathbb C$ are two sharp kernels (resp. generalized Bergman kernels), then so is the kernel 
$J_k(K_1,K_2)_{|\rm res \, \Delta}.$

In Section 5, we introduce the notion of a generalized Wallach set for an arbitrary non-negative definite kernel $K$ defined on a bounded domain $\Omega\subset \mathbb C^m$. Recall that the ordinary Wallach set associated with the Bergman kernel $B_{\Omega}$ of a bounded symmetric domain $\Omega$ is the set $\{ t > 0 : B_\Omega^t \mbox{\rm \: is non-negative definite}\}$. 
Replacing the Bergman kernel in the definition of the Wallach set by an arbitrary non-negative definite kernel $K$,  we define the ordinary Wallach set $\mathcal W(K)$.   More importantly, we introduce the generalized Wallach set   
$G\mathcal W(K)$ associated to the kernel $K$ to be the set $\{t \in \mathbb R :\; K^{t} \big(\partial_i \bar{\partial}_j \log K\big)_{i,j=1}^m \mbox{\rm\: is non-negative definite} \}$,
where we have assumed that $K^t$ is well defined for all $t\in \mathbb R$. 
In the particular  case of the Euclidean unit ball $\mathbb B_m$ in $\mathbb C^m$ and the Bergman kernel, the generalized Wallach set $G\mathcal W(B_{\mathbb B_m})$, $m >1$, is shown to be the set $\{t\in \mathbb R\,:\, t \geq 0\}$. If $m=1$, then it is the set  $\{t\in \mathbb R\,:\, t \geq -1\}$.

In Section 6, we  study quasi-invariant kernels. Let $J:{\rm Aut(\Omega)}\times \Omega\to  GL_k(\mathbb C)$ be a function such that $J(\varphi,\cdot)$ is holomorphic for each $\varphi$ in $\rm Aut(\Omega)$, where ${\rm Aut(\Omega)}$ is the group of all biholomorphic automorphisms of $\Omega$. A non-negative definite kernel $K:\Omega\times\Omega\to\mathcal M_k(\mathbb C)$ is said to be quasi-invariant with respect to $J$ if $K$ satisfies the following transformation rule:
\begin{equation*}
J(\varphi,z)K(\varphi(z),\varphi(w)){J(\varphi,w)}^*=K(z,w),~z,w\in \Omega,\; \varphi \in \rm Aut(\Omega).
\end{equation*}
It is shown that if $K:\Omega\times\Omega\to \mathbb C$ is a quasi-invariant kernel with respect to $J:
{\rm Aut(\Omega)}\times \Omega\to \mathbb C\setminus \{0\}$, then the kernel $K^{t} \big(\partial_i \bar{\partial}_j \log K\big)_{i,j=1}^m$ is also quasi-invariant with respect to
$\mathbb J$ whenever $t\in G\mathcal W(K)$, where $ \mathbb J(\varphi,z)=J(\varphi,z)^{t} D\varphi(z)^{\rm tr}$, $\varphi\in \rm Aut(\Omega),~ z\in \Omega.$  In particular, taking $\Omega\subset \mathbb C^m$ to be a bounded symmetric domain and setting $K$ to be the Bergman kernel $B_\Omega$,  in the language of \cite{Misra-Sastry}, we conclude that the multiplication tuple $\boldsymbol {M}_z$ on $(\hl, \boldsymbol B_\Omega^{(t)})$, where 
$\boldsymbol B_\Omega^{(t)}(z,w):= ( B_{\Omega}^t\partial_i\bar{\partial}_j\log B_{\Omega} )_{i,j=1}^m,$ is homogeneous with respect to the group ${\rm Aut}(\Omega)$ for $t$ in $G\mathcal W(B_\Omega)$.

\section{A new non-negative definite kernel}
The scalar version of the following lemma is well-known. However, the easy modifications necessary to  prove it in the case of $k\times k$ matrices are omitted. 
\begin{lemma}[Kolmogorov] \label{grahmnnd}
Let $\Omega\subset \mathbb C^m$ be a bounded domain, and let $\mathcal{H}$ be a Hilbert space. 
If $\phi_1,\phi_2,\ldots,\phi_k$
are  anti-holomorphic functions from $\Omega $ into $\mathcal{H}$, then  $K:\Omega\times \Omega\rightarrow \mathcal M_{k}(\mathbb{C})$
defined by $K(z,w)=\big( \left\langle \phi_j(w),\phi_i(z)\right\rangle_{\hl}\big)_{i,j=1}^k$, $z,w\in \Omega$, is a sesqui-analytic 
non-negative definite kernel. 
\end{lemma}

For any reproducing kernel Hilbert space $ (\hl, K)$, the following proposition, which is Lemma 4.1 of  \cite{Curtosalinas} is a basic tool in what follows. 
\begin{proposition}\label{derivativeofK}
Let $ K :\Omega\times \Omega\to \mathcal M_{k}(\mathbb{C})$ be a non-negative definite kernel. For every $\boldsymbol{i}\in \mathbb Z_{+}^m$, $\eta \in \mathbb{C}^k $ and $ w\in\Omega,$ we have  

\begin{itemize}
\item[\rm (i)] ${\bar{\partial}}^{\boldsymbol{i}}K(\cdot,w)\eta$ is in $(\hl, K),$ 
\item[\rm (ii)]$
\left \langle f,\bar{\partial}^{\boldsymbol{i}} K(\cdot,w)\eta \right \rangle _{(\mathcal{H}, K)}=\left \langle (\partial^{\boldsymbol i}f)(w),\eta\right \rangle_{\mathbb C^k}, 
f\in (\hl, K).$
\end{itemize}
\end{proposition}
Here and throughout  this paper, 
for any non-negative definite kernel $K:\Omega\times \Omega \to \mathcal M_k(\mathbb C)$ and $\eta \in \mathbb{C}^k$,  let 
$\;{\bar{\partial}}^{\boldsymbol{i}}K(\cdot,w)\eta$ denote the function $\big(\tfrac{\partial}{\partial\overbar{w}_1}\big)^{i_1}\cdots \big(\tfrac{\partial}{\partial \overbar{w}_m}\big)^{i_m}K(\cdot,w)\eta$ 
and $({{\partial}}^{\boldsymbol i}f)(z)$ be the function $\big(\tfrac{\partial}{\partial {z_1}}\big)^{i_1}\cdots \big(\tfrac{\partial}{\partial {z_m}}\big)^{i_m} f (z), \boldsymbol{i}=(i_1,\ldots,i_m)\in \mathbb Z_{+}^m$.

\begin{proposition}\label{k^2curv}
Let $\Omega$ be a bounded domain in $\mathbb{C}^m$ and  
$K:\Omega \times \Omega \to \mathbb C $ be a sesqui-analytic function. 
Suppose that  $K^\alpha$ and $K^\beta$, defined on $\Omega\times \Omega,$ are non-negative definite for some $\alpha,\, \beta >0$. 
Then the function 
$$K^{\alpha+\beta}(z,w)\Big(\,\big(\partial_i\bar{\partial}_j\log K\big)(z,w) \,\Big)_{i,j=1}^ m, \,\, z,w \in \Omega,$$ 	
is a non-negative definite kernel on $\Omega\times\Omega$ taking values in $\mathcal M_m(\mathbb C)$.
\end{proposition}

\begin{proof} For ${1\leq i\leq m}$, set $\phi_i(z)= 
\beta \bar{\partial}_iK^\alpha(\cdot,z)\otimes K^\beta(\cdot,z)-
\alpha K^\alpha(\cdot,z)\otimes\bar{\partial}_i K^\beta(\cdot,z).$    
From Proposition $\ref{derivativeofK}$, it follows that each 
$\phi_i$ is a function from $\Omega$ into the Hilbert space $(\mathcal{H}, K^\alpha)\otimes( \mathcal{H}, K^\beta)$. 
Then we have
\begin{align*}
\left\langle \phi_j(w), \phi_i(z)\right\rangle
& =\beta^2\partial_i\bar{\partial}_jK^{\alpha}(z,w)K^{\beta}(z,w)+
\alpha^2K^{\alpha}(z,w)\partial_i\bar{\partial}_j K^{\beta}(z,w)\\
&\quad-\alpha\beta\big(\partial_iK^{\alpha}(z,w)
\bar{\partial}_j K^{\beta}(z,w)+\bar{\partial}_j K^{\alpha}(z,w)\partial_iK^{\beta}(z,w)\big)\\
&=\beta^2\big(\alpha(\alpha-1)K^{\alpha+\beta-2}(z,w)\partial_i K(z,w)\bar{\partial}_j K(z,w)
+\alpha K^{\alpha+\beta-1}(z,w)\partial_i\bar{\partial}_jK(z,w)\big)\\
& \quad +\alpha^2\big(\beta(\beta-1)K^{\alpha+\beta-2}(z,w)\partial_iK(z,w)\bar{\partial}_j K(z,w)
+\beta K^{\alpha+\beta-1}(z,w)\partial_i\bar{\partial}_j K(z,w)\big)\\
& \quad\quad \quad\quad  -2\alpha^2\beta^2K^{\alpha+\beta-2}(z,w)\partial_i K(z,w)\bar{\partial}_j K(z,w)\\
& = (\alpha^2\beta+\alpha\beta^2)K^{\alpha+\beta-2}(z,w)\big(K(z,w)\partial_i\bar{\partial}_j K(z,w)
-\partial_iK(z,w)\bar{\partial}_j K(z,w)\big)\\
&=\alpha\beta(\alpha+\beta) \ktilab.
\end{align*}
An application of  Lemma $\ref{grahmnnd}$ now completes the proof. 
\end{proof} 
The particular case, when $\alpha=1=\beta$ occurs repeatedly in the following. We therefore record it separately as a corollary.  
\begin{corollary}\label{cork^2curv}
Let $\Omega$ be a bounded domain in $\mathbb{C}^m$. 
If $K:\Omega\times\Omega\to\mathbb C$ is a non-negative definite 
kernel, then 
$$K^2(z,w)\big (\,\big(\partial_i \bar{\partial}_j \log K\big) (z,w)\,\big )_{i,j=1}^m$$
is also a non-negative definite kernel, defined on $\Omega\times \Omega$, taking values in $\mathcal M_m(\mathbb C)$. 
\end{corollary}
A more substantial corollary is the following, which is taken from \cite{{infinitelydivisiblemetric}}.  Here we provide a slightly different proof. Recall that a non-negative definite kernel $K:\Omega\times \Omega\to \mathbb C$ is said to be \emph{infinitely divisible} if for all $t>0$,  $K^t$ is also non-negative definite. 
\begin{corollary}
Let $\Omega$ be a bounded domain in $\mathbb{C}^m$.
Suppose that $ K:\Omega\times\Omega \to \mathbb C $ is an infinitely divisible kernel. Then the function $\big (\;\big(\partial_i\bar{\partial}_j \log K \big)(z,w)\;\big )_{i,j=1}^m $
is a non-negative definite kernel taking values in $\mathcal{M}_m(\mathbb{C})$. 
\end{corollary}

\begin{proof} For $t>0,$    $K^t(z,w)$ is non-negative definite by hypothesis. Then it follows, from Corollary $\ref{cork^2curv}$,  that 
$\big(\,K^{2t}\partial_i\bar{\partial}_j\log K^t(z,w)\, \big)_{i,j=1}^m$ is non-negative definite. Hence   $ \big(\,K^{2t}\partial_i\bar{\partial}_j\log K(z,w)\,\big)_{i,j=1}^m$ is non-negative definite for all $t>0$.
Taking the limit as $t\to 0,$ we conclude that $\big(\,\partial_i\bar{\partial}_j\log K(z,w)\,\big)_{i,j=1}^m $ is non-negative definite. 
\end{proof}

\begin{remark}
It is known that even if $K$ is a positive definite kernel,  $\big(\;\big(\partial_i\bar{\partial}_j\log K\big)(z,w)\;\big)_{i,j=1}^ m$ need not be a non-negative definite kernel. In fact, $ \big(\,\big(\partial_i\bar{\partial}_j\log K\big)(z,w)\,\big)_{i,j=1}^ m $ 
is non-negative definite if and only if $ K $ is infinitely divisible (see  \cite[Theorem 3.3]{infinitelydivisiblemetric}).

Let $K:\mathbb D\times \mathbb D \to \mathbb C$ be the positive definite kernel given by $ K(z,w)= 1+\sum_{i=1}^\infty a_i z^i\bar{w}^i,$  $z,w\in \mathbb D$, $a_i>0$.  
For any $ t>0$,  a direct computation gives  
\begin{align*}
\big(K^t \partial \bar{\partial} \log K\big)(z,w)&
= \big(1+\sum_{i=1}^\infty a_i z^i\bar{w}^i\big)^t\;\partial \bar{\partial}\big(\sum_{i=1}^\infty a_i z^i\bar{w}^i-\frac{(\sum_{i=1}^\infty a_i z^i\bar{w}^i)^2}{2}+\cdots\big)\\
&=(1+ta_1z\bar{w}+\cdots)
(a_1+2(2a_2-a_1^2)z\bar{w}+\cdots)\\
&=a_1+(4a_2+(t-2)a_1^2)z\bar{w}+\cdots.
\end{align*}
Thus, if $~ t < 2 $, one may choose $ a_1, a_2>0 $ such that
$~ 4a_2+(t-2)a_1^2 < 0$. Hence $\big(K^t\partial \bar{\partial} \log K\big)(z,w)$ 
cannot be a non-negative definite kernel. Therefore, in general,  for $ \big(\;\big(K^t\partial_i\bar\partial_j\log K\big)(z,w)\;\big )_{i,j=1}^ m $ 
to be non-negative definite, it is necessary that $t\geq 2$. 
\end{remark}

\subsection[Multiplication operator on $\big (\mathcal H, \mathbb K \big )$]{ Boundedness of the  multiplication operator on {${\big(\mathcal H, \mathbb K \big )}$}}
For  $\alpha, \beta >0$,  let $\mathbb K^{(\alpha, \beta)}$ denote the kernel $ K^{\alpha+\beta}(z,w)\Big (\,\, \big (\partial_i \bar{\partial}_j\log K\big) (z,w)\,\, \Big )_{i,j=1}^m$. 
If $\alpha=1=\beta$, then we write $\mathbb K$ instead of $\mathbb K^{(1, 1)}$. 
For a holomorphic function $f:\Omega\to \mathbb C,$ the operator $M_f$ of multiplication by $f$ on the linear space ${\rm Hol}(\Omega, \mathbb C^k)$
is defined by the rule $M_f h =f \, h,$ $h\in {\rm Hol}(\Omega, \mathbb C^k),$ where $(f\,h) (z) = f(z) h(z)$, $z\in \Omega$. The  boundedness criterion for the multiplication operator $M_f$ restricted to the Hilbert space $(\mathcal H, K)$ is well-known for the case of positive definite kernels. In what follows, often  we have to work with a kernel which is merely non-negative definite. A precise statement is given below. The first part is from  \cite{PaulsenRaghupati} and the second part follows from the observation that the boundedness of the operator $\sum_{i=1}^n M_iM_i^*$ is equivalent to the non-negative definiteness of the kernel  $ (c^2 - \langle z,w\rangle)K(z,w) $ for some positive constant $c$. 
\begin{lemma}\label{lembounded}
Let $\Omega\subset\mathbb C^m$ be a bounded domain and $K:\Omega \times 
\Omega \to \mathcal M_k(\mathbb C)$ be a non-negative 
definite kernel. 
\begin{itemize}
\item [\rm (i)] For any holomorphic function $f:\Omega \to \mathbb C$, the operator $M_f$ of multiplication by $f$ is bounded on $(\mathcal H, K)$ if and only if 
 there exists 
a constant $c >0$ such that $\big (c^2 -f(z)\overbar{f(w)} \big ) K(z,w)$ is non-negative definite on $\Omega\times \Omega$. In case $M_f$ is bounded, $\|M_f\|$ is the infimum of all
$c>0$ such that $\big (c^2 -f(z)\overbar{f(w)} \big ) K(z,w)$ is non-negative definite.
\item [\rm (ii)] The operator $M_{z_i}$ of multiplication by the $i$th coordinate function $z_i$ is bounded on $(\hl, K)$  for $i=1,\ldots,m$, if and only if there exists a constant $c>0$ such that $ \big(c^2 - \langle z,w\rangle\big)K(z,w) $ is non-negative definite.
\end{itemize}
\end{lemma}

As we have pointed out, the distinction between the non-negative definite kernels and the positive definite ones is very significant. Indeed, as shown in \cite[Lemma 3.6]{Curtosalinas}, it is interesting that if the operator $\boldsymbol M_z:=(M_{z_1},\ldots,M_{z_m})$ is bounded on $(\mathcal H, K)$ for some non-negative definite kernel $K$ such that $K(z,z)$, $z\in \Omega$, is invertible, then $K$ is positive definite.  A direct proof of this statement, different from the inductive proof of Curto and Salinas  is in the PhD thesis of the first named author  \cite{SGhara}.

It is natural to ask if the operator $M_f$ is bounded on $(\mathcal H, K)$, then if it remains bounded on the Hilbert space $(\mathcal H, \mathbb K)$. From the Theorem stated below, in particular,  it follows that the operator $M_f$ is bounded on $(\mathcal H, \mathbb K)$ whenever it is bounded on $(\mathcal H, K)$.

\begin{theorem}\label{boundedness}
Let $\Omega\subset\mathbb C^m$ be a bounded domain and $K:\Omega \times \Omega \to \mathbb C$ be a non-negative definite kernel. Let $f:\Omega \to \mathbb C$ be an arbitrary holomorphic function. Suppose that there exists a constant $c>0$ such that $\big(c^2-f(z)\overbar{f(w)}\big)K(z,w)$ is 
non-negative definite on $\Omega\times\Omega$. Then the function 
$\big( c^2-f(z)\overbar{f(w)}\big)^2
\mathbb K (z,w)$
is non-negative definite on $\Omega\times \Omega$.
\end{theorem}

\begin{proof}
Without loss of generality, we assume that $f$ is non-constant
and $K$ is non-zero. The function $G(z,w):=\big(c^2-f(z)\overbar{f(w)}\big)K(z,w)$ 
is non-negative definite on $\Omega\times \Omega$ by hypothesis.  We claim that $|f(z)|< c$ for all $z$ in $\Omega$.
If not, then by the open mapping theorem, there exists an open set $\Omega_0\subset \Omega$ such that $|f(z)|>c$, $z\in \Omega_0$. Since $\big(c^2-|f(z)|^2\big)K(z,z)\geq 0$, it follows that $K(z,z)=0$ for all $z\in \Omega_0$. Now, let $h$ be an arbitrary vector in $(\hl, K)$. Clearly, 
$|h(z)|= |\left\langle h, K(\cdot,z)\right \rangle|\leq \|h\|\|K(\cdot,z)\|=
\|h\|{K(z,z)}^{\frac{1}{2}}=0$ for all $z\in \Omega_0$. 
Consequently, $h(z)=0$ on $\Omega_0$.
Since $\Omega$ is connected and $h$ is holomorphic, it follows that $h=0$. This contradicts the assumption that $K$ is non-zero verifying the validity of our claim. 

From the claim, we have that the function $c^2-f(z)\overbar{f(w)}$ is non-vanishing on $\Omega\times \Omega.$ 
Therefore, the kernel $K$ can be written as the product
$$K(z,w)=\frac{1}{\big(c^2-f(z)\overbar{f(w)}\big)}G(z,w),~z,w\in \Omega.$$
Since $|f(z)|<c$ on $\Omega$, the function $\frac{1}{\big(c^2-f(z)\overbar{f(w)}\big)}$ has a convergent power series expansion, namely, 
$$\frac{1}{\big(c^2-f(z)\overbar{f(w)}\big)} =\sum_{n=0}^\infty 
\frac{1}{c^{2(n+1)}} f(z)^n\overbar{f(w)}^n ,~~z,w\in \Omega.$$
Therefore it defines a non-negative definite kernel on $\Omega\times \Omega$. Note that
\begin{align*}
& \big(K(z,w)^2\partial_i\bar{\partial}_j\log K(z,w) \big)_{i,j=1}^m\\
&\quad \quad = \,\Big(K(z,w)^2  \partial_i\bar{\partial}_j\log \frac{1}{\big(c^2-f(z)\overbar{f(w)}\big)}\,
\Big)_{i,j=1}^m + \Big(\,K(z,w)^2\partial_i\bar{\partial}_j\log G(z,w)\,\Big)_{i,j=1}^m\\
&\quad \quad = \frac{1}{\big(c^2-f(z)\overbar{f(w)}\big)^2}\left(\,K(z,w)^2\Big(\partial_i f(z)\overbar{\partial_j f(w)}\,\Big)_{i,j=1}^m + G(z,w)^2\Big(\,\partial_i\bar{\partial}_j\log G(z,w)\, \Big)_{i,j=1}^m\right),
\end{align*}
where for the second equality, we have used that
$$\partial_i\bar{\partial}_j\log \frac{1}{\big(c^2-f(z)\overbar{f(w)}\big)}=\frac{\partial_i f(z)\overbar{\partial_j f(w)}}{\big(c^2-f(z)\overbar{f(w)}\big)^2},\;z,w\in\Omega,\;1\leq i,j\leq m.$$
Thus
\begin{align}\label{eqncurvineq} 
\begin{split}
&\big(c^2-f(z)\overbar{f(w)}\big)^2\mathbb K(z,w)\\
&\quad\quad \quad  = K(z,w)^2\Big(\, \partial_i f(z)\overbar{\partial_j f(w)}\,\Big)_{i,j=1}^m + \Big(\, G(z,w)^2\partial_i\bar{\partial}_j\log G(z,w)\, \Big)_{i,j=1}^m.
\end{split}
\end{align}
By Lemma $\ref{grahmnnd},\;$ the function $\big(\,\partial_i f(z)\overbar{\partial_j f(w)}\,\big)_{i,j=1}^m$ 
is non-negative definite on $\Omega\times\Omega$. Thus the product
$K(z,w)^2\big(\,\partial_i f(z)\overbar{\partial_jf(w)}\,\big)_{i,j=1}^m$ is also non-negative definite on $\Omega\times\Omega$.
Since $G$ is non-negative definite on $\Omega\times\Omega$, by Corollary $\ref{cork^2curv}$, the function $\big(\,G(z,w)^2\partial_i\bar{\partial}_j\log G(z,w)\,\big)_{i,j=1}^m$ 
is also non-negative definite on $\Omega\times\Omega$. The proof is now complete since the sum of two non-negative definite kernels remains non-negative definite.
\end{proof}
A sufficient condition for  the boundedness of the multiplication operator on the Hilbert space $\big(\mathcal H, \mathbb K\big)$
is an immediate Corollary.

\begin{corollary}\label{thmboundedness}
Let $\Omega\subset\mathbb C^m$ be a bounded domain and $K:\Omega \times \Omega \to \mathbb C$ 
be a non-negative definite kernel. Let $f:\Omega \to \mathbb C$ be a holomorphic function. Suppose that the multiplication operator $M_f$ on $(\mathcal{H}, K)$ is bounded. Then the multiplication operator $M_f$ is also bounded on $(\mathcal H, \mathbb K)$. 
\end{corollary}
\begin{proof}
Since the operator $M_f$ is bounded on $(\mathcal{H}, K)$, by Lemma $ \ref{lembounded}$, we find a constant $c>0$ such that 
$\big(c^2-f(z)\overbar{f(w)}\big)K(z,w)$ is non-negative definite on $\Omega\times\Omega$. Then, by Theorem \ref{boundedness}, it follows that 
$\big(c^2-f(z)\overbar{f(w)}\big)^2
\mathbb K(z,w)$ is non-negative 
definite on $\Omega\times\Omega$. 
Also, from the proof of Theorem \ref{boundedness}, we have that $\big(c^2-f(z)\overbar{f(w)}\big)^{-1}$ is non-negative 
definite on $\Omega\times\Omega$ (assuming that $f$ is non-constant). Hence $\big(c-f(z)\overbar{f(w)}\big)\mathbb K(z,w)$, being the product of two non-negative definite kernels, 
is non-negative definite on $\Omega\times\Omega$. An application of Lemma 
\ref{lembounded}, a second time,  completes the proof.  
\end{proof}

A second Corollary provides a sufficient condition for the positive definiteness of the kernel $\mathbb K$.
\begin{corollary}
Let $\Omega\subset\mathbb C^m$ be a bounded domain and $K:\Omega \times \Omega \to \mathbb C$ be a non-negative definite kernel satisfying $K(w,w)>0$, $w\in \Omega$. Suppose that the multiplication operator $M_{z_i}$ on $(\hl, K)$ is bounded  for $i=1,\ldots, m$. Then the kernel $\;\mathbb K$ is positive definite on $\Omega\times \Omega.$

\begin{proof}
By Corollary \ref{cork^2curv}, we already have that $\mathbb K$
is non-negative definite. Moreover, since $M_{z_i}$ on $(\hl, K)$ is bounded  for $i=1,\ldots, m$, it follows from Theorem \ref{thmboundedness} that $M_{z_i}$ is bounded on $(\hl, \mathbb K)$ also. Therefore, 
using \cite[Lemma 3.6]{Curtosalinas}, we see that $\mathbb K$ is  positive definite if $\mathbb K(w,w)$ is invertible for all $w\in \Omega$. 
To verify this, set $$\phi_i(w)=\bar{\partial}_iK(\cdot,w)\otimes K(\cdot,w)-K(\cdot,w)\otimes \bar{\partial}_iK(\cdot,w),\; 1\leq i\leq m.$$
From the proof of Proposition \ref{k^2curv}, we see that $\mathbb K(w,w)=\frac{1}{2}\big( \left\langle \phi_j(w),\phi_i(w)\right\rangle\big)_{i,j=1}^m$. Therefore $\mathbb K(w,w)$ is invertible if the vectors $\phi_1(w),\ldots,\phi_m(w)$
are linearly independent.
Note that for $w=(w_1,\ldots,w_m)$ in $\Omega$ and $j=1,\ldots,m$, we have
$(M_{z_j}-w_j)^*K(\cdot,w)=0$.
Differentiating this equation with respect to $\bar{w}_i$, we obtain 
$$(M_{z_j}-w_j)^*\bar{\partial}_i K(\cdot,w)=\delta_{ij}K(\cdot,w),\,\, 1\leq i,j \leq m.$$ Thus 
\begin{equation}\label{eqnpostivedef}
\big((M_{z_j}-w_j)^*\otimes I\big)\big(\phi_i(w)\big)=\delta_{ij}K(\cdot,w)\otimes K(\cdot,w),\,\, 1\leq i,j \leq m.
\end{equation}

Now assume that $\sum_{i=1}^m c_i\phi_i(w)=0$ for some scalars $c_1,\ldots,c_m$. Then, for $1\leq j\leq m $, we have that
$\sum_{i=1}^m \big((M_{z_j}-w_j)^*\otimes I\big)\big(\phi_i(w)\big)=0$.  
Thus, using \eqref{eqnpostivedef}, we see that
$ c_j K(\cdot,w)\otimes K(\cdot,w)=0$. Since $K(w,w)>0$, we conclude that $c_j=0$. Hence the vectors $\phi_1(w),\ldots,\phi_m(w)$ are linearly independent. This completes the proof.
\end{proof}

\end{corollary}

\begin{remark}
Recall that an operator $T$ is said to be a $2-$hyper contraction if
$I-T^*T\geq 0$ and $I-2T^*T+{T^*}^2T^2\geq 0$. If $K:\mathbb D\times \mathbb D\to \mathbb C$ is a non-negative definite kernel, then  it is not hard to verify that the adjoint $M^*_z$ of the multiplication by  the coordinate function $z$ is a $2-$hyper contraction on $(\hl, K)$ if and only if $(1-z\bar{w})^2K$ is  non-negative definite. It follows from Theorem \ref{boundedness} that if $M_z^*$ on $(\hl, K)$
is a contraction, then $M_z^*$ on $(\hl, \mathbb K)$ is a $2-$hyper contraction. 
\end{remark}
\section
{Realization of $\big(\mathcal H, \mathbb K^{(\alpha, \beta)} \big )$}\label{secrealization}
Let $\Omega \subset \mathbb{C}^m$ be a bounded domain and $K:\Omega\times\Omega\to \mathbb C$ be a sesqui-analytic function.
Suppose that the functions $K^{\alpha}$ and $K^{\beta}$ are non-negative definite for some $\alpha,\beta>0$. 
In this section, we give a description of the Hilbert space $\big(\mathcal H, \,\mathbb K^{(\alpha, \beta)}\big)$.
As before, we set 
\begin{equation}\label{eqnphi_iw}
\phi_i(w)=\pa, \, 1\leq i\leq m,~w\in \Omega.
\end{equation}
Let $\mathcal{N}$ be the subspace of 
$(\mathcal{H}, K^\alpha)\otimes (\mathcal{H}, K^\beta)$ which is the closed linear span of the vectors 
$$\big\{\,\phi_i(w):w\in \Omega,\,1\leq i\leq m \, \big\}.$$
From the definition of $\mathcal N$, it is not easy to determine which vectors are in it. 
A useful alternative description of the space $\mathcal N$ is given below. 

Recall that $K^{\alpha}\otimes K^{\beta}$
is the reproducing kernel for the Hilbert space $\hla$, where the kernel $K^{\alpha}\otimes K^{\beta}$ on $(\Omega\times\Omega)\times (\Omega\times \Omega )$ is given by 
$$K^{\alpha}\otimes K^{\beta}(z,\zeta;z^\prime,\zeta')=
K^\alpha(z,z')K^\beta(\zeta,\zeta'),$$  
 $z=(z_1,\ldots,z_m)$, $\zeta=
(\zeta_1,\ldots,\zeta_m)$, $z'=(z_{m+1},\ldots,z_{2m})$, $\zeta'=(\zeta_{m+1},\ldots,\zeta_{2m})$ 
are in $\Omega$.
We realize the Hilbert space 
$(\mathcal H, K^\alpha)\otimes (\mathcal H, K^\beta)$ as a space consisting of holomorphic functions on $\Omega\times \Omega$. 
Let  $\mathcal{A}_0$ and $\mathcal{A}_1$ be the subspaces defined by
$$\mathcal{A}_0=\big\{f\in(\mathcal{H}, K^\alpha)\otimes 
(\mathcal{H}, K^\beta):f_{|\Delta}=0\big\}$$
and $$\mathcal{A}_1=\big\{f\in(\mathcal{H},K^\alpha)
\otimes (\mathcal{H}, K^\beta):f_{|\Delta}=
(\partial_{m+1} f)_{|\Delta}=\cdots=(\partial_{2m} f)_{|\Delta}=0\big\},$$
where $\Delta$ is the diagonal set $\{(z,z)\in \Omega\times \Omega:z\in \Omega\}$, 
$\partial_i f$ is the derivative of $f$ with respect to the $i$th variable,
and $f_{|\Delta}$, $(\partial_i f)_{|\Delta}$ denote the restrictions to the set $\Delta$ of the functions $f$, $\partial_i f$, respectively.
It is easy to see that both $\mathcal A_0$ and $\mathcal A_1$ are closed subspaces of the Hilbert space $(\mathcal{H}, K^\alpha)\otimes 
(\mathcal{H}, K^\beta)$ and $\mathcal A_1$ is a closed subspace of 
$\mathcal A_0.$ 

Now observe that, for $1\leq i \leq m,$ we have 
\begin{align}\label{eqndescrpN}
\begin{split}
&\bar{\partial}_i (K^{\alpha}\otimes K^{\beta})(\cdot,(z^\prime,\zeta^\prime))=\bar{\partial}_iK^\alpha(\cdot,z^\prime)\otimes K^\beta(\cdot,\zeta^\prime),~z^\prime,\zeta^\prime\in \Omega\\
&\bar{\partial}_{m+i}(K^{\alpha}\otimes K^{\beta}) (\cdot,(z^\prime,\zeta^\prime))
= K^{\alpha}(\cdot,z^\prime)\otimes \bar{\partial}_{i} K^\beta(\cdot,\zeta^\prime),~z^\prime,\zeta^\prime\in \Omega.
\end{split}
\end{align}
Hence, taking $z^\prime=\zeta^\prime=w, $ we see that 
\begin{equation}\label{eqnlemspan}
\phi_i(w) =\beta\bar{\partial}_i(K^{\alpha}\otimes K^{\beta})(\cdot,(w,w))-\alpha\bar{\partial}_{m+i}(K^{\alpha}\otimes K^{\beta})(\cdot,(w,w)).
\end{equation}
We now state a useful lemma on the Taylor coefficients of an analytic functions.  The straightforward proof follows from the chain rule  \cite[page
8]{Rudinunitball}, which is  omitted. 

\begin{lemma}\label{lemvanish}
Suppose that $f:\Omega\times\Omega\to\mathbb C$ is a holomorphic 
function satisfying $f_{|_{\Delta}}=0.$ Then
$$(\partial_i f)_{|\Delta}+(\partial_{m+i}f)_{|\Delta}=0, \quad 1 \leq i \leq m. $$
\end{lemma}
An alternative description of the subspace $\mathcal N$ of $(\mathcal{H}, K^\alpha)\otimes (\mathcal{H}, K^\beta)$ is provided below. 
\begin{proposition}\label{lemmaspan}
$\mathcal{N}=\mathcal{A}_0\ominus\mathcal{A}_1.$
\end{proposition}

\begin{proof}
For all $z\in \Omega$, we see that 
$$\phi_i(w)(z,z)=\alpha \beta K^{\alpha+\beta-1}(z,w)\bar \partial_iK(z,w)
-\alpha \beta K^{\alpha+\beta-1}(z,w)\bar \partial_iK(z,w)=0. $$
Hence each $\phi_i(w)$, $w\in \Omega, 1\leq i\leq m$, belongs to $\mathcal A_0$ and consequently, $\mathcal N \subset \mathcal A_0.$ Therefore, to complete the proof of
the proposition, it is enough to show that 
$\mathcal A_0\ominus\mathcal N =\mathcal A_1$.        
        
To verify this, note that $f\in {\mathcal N}^{\perp}$ if and only if $\left\langle f,\phi_i(w)\right\rangle=0$, $1\leq i \leq m$,  $w\in \Omega$. Now, in view of \eqref{eqnlemspan} and Proposition \ref{derivativeofK}, we have that 
\begin{align}\label{eqperp}
\begin{split}
\left\langle f,\phi_i(w)\right\rangle = 
&\left\langle f,\beta\bar{\partial}_i (K^{\alpha}\otimes K^{\beta})(\cdot,(w,w))-\alpha\bar{\partial}_{m+i}(K^{\alpha}\otimes K^{\beta})(\cdot,(w,w))\right\rangle \\
=&\beta({\partial}_{i}f)(w,w)
-\alpha ({\partial}_{m+i}f)(w,w),\,\,1\leq i \leq m,\,\, w\in \Omega.
\end{split}
\end{align}
Thus $f\in {\mathcal N}^{\perp}$ if and only if the function $\beta\,({\partial}_{i}f)_{|\Delta}
-\alpha\,({\partial}_{m+i}f)_{|\Delta}=0$, $1 \leq i \leq m$.  
Combining this with Lemma $\ref{lemvanish},$ we see that any $f\in \mathcal A_0\ominus\mathcal N,$ satisfies 
\begin{align*}
&\beta({\partial}_{i}f)_{|\Delta}
-\alpha ({\partial}_{m+i}f)_{|\Delta}=0,\\
&({\partial}_{i}f)_{|\Delta}+({\partial}_{m+i}f)_{|\Delta}=0, 
\end{align*}
for $1 \leq i \leq m$.
Therefore, we have $({\partial}_{i}f)_{|\Delta}=({\partial}_{m+i}f)_{|\Delta}=0$, $1\leq i \leq m$.
Hence $f$ belongs to $\mathcal{A}_1.$ 

Conversely, let $f\in \mathcal{A}_1$. In particular, $f\in \mathcal{A}_0.$ Hence invoking Lemma $\ref{lemvanish}$ once again, we see that
$$(\partial_{i}f)_{|\Delta}+(\partial_{m+i}f)_{|\Delta}=0, \,\, 1 \leq i \leq m.$$ 
Since $f$ is in $\mathcal A_1$, $({\partial}_{m+i}f)_{|\Delta}=0,$ $1\leq i \leq m$, by definition. Therefore,  $({\partial}_{i}f)_{|\Delta}=({\partial}_{m+i}f)_{|\Delta}=0,~1\leq i \leq m$, which implies
$$\beta({\partial}_{i}f)_{|\Delta}
-\alpha ({\partial}_{m+i}f)_{|\Delta}=0,~1\leq i \leq m.$$
Hence $f\in \mathcal A_0\ominus \mathcal N$, completing the proof.
\end{proof}
We now give a description of the Hilbert space 
$\big(\mathcal H, \,\mathbb K^{(\alpha, \beta)}\big)$.
Define a linear map $\mathcal R_1:\hla \to {\rm Hol}(\Omega,\mathbb{C}^m)$ by setting 
\begin{equation}\label{the map R_1}
\mathcal R_1(f)=
\frac{1}{\sqrt{\alpha\beta(\alpha+\beta)}}\begin{pmatrix}
(\beta\partial_{1}f-\alpha\partial_{m+1}f)_{|\Delta}\\
\vdots\\
(\beta\partial_{m}f-\alpha\partial_{2m}f)_{|\Delta}
\end{pmatrix}
\end{equation}
for $f\in \hla$ and note that 
\begin{equation}\label{eqnthemapR_1}
\mathcal R_1(f)(w)=
\frac{1}{\sqrt{\alpha\beta(\alpha+\beta)}}\begin{pmatrix}
\left\langle f, \phi_1(w)\right\rangle\\
\vdots\\
\left\langle f, \phi_m(w)\right\rangle
\end{pmatrix},~w\in \Omega.
\end{equation}

From Equation \eqref{eqnthemapR_1}, it is easy to see that 
$\ker \mathcal R_1=\mathcal N^\perp.$ We have $\mathcal N =\mathcal A_0\ominus\mathcal{A}_1$, see  Proposition \ref{lemmaspan}. Therefore, 
$\ker \mathcal R_1^\perp = \mathcal A_0\ominus\mathcal{A}_1$ 
and the map ${\mathcal R_1}_{|\mathcal{A}_0\ominus\mathcal{A}_1}\to \ran\mathcal R_1$ is bijective. 
Require this map to be a unitary by defining an appropriate inner product on $\ran {\mathcal R}_1$, that is, 
Set 
\begin{equation}\label{innerR}
\left\langle\mathcal R_1(f),\mathcal R_1(g)\right\rangle
:=\left\langle P_{\mathcal{A}_0\ominus\mathcal{A}_1}f,P_{\mathcal{A}_0\ominus\mathcal{A}_1}g\right\rangle,~f,g\in \hla,
\end{equation}
where $P_{\mathcal{A}_0\ominus\mathcal{A}_1}$ is the orthogonal projection of $\hla$ onto the subspace $\mathcal{A}_0\ominus\mathcal{A}_1$. This choice of the inner product on the range of  $\mathcal R_1$ makes the map $\mathcal R_1$ unitary.

\begin{theorem}\label{realization}
Let $\Omega \subset \mathbb{C}^m$ be a bounded domain and $K:\Omega\times\Omega\to \mathbb C$ be a sesqui-analytic function.
Suppose that the functions $K^{\alpha}$ and $K^{\beta}$ are non-negative definite for some $\alpha,\beta>0$. Let $\mathcal R_1$ be the map defined by $\eqref{the map R_1}$. Then the Hilbert space determined by the non-negative definite kernel $\mathbb K^{(\alpha,\beta)}$ coincides with the space ${\rm ran}~\mathcal R_1$ and the inner product given by \eqref{innerR} on ${\rm ran}~\mathcal R_1$ agrees with the one induced by the kernel ~$\mathbb K^{(\alpha,\beta)}$. 
\end{theorem}

\begin{proof}
Let $\{e_1,\ldots,e_m\}$ be the standard orthonormal basis of $\mathbb C^m$. For $1\leq i,j\leq m,$ from the proof of Proposition $\ref{k^2curv}$, we have 
\begin{align}
\left\langle \phi_j(w),\phi_i(z)\right\rangle&=
\alpha\beta(\alpha+\beta) \ktilab\\
&=\alpha\beta(\alpha+\beta)\left\langle\mathbb K^{(\alpha,\beta)}(z,w)e_j,e_i\right\rangle_{\mathbb C^m},~z,w\in \Omega.
\end{align}
Therefore, from \eqref{eqnthemapR_1}, it follows that for all $w\in \Omega$ and $1\leq j\leq m,$
$$\mathcal R_1(\phi_j(w))=
\sqrt{\alpha\beta(\alpha+\beta)}\mathbb K^{(\alpha,\beta)}(\cdot,w)e_j.$$
Hence, for all $w\in \Omega$ and $\eta\in \mathbb C^m,$ $\mathbb K^{(\alpha,\beta)}(\cdot,w)\eta$ belongs to 
$\ran \mathcal R_1.$ Let $\mathcal R_1(f)$ be an arbitrary element in $\ran \mathcal R_1$ where 
$f\in \mathcal{A}_0\ominus\mathcal{A}_1.$ Then 
\begin{align*}
\left\langle \mathcal R_1(f),\mathbb K^{(\alpha,\beta)}(\cdot,w)e_j\right\rangle &
=\frac{1}{\sqrt{\alpha\beta(\alpha+\beta)}}\left\langle \mathcal R_1(f),
\mathcal R_1(\phi_j(w))\right\rangle\\
&=\frac{1}{\sqrt{\alpha\beta(\alpha+\beta)}}\left\langle f,\phi_j(w)\right\rangle\\
&=\frac{1}{\sqrt{\alpha\beta(\alpha+\beta)}}(\beta\partial_j
f(w,w)-\alpha\partial_{m+j}f(w,w))\\
&=\left\langle \mathcal R_1(f)(w),e_j\right\rangle_{\mathbb C^m},
\end{align*}
where the second equality follows since both $f$ and $\phi_j(w)$ belong to $\mathcal{A}_0\ominus\mathcal{A}_1$. 
This completes the proof.
\end{proof}
We obtain the density of polynomials in $\big(\,\mathcal H, \mathbb K^{(\alpha,\beta)}\big)$ as a consequence of this theorem. 
Let $\boldsymbol z = (z_1, \ldots , z_m)$ and let $\mathbb C[\boldsymbol z]:=\mathbb C[z_1,\ldots ,z_{m}]$ denote the ring of polynomials in $m$-variables. 
The 
following proposition gives a sufficient condition for density of $\mathbb C[\boldsymbol z]\otimes \mathbb C^m$ in the Hilbert space $\big(\,\mathcal H, \mathbb K^{(\alpha,\beta)}\big)$.

\begin{proposition}\label{proppoldense}
Let $\Omega \subset \mathbb{C}^m$ be a bounded domain and $K:\Omega\times\Omega\to \mathbb C$ be a sesqui-analytic function such that the functions $K^{\alpha}$ and $K^{\beta}$ are non-negative definite on $\Omega\times\Omega$ for some $\alpha,\beta>0$. Suppose that both the Hilbert spaces $(\mathcal H, K^\alpha)$ and
$(\mathcal H, K^{\beta})$ contain the polynomial ring $\mathbb C[\boldsymbol z]$ as a dense subset. Then 
the Hilbert space $\big(\mathcal H, \mathbb K^{(\alpha,\beta)}\big)$ contains  the ring $\mathbb C[\boldsymbol z]\otimes \mathbb C^m$ as a dense subset.  
\end{proposition}

\begin{proof}
Since $C[\boldsymbol z]$ is dense in both the Hilbert spaces 
$(\hl, K^{\alpha})$ and $(\hl, K^{\beta})$, it follows that 
$C[\boldsymbol z]\otimes C[\boldsymbol z]$, which is 
$\mathbb C[z_1,\ldots,z_{2m}]$,
is contained in the Hilbert space $(\mathcal H, K^\alpha)\otimes(\mathcal H, K^\beta)$ and is dense in it. Since $\mathcal R_1$ maps $(\mathcal H, K^\alpha)\otimes(\mathcal H, K^\beta)$ onto $\big(\mathcal H,  \mathbb K^{(\alpha,\beta)}\big)$,
to complete the proof, it suffices to show that $\mathcal R_1(\mathbb C[z_1,\ldots,z_{2m}])=\mathbb C[\boldsymbol z]\otimes \mathbb C^m$. It is easy to see that $\mathcal R_1(\mathbb C[z_1,\ldots,z_{2m}])\subseteq \mathbb C[\boldsymbol z]\otimes \mathbb C^m$. Conversely, if $\sum_{i=1}^m p_i(z_1,\ldots,z_m)\otimes e_i$ is an arbitrary
element of $\mathbb C[\boldsymbol z]\otimes \mathbb C^m$, then it is easily verified that the function $p(z_1,\ldots,z_{2m}):=\sqrt{\tfrac{\alpha\beta}{\alpha+\beta}}\sum_{i=1}^m (z_i-z_{m+i})p_i(z_1,\ldots,z_m)$ belongs to $\mathbb C[z_1,\ldots,z_{2m}]$
and $\mathcal R_1(p)=\sum_{i=1}^m p_i(z_1,\ldots,z_m)\otimes e_i$ . Therefore
$\mathcal R_1(\mathbb C[z_1,\ldots,z_{2m}])=\mathbb C[\boldsymbol z]\otimes \mathbb C^m$, completing the proof.
%
\end{proof}
\subsection{Description of the Hilbert module $\mathcal S_1$}
In this subsection, we give a description of the Hilbert module $\mathcal S_1$  
in the particular case when $K_1=K^{\alpha}$ and $K_2=K^\beta$ for some sesqui-analytic function $K$ defined on $\Omega\times\Omega$ and a pair of positive real numbers $\alpha, \beta$.

\begin{theorem}\label{module on 2nd box}Let $K:\Omega\times\Omega\to \mathbb C$ be a sesqui-analytic function such that the functions $K^{\alpha}$ and $K^{\beta}$, defined on $\Omega\times\Omega$, are non-negative definite for some $\alpha,\beta>0$. Suppose that the multiplication operators
$M_{z_i}, i=1,2,\ldots,m,$ are bounded on both $(\hl, K^\alpha)$ and $(\hl, K^{\beta})$.
Then the Hilbert module $\mathcal S_1$ is isomorphic to the push-forward module $\iota_\star \big(\hl, \mathbb{K}^{(\alpha,\beta)}\big)$  via the module map ${\mathcal R_1}_{|\mathcal S_1}$.
\end{theorem}
\begin{proof}
From Theorem \ref{realization}, it follows that the map $\mathcal R_1$ defined in \eqref{the map R_1} is a unitary map from
 $\mathcal S_1$ onto 
 $(\mathcal H, \mathbb K^{(\alpha,\beta)}).$ Now we will show that $\mathcal R_1P_{\mathcal S_1}(ph)=(p\circ \iota)\mathcal R_1h,\; h\in \mathcal S_1, p\in \mathbb C[z_1,\ldots,z_{2m}]$.
Let $h$ be an arbitrary element of $\mathcal S_1.$ Since $\ker \mathcal R_1= \mathcal S_1^\perp$ (see the discussion before Theorem \ref{realization}), it follows that $\mathcal R_1P_{\mathcal S_1}(p h) = \mathcal R_1 (ph),\;p\in \mathbb C[z_1,\ldots,z_{2m}]$.  Hence 
\begin{align*}
\mathcal R_1P_{\mathcal S_1}(p h) 
&=\mathcal R_1(ph)\\
&=\frac{1}{\sqrt{\alpha\beta(\alpha+\beta)}}\sum_{j=1}^m (\beta {\partial}_j(ph)-\alpha {\partial}_{m+j}(ph))_{|\Delta}\otimes e_j\\
&=\frac{1}{\sqrt{\alpha\beta(\alpha+\beta)}}
\sum_{j=1}^m p_{|\Delta}(\beta\partial_jh-\alpha\partial_{m+j}h)_{|\Delta}\otimes e_j+\sum_{j=1}^m h_{|\Delta}(\beta\partial_jp-\alpha\partial_{m+j}p)_{|\Delta}\otimes e_j\\
&=\frac{1}{\sqrt{\alpha\beta(\alpha+\beta)}}\sum_{j=1}^m p_{|\Delta}(\beta\partial_jh-\alpha\partial_{m+j}h)_{|\Delta}\otimes e_j
\;\;(\mbox{since}\; h\in \mathcal S_1)\\
&=(p\circ \iota)\mathcal R_1h,
\end{align*}
completing the proof.
\end{proof}
\begin{notation}
For $1\leq i\leq m$, let $M_i^{(1)}$ and $M_i^{(2)}$ denote the operators of multiplication by the coordinate function $z_i$ on the Hilbert spaces $(\hl, K_1)$ and $(\hl, K_2)$, respectively. If $m=1,$ we let $M^{(1)}$ and $M^{(2)}$ denote the operators $M_1^{(1)}$ and $M_1^{(2)}$, respectively. 

In case $K_1=K^\alpha$ and $K_2=K^\beta$, let $M_i^{(\alpha)},$
$M_i^{(\beta)}$ and $M_i^{(\alpha+\beta)}$ denote the operators of multiplication by the coordinate function $z_i$ on the Hilbert spaces $(\hl, K^\alpha)$, $(\hl, K^\beta)$ and $(\hl, K^{\alpha+\beta})$, respectively. 
If $m=1$, we write $M^{(\alpha)},$
$M^{(\beta)}$ and $M^{(\alpha+\beta)}$
instead of $M_1^{(\alpha)},$
$M_1^{(\beta)}$ and $M_1^{(\alpha+\beta)}$, respectively. 

Finally, let $\mathbb M_i^{(\alpha,\beta)}$ denote the operator of multiplication by the coordinate function $z_i$ on $(\hl, \mathbb K^{(\alpha,\beta)})$. Also let $\mathbb M^{(\alpha,\beta)}$ denote the operator $\mathbb M_1^{(\alpha,\beta)}$ whenever $m=1$. 
\end{notation}

\begin{remark}\label{remmobuleon2ndbox}
It is verified that $({M_{i}^{(\alpha)}}\otimes I)^*(\phi_j(w))=\bar{w}_i \phi_j(w)+\beta\delta_{ij} K^{\alpha}(\cdot,w)\otimes K^{\beta}(\cdot,w)$ and 
$(I \otimes {M_{i}^{(\beta)}})^*(\phi_j(w))=\bar{w}_i \phi_j(w)-\alpha\delta_{ij}K^{\alpha}(\cdot,w)\otimes K^{\beta}(\cdot,w),$ $1\leq i,j\leq m, w\in \Omega$.  Therefore, 
$$P_{\mathcal S_1}({M_{i}^{(\alpha)}}\otimes I)_{|\mathcal S_1}=P_{\mathcal S_1}(I\otimes {M_{i}^{(\beta)}})_{|\mathcal S_1},~i=1,2,\ldots,m.$$
\end{remark}

\begin{corollary}\label{corbddab}
The $m$-tuple of operators
$\big({P_{\mathcal S_1}({M_{1}^{(\alpha)}}\otimes 
I)}_{|\mathcal S_1},\ldots,{P_{\mathcal S_1}({M_{m}^{(\alpha)}}\otimes I)}_{|\mathcal S_1}\big)$
is unitarily equivalent to the $m$-tuple of operators $(\mathbb M_{1}^{(\alpha,\beta)},\ldots,\mathbb M_{m}^{(\alpha,\beta)})$ 
on $\big(\;\mathcal H, \mathbb{K}^{(\alpha,\beta)}\big)$.
In particular, if either the  $m$-tuple of operators $({M_{1}^{(\alpha)}},\ldots ,{M_{m}^{(\alpha)}})$ on $(\hl, K^\alpha)$
or the $m$-tuple of operators $({M_{(1)}^{(\beta)}},\ldots ,{M_{m}^{(\beta)}})$ on $(\hl, K^\beta)$ is bounded, then the $m$-tuple $(\mathbb M_{1}^{(\alpha,\beta)},\ldots ,\mathbb M_{m}^{(\alpha,\beta)})$ is also bounded on $\big(\mathcal H, \mathbb K^{(\alpha,\beta)}\big).$
\end{corollary}
\begin{proof}
The proof of the first statement follows from Theorem \ref{module on 2nd box} and the proof of the second statement follows from the first together with Remark \ref{remmobuleon2ndbox}.
\end{proof}

\subsection{Description of the quotient module ${\mathcal A}_1^\perp$}
In this subsection, we give a description of the quotient module 
$\mathcal A_1^\perp$. Let $(\hl, K^{\alpha+\beta})\widehat{\oplus} (\mathcal H, \mathbb{K}^{(\alpha,\beta)})$ be the Hilbert module, which is the Hilbert space $(\hl, K^{\alpha+\beta}){\oplus} (\mathcal H, \mathbb{K}^{(\alpha,\beta)})$ equipped with the multiplication over the polynomial ring $\mathbb C[z_1,\ldots,z_{2m}]$ 
induced by the $2m$-tuple of operators  
$(T_1, \ldots, T_m, T_{m+1}, \ldots , T_{2m})$ described below. 
First, for any polynomial $p\in \mathbb C[z_1,\ldots,z_{2m}]$, let $p^*(z) : = (p\circ \iota)(z) = p(z,z)$, $z\in \Omega$ and let  
$S_p:(\mathcal H, K^{\alpha + \beta}) \to (\mathcal H, \mathbb{K}^{(\alpha,\beta)})$ be the operator given by 
$$S_p(f_0)=\frac{1}{\sqrt{\alpha\beta(\alpha+\beta)}}
\sum_{j=1}^m(\beta(\partial_j p)^* -\alpha(\partial_{m+j}p)^*)f_0\otimes e_j
,f_0\in (\mathcal{H}, K^{\alpha+\beta}).$$
On the Hilbert space $(\mathcal H, K^{\alpha + \beta}) \oplus (\mathcal H, \mathbb{K}^{(\alpha,\beta)})$, let 
$T_i = \left (  \begin{smallmatrix}
M_{z_i}&0\\
S_{z_i} &M_{z_i}
\end{smallmatrix}\right )$,  and  $T_{m+i} = \left (  \begin{smallmatrix}
M_{z_i}&0\\
S_{z_{m+i}} &M_{z_i}
\end{smallmatrix}\right )$, $1\leq i \leq m$. Now, a straightforward verification shows that the module multiplication induced by  these $2m$-tuple of operators  
is given by the formula: 
 \begin{equation}\label{eqn:2box}
\mathbf m_p(f_0\oplus f_1)=\begin{pmatrix}
M_{p^*}f_0&0\\
S_pf_0&M_{p^*}f_1
\end{pmatrix},\; f_0\oplus f_1\in (\hl, K^{\alpha+\beta})\oplus (\hl, \mathbb K^{(\alpha,\beta)}).
\end{equation}
Clearly, this module multiplication is distinct from the one induced by the $M_p\oplus M_p$, $p\in \mathbb C[z_1, \ldots, z_m]$ on the direct sum $(\hl, K^{\alpha+\beta}){\oplus} (\mathcal H, \mathbb{K}^{(\alpha,\beta)})$.
\begin{theorem} \label{thm:2.2.8}
Let $K:\Omega\times\Omega\to \mathbb C$ be a sesqui-analytic function such that the functions $K^{\alpha}$ and $K^{\beta}$, defined on $\Omega\times\Omega$, are non-negative definite for some $\alpha,\beta>0$. Suppose that the multiplication operators
$M_{z_i}, i=1,2,\ldots,m,$ are bounded on both $(\hl, K^\alpha)$ and $(\hl, K^{\beta})$.
Then the quotient module $\mathcal A_1^\perp$
and the Hilbert module $(\hl, K^{\alpha+\beta})\widehat{\oplus} (\mathcal H, \mathbb{K}^{(\alpha,\beta)})$ are isomorphic.
\end{theorem}

\begin{proof} The proof is accomplished by showing that the compression operator $P_{\mathcal{A}_1^\perp}{M_p}_{|\mathcal{A}_1^\perp}$ is unitarily equivalent to the operator $
\left ( \begin{smallmatrix}
M_{p^*}&0\\
S_p& M_{p^*}
\end{smallmatrix}\right )
$ on $(\mathcal{H}, K^{\alpha+\beta})\bigoplus (\mathcal H, \mathbb{K}^{(\alpha,\beta)})$ for an arbitrary polynomial $p$ in $\mathbb C[z_1,\ldots,z_{2m}].$ 

We recall that the map $\mathcal R_0:(\mathcal{H}, K^\alpha)\otimes
(\mathcal H,K^\beta)\to(\mathcal{H}, K^{\alpha+\beta})$ given by 
$\mathcal R_0(f)=f_{|\Delta}$, $f$ in $(\mathcal{H}, K^\alpha)\otimes
(\mathcal H,K^\beta)$ 
defines a unitary map from $\mathcal S_0$ onto 
$(\mathcal H,K^{\alpha+\beta})$, and it intertwines the operators $P_{\mathcal S_0}{M_p}_{|\mathcal S_0}$ on $\mathcal S_0$ and $M_{p^*}$ on 
$(\mathcal{H}, K^{\alpha+\beta})$, that is, $M_{p^*}{\mathcal R_0}_{|\mathcal S_0}= {\mathcal R_0}_{|\mathcal S_0} P_{\mathcal S_0}{M_p}_{|\mathcal S_0}$. Combining this with Theorem \ref{realization}, we 
conclude that the map 
$\mathcal R=\left ( \begin{smallmatrix}
{\mathcal R_0}_{|\mathcal S_0}&0\\
0&{\mathcal R_1}_{|\mathcal S_1}
\end{smallmatrix} \right )$
is unitary from $\mathcal S_0\bigoplus \mathcal S_1$ (which is $\mathcal A_1^\perp$) to $(\mathcal{H}, K^{\alpha+\beta})\bigoplus(\mathcal{H},\mathbb{K}^{(\alpha,\beta)}).$
Since $\mathcal S_0$ is invariant under $M_{p}^*$, it follows that $P_{\mathcal S_1}{M_p^*}_{|\mathcal S_0}=0$. Hence 
$$\mathcal R P_{\mathcal{A}_1^\perp}{M_p^*}_{|\mathcal{A}_1^\perp}\mathcal R^*=
\begin{pmatrix}
\mathcal R_0P_{\mathcal{S}_0}{M_p^*}_{|\mathcal{S}_0}\mathcal R_0^*&
\mathcal R_0 P_{\mathcal{S}_0}{M_p^*}_{|\mathcal{S}_1}\mathcal R_1^*\\
0&\mathcal R_1P_{\mathcal{S}_1}{M_p^*}_{|\mathcal{S}_1}\mathcal R_1^*
\end{pmatrix}
$$
on $\mathcal S_0\bigoplus \mathcal S_1.$
We  have $\mathcal R_0P_{\mathcal{S}_0}{M_p^*}_{|\mathcal{S}_0}\mathcal R_0^*=(M_{p^*})^*$, already, on $(\mathcal{H},K^{\alpha+\beta})$. From Theorem \ref{module on 2nd box}, we see that $\mathcal R_1P_{\mathcal{S}_1}{M_p^*}_{|\mathcal{S}_1}\mathcal R_1^*
=(M_{p^*})^*$ on $(\mathcal H, \mathbb K^{(\alpha,\beta)})$.
To prove this, note that $\mathcal R_0 P_{\mathcal{S}_0}{M_p^*}_{|\mathcal{S}_1}\mathcal R_1^*=S_p^*$.
Recall that $\mathcal R_1^*(\mathbb K^{(\alpha,\beta)}(\cdot,w)e_j)= \phi_j(w)$. Consequently, an  easy computation gives

\begin{align*}
\mathcal R_0P_{\mathcal{S}_0}{M_p^*}_{|\mathcal{S}_1}\mathcal R_1^*
(\mathbb K^{(\alpha,\beta)}(\cdot,w)e_j)
&=\frac{1}{\sqrt{\alpha\beta(\alpha+\beta)}}(\overbar{\beta(\partial_j p)(w,w)
-\alpha(\partial_{m+j} p)(w,w)})K^{\alpha+\beta}(\cdot,w).
\end{align*}
Set $S_p^\sharp=\mathcal R_1P_{\mathcal{S}_1}{M_p}_{|\mathcal{S}_0}\mathcal R_0^*$. Then for $1\leq j\leq m$, and $w\in \Omega$, we get  
$$(S_p^\sharp)^*(\mathbb K^{(\alpha,\beta)}(\cdot,w)e_j)=\frac{1}{\sqrt{\alpha\beta(\alpha+\beta)}}
(\overbar{\beta(\partial_j p)(w,w)
-\alpha(\partial_{m+j} p)(w,w)})K^{\alpha+\beta}(\cdot,w).$$
For $f$ in $(\mathcal{H}, K^{\alpha+\beta})$, we have
\begin{align*}
\langle S_p^\sharp f(z),e_j\rangle &=\langle S_p^\sharp f,\mathbb K^{(\alpha,\beta)}(\cdot,z)e_j\rangle\\
&=\langle f,(S_p^\sharp)^*(\mathbb K^{(\alpha,\beta)}(\cdot,z)e_j)\rangle\\
&=\frac{1}{\sqrt{\alpha\beta(\alpha+\beta)}}({\beta(\partial_jp)(z,z)
-\alpha(\partial_{m+j} p)(z,z)})\langle f,
K^{\alpha+\beta}(\cdot,z)\;\rangle\\
&=\frac{1}{\sqrt{\alpha\beta(\alpha+\beta)}}\big({\beta(\partial_j p)(z,z)
-\alpha(\partial_{m+j} p)(z,z)}\big)f(z).
\end{align*}
Hence $S_p^\sharp=S_p$, completing the proof of the theorem.
\end{proof}

\begin{corollary}\label{quotient module}
Let $\Omega\subset \mathbb C$ be a bounded domain. The operator
$P_{\mathcal{A}_1^\perp}(M^{(\alpha)}\otimes I)_{|\mathcal{A}_1^\perp}$ is unitarily equivalent to the operator 
$\left ( \begin{smallmatrix}
M^{(\alpha+\beta)}&0\\
\delta\, \mathsf{inc} & \mathbb M^{(\alpha,\beta)})
\end{smallmatrix}\right )
$ on $(\mathcal{H}, K^{\alpha+\beta})\bigoplus(\mathcal H, \mathbb{K}^{(\alpha,\beta)})$,  where $\delta=\frac{\beta}{\sqrt{\alpha\beta(\alpha+\beta)}}$ and 
$\mathsf{inc}$ is the inclusion operator from $(\mathcal{H}, K^{\alpha+\beta})$ into $(\mathcal H, \mathbb{K}^{(\alpha,\beta)})$. 
\end{corollary}

\section{Generalized Bergman Kernels}\label{secgbk_}
We now discuss an important class of operators introduced by Cowen and Douglas in the very influential paper \cite{CD}. 
The case of $2$ variables was discussed in \cite{CDopen}, while a detailed study in the general case appeared later in \cite{Curtosalinas}. The definition below is taken from \cite{Curtosalinas}. 
Let 
$\boldsymbol T:=(T_1,...,T_m)$ be a $m$-tuple of commuting bounded linear operators on a separable Hilbert space 
$\mathcal H.$ Let $D_{\boldsymbol T}:\mathcal H\to \mathcal H\oplus\cdots\oplus \mathcal H$
be the operator defined by $D_{\boldsymbol T}(x)=(T_1x,...,T_mx), ~x\in \mathcal H.$ 
\begin{definition}[Cowen-Douglas class operator]
Let $\Omega\subset \mathbb C^m$ be a bounded domain. The operator $\boldsymbol T$ is said to be in the Cowen-Douglas class $B_n(\Omega)$ if $\boldsymbol T$ satisfies the following requirements:
\begin{enumerate}
\item[\rm(i)] 
\rm dim $\ker D_{\boldsymbol T-w} = n,~~w\in\Omega$
\item[\rm(ii)]
$\ran D_{\boldsymbol T-w}$ is closed for all $w\in\Omega$
\item[\rm(iii)]
$\overbar \bigvee \big\{\ker D_{\boldsymbol T-w}: w\in \Omega \big\}=\mathcal H.$ 
\end{enumerate} 
\end{definition}

If $\boldsymbol T\in B_n(\Omega)$, then  for each $w\in \Omega$, there exist functions $\gamma_1,\ldots,\gamma_n$ holomorphic in a neighbourhood $\Omega_0\subseteq \Omega$ containing $w$ such that $\ker D_{\boldsymbol T -w^\prime}=\bigvee\{\gamma_1(w^\prime),\ldots,\gamma_n(w^\prime)\}$ for all $w^\prime \in \Omega_0$ (cf. \cite{CDopen}). Consequently, every $\boldsymbol T\in B_n(\Omega)$ corresponds to a rank $n$
holomorphic hermitian vector bundle $E_{\boldsymbol T}$ defined by 
$$E_{\boldsymbol T}=\{(w,x)\in \Omega\times \mathcal H:x\in \ker D_{\boldsymbol T-w}\}$$ 
and $\pi(w,x)=w$, $(w,x)\in E_{\boldsymbol T}$.

For a bounded domain $\Omega$ in $\mathbb C^m$, let $\Omega^*=\{z:\bar{z}\in \Omega\}.$
It is known that if $T$ is an operator in $B_n(\Omega^*)$,
then  for each $w\in \Omega$, $T$ is unitarily equivalent to the adjoint of the multiplication tuple $(M_{z_1},\ldots,M_{z_m})$ on some 
reproducing kernel Hilbert space $(\mathcal H, K)\subseteq {\rm Hol}(\Omega_0, \mathbb C^n)$  for some  open subset  $\Omega_0\subseteq\Omega$ containing $w$. Here the kernel $K$ can be described explicitly as follows. 
Let $\Gamma=\{\gamma_1,\ldots,\gamma_n\}$ be a holomorphic frame of  the vector bundle $E_{\boldsymbol T}$ on a neighbourhood 
$\Omega_0^*\subseteq \Omega^*$ containing $\bar w$. Define 
$K_{\Gamma}: \Omega_0\times\Omega_0\to \mathcal M_n(\mathbb C)$ by
$K_{\Gamma}(z,w)=\big(\left\langle \gamma_j(\bar w),\gamma_i(\bar z)\right\rangle\big)_{i,j=1}^n$, $z,w\in \Omega_0$.  Setting $K=K_\Gamma$, one may verify that the operator $\boldsymbol T$ is unitarily equivalent to the adjoint of the $m$-tuple of  multiplication
operators $(M_{z_1},\ldots, M_{z_m})$ on the Hilbert space $(\hl, K)$.

If $T\in B_1(\Omega^*)$, the curvature matrix $\mathcal K_T(\bar{w})$ at a fixed but arbitrary point $\bar{w}\in \Omega^* $ is defined by
$$\mathcal K_T(\bar{w})=\big(\partial_i\bar{\partial}_j \log \|\gamma(\bar{w})\|^2\big)_{i,j=1}^m,$$
where $\gamma$ is a holomorphic frame of $E_{T}$ defined on some open subset $\Omega_0^*\subseteq \Omega^*$ containing $\bar{w}$. If $T$
is realized as the adjoint of the multiplication tuple $(M_{z_1},\ldots,M_{z_m})$ on some reproducing kernel Hilbert space $(\hl, K)\subseteq \rm{Hol}(\Omega_0)$, where $w\in \Omega_0$,  the curvature $\mathcal K_T(\bar{w})$  is then equal to $$\big(\partial_i\bar{\partial}_j \log K(w,w)\big)_{i,j=1}^m.$$ 

The study of operators in the Cowen-Douglass class using the properties of the kernel functions was initiated by Curto and Salinas in 
\cite{Curtosalinas}. The following definition is taken from \cite{Sal}.

\begin{definition}[Sharp kernel and generalized Bergman kernel]
A positive definite kernel $K:\Omega\times \Omega\to \mathcal{M}_k(\mathbb C)$ is said to be  sharp if 
\begin{itemize}
\item[\rm(i)]the multiplication operator $M_{z_i}$ is bounded on $(\mathcal H, K)$ for $i=1,\ldots,m, $
\item[\rm(ii)]$\ker D_{(\boldsymbol{M}_z-w)^*}=\ran K(\cdot,w)$,$~w\in\Omega,$
\end{itemize}
where $\boldsymbol{M}_z$ denotes the $m$-tuple $(M_{z_1},M_{z_2},\ldots,M_{z_m})$
on $(\mathcal H, K).$ Moreover, if $ ~\ran D_{( \boldsymbol{M}_z-w)^*}$ is closed for all $w\in \Omega$, then $K$ is said to be a generalized Bergman kernel. 
\end{definition}
We start with the following lemma (cf. \cite[page 285]{equivofquotient}) which provides a sufficient condition for the sharpness of a non-negative definite kernel $K$.
\begin{lemma}\label{lemsharpK}
Let $\Omega \subset \mathbb{C}^m$ be a bounded domain and $K:\Omega\times \Omega\to \mathcal M_k(\mathbb C)$
be a non-negative definite kernel. Assume that the multiplication operator $M_{z_i}$ on $(\hl, K)$
is bounded  for $1\leq i\leq m$. If the vector valued polynomial ring $\mathbb C[z_1,\ldots,z_m]\otimes \mathbb C^k $ is contained in $(\hl, K)$ as a dense subset, then $K$ is a sharp kernel.
\end{lemma}

\begin{corollary}
Let $\Omega \subset \mathbb{C}^m$ be a bounded domain and $K:\Omega\times\Omega\to \mathbb C$ be a sesqui-analytic function such that the functions $K^{\alpha}$ and $K^{\beta}$ are non-negative definite on $\Omega\times\Omega$ for some $\alpha,\beta>0$.
Suppose that either the  $m$-tuple of operators $({M_{1}^{(\alpha)}},\ldots ,{M_{m}^{(\alpha)}})$ on $(\hl, K^\alpha)$
or the $m$-tuple of operators $({M_{1}^{(\beta)}},\ldots ,{M_{m}^{(\beta)}})$ on $(\hl, K^\beta)$ is bounded. If both the Hilbert spaces $(\mathcal H, K^\alpha)$ and $(\mathcal H, K^{\beta})$ contain the polynomial ring $\mathbb C[z_1,\ldots,z_m]$ as a dense subset, then the kernel $\mathbb K^{(\alpha,\beta)}$ is sharp. 
\end{corollary}

\begin{proof}
By Corollary \ref{corbddab}, we have that the $m$-tuple of operators $(\mathbb M_{1}^{(\alpha,\beta)},\ldots ,\mathbb M_{m}^{(\alpha,\beta)})$ is bounded on $\big(\mathcal H, \mathbb K^{(\alpha,\beta)}\big)$. If both the Hilbert spaces $(\mathcal H, K^\alpha)$ and $(\mathcal H, K^{\beta})$ contain the polynomial ring $\mathbb C[z_1,\ldots,z_m]$ as a dense subset, then by Proposition \ref{proppoldense}, we see that the ring 
$\mathbb C[z_1,\ldots,z_m]\otimes \mathbb C^m$ is contained in $(\hl, \mathbb K^{(\alpha, \beta)})$ and is dense in it. An application of Lemma \ref{lemsharpK} now completes the proof. 
\end{proof}

Some of the results in this paper generalize, among other things,  one of the main results of \cite{Sal}, which  is reproduced below.
\begin{theorem}[Salinas, {\cite[Theorem 2.6]{Sal}}]{\label{thmsalinas}}
Let $\Omega \subset \mathbb{C}^m$ be a bounded domain. If $K_1,K_2:\Omega\times\Omega\to \mathbb C$ are two  sharp kernels
(resp. generalized Bergman kernels), then $K_1\otimes K_2$ and $K_1K_2$ are also sharp kernels (resp. generalized Bergman kernels).
\end{theorem}

For two scalar valued non-negative definite kernels $K_1$ and $K_2$, defined on $\Omega\times \Omega$, the jet construction (Theorem \ref{thmjetcons}) gives rise to a family of non-negative kernels 
$J_k(K_1,K_2)_{|\rm res \, \Delta}$, $k\geq 0$, where
$$J_k(K_1,K_2)_{|\rm res \, \Delta}(z,w):=\big(K_1(z,w)\partial^{\boldsymbol i}\bar{\partial}^{\boldsymbol j} K_2(z,w)\big)_{|\boldsymbol i|,|\boldsymbol j|=0}^k,\;\;z,w\in \Omega.$$
In the particular case when $k=0$, it coincides with the point-wise product $K_1K_2$. In this section, we generalize Theorem \ref{thmsalinas} for all kernels  of the form $J_k(K_1,K_2)_{|\rm res \, \Delta}$. First, we discuss two important corollaries of the jet construction which will be used later in this paper.

\emph{
For $1\leq i\leq m$, let $J_kM_i$ denote the operator of multiplication by the 
$i$th coordinate function $z_i$ on the Hilbert space $\big(\hl,J_k(K_1,K_2)_{|\rm res \, \Delta}\big)$. In case $m=1$, we write $J_kM$ instead of $J_kM_1$.
}

Taking $p(z,\zeta)$ to be the $i$th coordinate function $z_i$ in Proposition \ref{prop:jetq}, we obtain the following corollary. 

\begin{corollary}\label{coroperatoronJetk}
Let $K_1,K_2:\Omega\times\Omega\to \mathbb C$ be two non-negative definite kernels. 
Then the  $m$-tuple of operators $\big(P_{\mathcal A_k^{\perp}}
{(M_1^{(1)}\otimes I)}_
{|{\mathcal A_k^\perp}},\ldots,P_{\mathcal A_k^{\perp}}
{(M^{(1)}_m\otimes I)}_{|{\mathcal A_k^\perp}}\big)$ 
is unitarily equivalent to the $m$-tuple $(J_kM_{1},\ldots,J_kM_m)$ on the Hilbert space $\big(\hl,J_k(K_1,K_2)_{|\rm res \, \Delta}\big)$. 
\end{corollary}

%

Combining this with Corollary $\ref{quotient module}$ we obtain
the following result. 
\begin{corollary}\label{quotient module1}
Let $\Omega \subset \mathbb{C}$ be a bounded domain and $K:\Omega\times\Omega\to \mathbb C$ be a sesqui-analytic function such that the functions $K^{\alpha}$ and $K^{\beta}$ are non-negative definite on $\Omega\times\Omega$ for some $\alpha,\beta>0$. The following operators are unitarily equivalent:
\begin{itemize}
\item[\rm(i)] the operator 
$P_{\mathcal{A}_1^\perp}(M^{(\alpha)}\otimes I)_{|\mathcal{A}_1^\perp}$
\item[\rm(ii)]the multiplication operator $J_1M$ on $\big(\mathcal H, J_1(K^\alpha, K^\beta)_{|\rm res \, \Delta}\big)$
\item[\rm (iii)]the operator $\begin{pmatrix}
M^{(\alpha+\beta)}&0\\
\delta \;\mathsf{inc} & \mathbb M^{(\alpha,\beta)}
\end{pmatrix}
$ on $(\mathcal{H}, K^{\alpha+\beta})\bigoplus(\mathcal H, \mathbb K^{(\alpha,\beta)})$  where $\delta=\frac{\beta}{\sqrt{\alpha\beta(\alpha+\beta)}}$ and $\mathsf{inc}$ is the inclusion operator from $(\mathcal{H}, K^{\alpha+\beta})$ into $(\mathcal H, \mathbb K^{(\alpha,\beta)}).$
\end{itemize}
\end{corollary}


%


We need the following lemmas for the generalization of Theorem \ref{thmsalinas}.


\begin{lemma}\label{ker of tensor}
Let $\hl_1$ and $\hl_2$ be two Hilbert spaces and $T$ be a bounded linear operator on $\hl_1$. 
Then $$\ker (T\otimes I_{\hl_2})=\ker T\otimes \mathcal H_2.$$
\end{lemma}

\begin{proof} 
It is easily seen that  $\ker T\otimes \mathcal 
{H}_2\subset \ker (T\otimes I_{\mathcal H_2}).$
To establish the opposite inclusion, let $x$ be an arbitrary element in $\ker (T\otimes I_{\mathcal H_2}).$ 
Fix an orthonormal basis $\{f_i\}$  of $\mathcal {H}_2.$ Note that $x$ is of the form $\sum v_i\otimes f_i$ for some $v_i$'s in $\hl_1.$ 
Since $x\in \ker (T\otimes I_{\mathcal H_2})$, we have $\sum Tv_i\otimes f_i=0.$ Moreover, since $\{f_i\}$
is an orthonormal basis of $\hl_2$, it follows that $Tv_i=0$ for all $i$. 
Hence $x$ belongs to 
$\ker (T)\otimes \mathcal {H}_2$, completing the proof of the lemma.   
\end{proof}  

\begin{lemma}\label{intersection of tns}
Let $\hl_1$ and $\hl_2$ be two Hilbert spaces. If  $B_1,\ldots,B_m$ are closed subspaces of $\mathcal H_1$, then $$\bigcap_{l=1}^m (B_l\otimes \mathcal H_2)=\Big(\bigcap_{l=1}^m B_l\Big)\otimes\mathcal H_2.$$
\end{lemma}

\begin{proof}We only prove the non-trivial inclusion, namely, $\cap_{l=1}^m \left(B_l\otimes \mathcal H_2\right)
\subset \left(\cap_{l=1}^m B_l\right)\otimes\mathcal H_2.$ 

Let $\{f_j\}_j$ be an orthonormal basis of $\mathcal H_2$ and  $x$ be an arbitrary element in $\mathcal H_1\otimes\mathcal H_2$. Recall that $x$ can be written uniquely as $\sum x_j\otimes f_j$, $x_j\in \mathcal H_1$.

Claim: If $x$ belongs to $ B_l\otimes \mathcal H_2,$ then $x_j$ belongs to $B_l$ for all $j.$ 

To prove the claim, 
assume that $\{e_i\}_i$ is an orthonormal basis of $B_l.$ Since 
$\{e_i\otimes f_j\}_{i,j}$ is an orthonormal basis of $B_l\otimes
\mathcal H_2$ and $x$ can be written as $\sum x_{ij}e_i\otimes f_j=
\sum_{j}(\sum_i x_{ij}e_i)\otimes f_j.$ Then, the uniqueness of 
the representation $x=\sum x_j\otimes f_j,$ ensures that $x_j=\sum_i x_{ij}e_i.$
In particular, $x_j$ belongs to $B_l$ for all $j.$ Thus the claim is 
verified. 

Now let $y$ be any element in $\cap_{l=1}^m 
\left(B_l\otimes \mathcal H_2\right).$ 
Let $\sum y_j\otimes f_j$ be the unique representation of $y$ 
in $\mathcal H_1\otimes \mathcal H_2.$ Then from the claim, 
it follows that $y_j\in \cap_{l=1}^m B_l.$  Consequently, $y\in (\cap_{l=1}^m B_l)\otimes \hl_2$. 
This completes the proof.
\end{proof}
The proof of the following lemma is straightforward and therefore it is omitted.  
\begin{lemma}\label{lemkerofunitary}
Let $\mathcal H_1$ and $\mathcal H_2$ be two Hilbert spaces. Let $A:\mathcal H_1\to \mathcal H_1 $ be a bounded
linear operator and $B:\mathcal H_1\to \mathcal H_2$ be a unitary operator. Then 
$$\ker BAB^*= B(\ker A).$$
\end{lemma}
The lemma given below is a generalization of \cite[Lemma 1.22 (i)]{CD} to commuting tuples. Recall that for a commuting $m$-tuple $\boldsymbol T=(T_1,\ldots,T_m)$, the operator $\boldsymbol T^{\boldsymbol i}$ is defined by 
$T_1^{i_1}\cdots T_m^{i_m}$, where 
$\boldsymbol i=(i_1,\ldots,i_m)\in \z^m$.

\begin{lemma}\label{lemactiononder}
If $K:\Omega\times\Omega\to \mathbb C$ is a positive definite kernel such that 
the $m$-tuple of multiplication operators ${\boldsymbol M}_z=(M_{z_1},\ldots,M_{z_m})$ on $(\mathcal H, K)$ is bounded, 
then for $w\in \Omega$ and $\boldsymbol i=(i_1,\ldots,i_m),
\boldsymbol j=(j_1,\ldots,j_m)$ in $\mathbb Z_+^m$,
\begin{itemize}
\item[\rm (i)] $(\boldsymbol{M}_z^*-\bar{w})^{\boldsymbol i}\bar{\partial}^{\boldsymbol j} K(\cdot,w)=0 $ 
if $|\boldsymbol i|> |\boldsymbol j|$,
\item[\rm (ii)] $(\mathbf{M}_z^*-\bar{w})^{\boldsymbol i}\bar{\partial}^{\boldsymbol j} K(\cdot,w)
={\boldsymbol  j}! \delta_{\boldsymbol i\boldsymbol j}  K(\cdot,w)$ if $|\boldsymbol i|=|\boldsymbol j|.$
\end{itemize}

\end{lemma}

\begin{proof}
First, we claim that if $ {i}_l> {j}_l$ for some $1\leq l\leq m$, then
$(M_{z_l}^*-\bar{w}_l)^{{i}_l}\bar {\partial}_l^{{j}_l}K(\cdot,w)=0.$
The claim is verified by induction on ${j}_l$. The case ${j}_l=0$ holds trivially since $(M_{z_l}^*-\bar{w}_l)K(\cdot,w)=0$. Now assume that the claim is valid for ${j}_l=p.$
We have to show that it is true for ${j}_l=p+1$ also. 
Suppose ${i}_l > p+1 .$ Then $ {i}_l-1 > p .$ Hence, by the induction
hypothesis, $(M_{z_l}^*-\bar{w}_l)^{i_l-1}\bar {\partial}_l^p K(\cdot,w)=0.$ 
Differentiating this with respect to 
$\bar {w}_l$, we see that 
$$({i}_l-1)(M_{z_l}^*-\bar{w}_l)^{{i}_l-2}(-1)\bar {\partial}_l^p K(\cdot,w)
+ (M_{z_l}^*-\bar{w}_l)^{{i}_l-1}\bar {\partial}_l^{p+1} K(\cdot,w)=0 .$$
Applying $(M_{z_l}^*-\bar{w}_l)$ to both sides of the equation above, we obtain 
$$( i_l-1)(M_{z_l}^*-\bar{w}_l)^{i_l-1}(-1)\bar {\partial}_l^p K(\cdot,w)
+ (M_{z_l}^*-\bar{w}_l)^{i_l}\bar {\partial}_l^{p+1} K(\cdot,w)=0 .$$
Using the induction hypothesis once again, we conclude that 
$(M_{z_l}^*-\bar{w}_l)^{i_l}\bar {\partial}_l^{p+1} K(\cdot,w)=0 .$
Hence the  claim is verified. 

Now, to prove the 
first part of the lemma, assume that $|{\boldsymbol  i}| > |{\boldsymbol  j}|.$
Then there exists a $l$ such that $i_l > j_l.$
Hence from the claim, we have $(M_{z_l}^*-\bar{w}_l)^{i_l}\bar 
{\partial}_l^{j_l} K(\cdot,w)=0.$
Differentiating with respect to all other variables except $\bar{w}_l$,
we get $(M_{z_l}^*-\bar{w}_l)^{i_l}\bar 
{\partial}^{\boldsymbol  j} K(\cdot,w)=0. $ Applying the operator $(\boldsymbol{M}_z^*-\bar{w})^{{\boldsymbol  i}- i_l e_l}$, where $e_l$ is the $l$th standard unit vector of $\mathbb C^m$, we see that $(\boldsymbol{M}_z^*-\bar{w})^{{\boldsymbol  i}}\bar{\partial}^{\boldsymbol  j} K(\cdot,w)=0$, completing the proof of the first part. 

For the second part,
assume that $|{\boldsymbol  i}|=|{\boldsymbol  j}|$ and ${\boldsymbol  i}\neq {\boldsymbol  j}.$ 
Then there is atleast one $l$ such that $ i_l>  j_l.$ 
Hence by the argument used in the last paragraph, we conclude that 
$(\boldsymbol{M}_z^*-\bar{w})^{{\boldsymbol  i}}\bar{\partial}^{\boldsymbol  j} K(\cdot,w)=0.$ 
Finally, if ${\boldsymbol  i}={\boldsymbol  j},$ we use induction on ${\boldsymbol  i}$ to proof the lemma. There is nothing to prove if
${\boldsymbol  i}=0$. For the proof by induction, now, assume that
$(\boldsymbol{M}_z^*-\bar{w})^{{\boldsymbol  i}}\bar{\partial}^{\boldsymbol  i} K(\cdot,w)={{\boldsymbol  i}}!K(\cdot,w)$
for some ${\boldsymbol  i}\in \mathbb Z_+^m .$ To complete the induction step,  we have to prove that
$(\boldsymbol{M}_z^*-\bar{w})^{{\boldsymbol  i}+e_l}\bar{\partial}^{{\boldsymbol  i}+e_l}K(\cdot,w)
=({\boldsymbol  i}+e_l)!K(\cdot,w).$ By the first part of the lemma, we have 
$(\boldsymbol{M}_z^*-\bar{w})^{{\boldsymbol  i}+e_l}\bar{\partial}^{\boldsymbol  i} K(\cdot,w)=0.$ 
Differentiating with respect to $\bar{w}_l,$ we get that
$$(\boldsymbol{M}_z^*-\bar{w})^{{\boldsymbol  i}+e_l}\bar{\partial}^{{\boldsymbol  i}+e_l}
K(\cdot,w)-(i_{l}+1)(\boldsymbol{M}_z^*-\bar{w})^{\boldsymbol  i} \bar{\partial}^{\boldsymbol  i} K(\cdot,w)=0.$$   
Hence, by the induction hypothesis, 
$(\boldsymbol{M}_z^*-\bar{w})^{{\boldsymbol  i}+e_l}\bar{\partial}^{{\boldsymbol  i}+e_l}K(\cdot,w)
=({\boldsymbol  i}+e_l)!K(\cdot,w).$ This completes the proof.
\end{proof}

\begin{corollary}\label{cortaylor}
Let $K:\Omega\times\Omega\to \mathbb C$ be a positive definite kernel.  Suppose that the $m$-tuple of multiplication operators ${\boldsymbol M}_z$ on $(\mathcal H, K)$ is bounded. Then, for all $w\in \Omega$,
the set $\big \{\;\bar{\partial}^{\boldsymbol i} K(\cdot,w):\boldsymbol{i}\in \mathbb Z^m_+ \;\big \}$ is linearly independent. Consequently, the matrix $\big(\partial^{\boldsymbol i}\bar{\partial}^{\boldsymbol j}K(w,w)\big)_{\boldsymbol i,\boldsymbol j\in \Lambda}$ is positive definite
for any finite subset $\Lambda$ of $\mathbb Z^m_+ $. 
\end{corollary}
\begin{proof}
Let $w$ be an arbitrary point in $\Omega$. It is enough to show that the set $\big \{\;\bar{\partial}^{\boldsymbol i} K(\cdot,w):\boldsymbol{i}\in \mathbb Z^m_+, |\boldsymbol{i}|\leq k  \;\big\}$ is linearly independent for each non-negative integer $k$. Since $K$ is positive definite, there is nothing to prove if $k=0$. To complete the proof by induction on $k$, assume that the set $\big \{\;\bar{\partial}^{\boldsymbol i} K(\cdot,w):\boldsymbol{i}\in \mathbb Z^m_+, |\boldsymbol{i}|\leq k \;\big \}$
is linearly independent for some non-negative integer $k$. Suppose that 
$\sum_{|\boldsymbol{i}|\leq k+1}a_{\boldsymbol{i}}\bar{\partial}^{\boldsymbol i} K(\cdot,w)=0$ 
for some $a_{\boldsymbol{i}}$'s in $\mathbb C$.
Then 
$(\boldsymbol M^*_z-\bar w)^{\boldsymbol q}(\sum_{|\boldsymbol{i}|\leq k+1}a_{\boldsymbol{i}}\bar{\partial}^{\boldsymbol i} K(\cdot,w))=0$, for
all $\boldsymbol q\in \mathbb Z^m_+$ with $|\boldsymbol q|\leq k+1$. If $|\boldsymbol q|=k+1$, by Lemma \ref{lemactiononder}, we have that $a_{\boldsymbol q}\;q! K(\cdot,w)=0$. Consequently, $a_{\boldsymbol q}=0$
for all $\boldsymbol q\in \mathbb Z^m_+$ with 
$|\boldsymbol q|=k+1$. Hence, by the induction hypothesis, we conclude that $a_{\boldsymbol i}=0$
for all $\boldsymbol i\in \mathbb Z^m_+$, $|\boldsymbol i|\leq k+1$ and the set $\big \{\;\bar{\partial}^{\boldsymbol i} K(\cdot,w):\boldsymbol{i}\in \mathbb Z^m_+, |\boldsymbol{i}|\leq k+1 \;\big\}$ is linearly independent, completing the proof of the first part of the corollary. 

If $\Lambda$ is a finite subset of 
$\mathbb Z^m_+$, then it follows form 
the linear independence of the vectors $\big\{\bar{\partial}^{\boldsymbol i} K(\cdot,w):\boldsymbol i\in \Lambda\big\}$ that the matrix 
$\big(\left\langle\bar{\partial}^{\boldsymbol j}K(\cdot,w), \bar{\partial}^{\boldsymbol i}K(\cdot,w)\right\rangle \big)_{\boldsymbol i,\boldsymbol j\in \Lambda}$ is positive definite.
Now the proof is complete since $\left\langle\bar{\partial}^{\boldsymbol j}K(\cdot,w), \bar{\partial}^{\boldsymbol i}K(\cdot,w)\right\rangle=\partial^{\boldsymbol i}\bar{\partial}^{\boldsymbol j}K(w,w)$ (see Proposition \ref{derivativeofK}). 
\end{proof}

The following proposition is also a generalization to the multi-variate setting of \cite[Lemma 1.22 (ii)]{CD}( see also \cite{CDopen}).

\begin{proposition}\label{prop joint kernel}
If $K:\Omega\times \Omega\to \mathbb C$ is a sharp kernel, then for every $ w\in \Omega $
$$ \bigcap_{|\boldsymbol j|=k+1}\ker\;(\boldsymbol{M}_z^*-\bar {w})^{\boldsymbol j}
=\bigvee \big \{\bar{\partial}^{{\boldsymbol  j}} K(\cdot,w):|{\boldsymbol  j}|\leq k \big \}.$$
\end{proposition}

\begin{proof}
The inclusion $\bigvee\{\bar{\partial}^{{\boldsymbol  j}} K(\cdot,w):|{\boldsymbol  j}|\leq k \} \subseteq \bigcap_{|\boldsymbol j|=k+1}\ker\;(\boldsymbol{M}_z^*-\bar {w})^{\boldsymbol j}$ follows from part (i) of Lemma \ref{lemactiononder}.
We use induction on $k$ for the opposite inclusion.
From the definition of  sharp kernel, this inclusion is evident if $k=0$. 
Assume that  
$$\bigcap_{|\boldsymbol j|=k+1}\ker\;(\boldsymbol{M}_z^*-\bar {w})^{\boldsymbol j} \subseteq \bigvee\big \{\bar{\partial}^{{\boldsymbol  j}} K(\cdot,w):|{\boldsymbol  j}|\leq k \big \}$$  for some non-negative integer $k$. To complete the proof by induction, we show that the inclusion remains valid  for $k+1$ as well. Let $f$ be an arbitrary element of $\bigcap_{|\boldsymbol  i|=k+2}\ker (\boldsymbol{M}_z^*-\bar {w})^{\boldsymbol  i}.$ 
Fix a $\boldsymbol  j\in \mathbb Z_+^m$ with 
$|{\boldsymbol  j}|=k+1$. Then it follows that $(\boldsymbol{M}_z^*-\bar {w})^{\boldsymbol  j} f$ belongs to $\cap_{l=1}^m\ker (M_{z_l}^*-\bar {w}_l).$ 
Since $K$ is sharp, we see that 
$(\boldsymbol{M}_z^*-\bar {w})^{{\boldsymbol  j}} f= c_{{\boldsymbol  j}} K(\cdot,w)$ for some constant $c_{{\boldsymbol  j}}$ depending on $w$.
Therefore
\begin{align*}
(\boldsymbol{M}_z^*-\bar {w})^{\boldsymbol j }\Big(f-\sum_{|{\boldsymbol  q}|=k+1}
\frac{c_{{\boldsymbol  q}}}{{\boldsymbol  q}!}\bar {\partial}^{{\boldsymbol  q}}K(\cdot,w)\Big) 
= & c_{\boldsymbol j } K(\cdot,w)-\sum_{|{\boldsymbol  q}|=k+1} \frac{c_{{\boldsymbol  q}}}{{\boldsymbol  q}!}
(\boldsymbol{M}_z^*-\bar {w})^{\boldsymbol j }{\bar \partial}^{{\boldsymbol q}}K(\cdot,w)\\
= &  c_{\boldsymbol j} K(\cdot,w)-\sum_{|{\boldsymbol  q}|=k+1}c_{{\boldsymbol  q}}\delta_{\boldsymbol j  {\boldsymbol q}}\tfrac{\boldsymbol j !}{{\boldsymbol q} !} K(\cdot,w)\\
= & 0,
\end{align*}
where the last equality follows from Lemma \ref{lemactiononder}.
Hence the element $f-\sum_{|{\boldsymbol  q}|=k+1}
\frac{c_{\boldsymbol q}}{{\boldsymbol  q}!}{\bar \partial}^{{\boldsymbol  q}}K(\cdot,w)$ belongs to 
$\bigcap_{|\boldsymbol j |=k+1}\ker (\boldsymbol{M}_z^*-\bar {w})^{\boldsymbol  j}.$
Thus by the induction hypothesis, 
$f-\sum_{|{\boldsymbol q}|=k+1}\frac{c_{{\boldsymbol q}}}{{\boldsymbol q}!}
{\bar \partial}^{{\boldsymbol  q}}K(\cdot,w)=\sum_{|{\boldsymbol  j}|\leq k}
d_{{\boldsymbol  j}} \bar{\partial}^{\boldsymbol  j} K(\cdot,w).$ 
Hence $f$ belongs to $\bigvee\{\bar{\partial}^{{\boldsymbol  j}} K(\cdot,w):|{\boldsymbol  j}|\leq k+1 \}.$
This completes the proof. 
\end{proof}

For a $m$-tuple of bounded operators $\mathbf T=(T_1,\ldots,T_m)$ on a Hilbert space
$\mathcal H,$ we define an operator 
$D^ {\mathbf T}:\mathcal H\bigoplus\cdots\bigoplus\mathcal H\to \mathcal H $
by $$ D^{\boldsymbol T}(x_1,\ldots,x_m)=\sum_{i=1}^m T_ix_i,\;\;x_1,\ldots,x_m \in\mathcal H.$$  
A routine verification shows that $(D_{\boldsymbol T})^*=D^{{\boldsymbol T}^*}$. 
The following lemma is undoubtedly well known, however, we provide a proof for the sake of completeness.
\begin{lemma}\label{lemGleason1}
Let $K:\Omega\times\Omega\to \mathbb C $ be a positive definite kernel
such that the $m$-tuple of multiplication operators ${\boldsymbol M}_z$ on $(\mathcal H, K)$ is bounded. Let $w=(w_1,\ldots,w_m)$ be a fixed but arbitrary point in $\Omega$ and  let $\mathcal V_w$ be the subspace given by 
$\{f\in (\hl, K): f(w)=0\}$. Then $K$ is a 
generalized Bergman kernel if and only 
if for every $w\in \Omega$,
\begin{equation}\label{eqnGleason1}
\mathcal V_w=
\Big\{\textstyle\sum_{i=1}^m (z_i-w_i)g_i:g_i\in (\mathcal H, K)\Big\}.
\end{equation}
\end{lemma}

\begin{proof}
First, observe that the right-hand 
side of $\eqref{eqnGleason1}$ 
is equal to $\ran D^{ {\boldsymbol M}_z-w}.$ Hence it suffices to show that  $K$ is a generalized Bergman kernel if and only if 
$\mathcal V_w = \ran D^{{\boldsymbol M}_z-w}.$ 
In any case, we have the following inclusions 
\begin{align}\label{eqnGleason2}
\ran D^{{\boldsymbol M}_z-w}=\ran (D_{({\boldsymbol M}_z-w)^*})^*\subseteq \overbar{\ran (D_{({\boldsymbol M}_z-w)^*})^*}
&={\ker D_{({\boldsymbol M}_z-w)^*}}^\perp\\
&\subseteq \{cK(\cdot,w):c\in \mathbb C\}^\perp\nonumber\\
&=\mathcal V_w\nonumber.
\end{align}
Hence it  follows that $\mathcal V_w = \ran D^{\boldsymbol {M}_z-w}$
if and only if equality is forced everywhere in these inclusions, that is, $\ran (D_{(\boldsymbol {M}_z-w)^*})^* = \overbar{\ran (D_{(\boldsymbol {M}_z-w)^*})^*}$
and  ${\ker D_{(\boldsymbol {M}_z-w)^*}}^\perp = \{cK(\cdot,w):c\in \mathbb C\}^\perp .$
Now $\ran (D_{(\boldsymbol {M}_z-w)*})^* = \overbar{\ran (D_{(\boldsymbol {M}_z-w)^*})^*}$ 
if and only if $\ran (D_{(\boldsymbol {M_z}-w)^*})^*$ is closed.  
Recall that, if $\mathcal H_1,\mathcal H_2$ are two Hilbert spaces, and an operator 
$T:\mathcal H_1\to \mathcal H_2$ has closed range, then $T^*$
also has closed range. Therefore, $\ran (D_{(\boldsymbol {M}_z-w)^*})^*$ is closed if and only if $\ran D_{(\boldsymbol {M}_z-w)^*}$ is 
closed.   Finally, note that
${\ker D_{(\boldsymbol {M}_z-w)^*}}^\perp = \{cK(\cdot,w):c\in \mathbb C\}^\perp $ holds
if and only if ${\ker D_{(\boldsymbol {M}_z-w)^*}} = \{cK(\cdot,w):c\in \mathbb C\}.$
This completes the proof. 
\end{proof}
\begin{notation}
Recall that for $1\leq i\leq m$, $M_i^{(1)}, M_i^{(2)}, J_kM_i$ denote the operators of multiplication by the coordinate function $z_i$
on the Hilbert spaces $(\hl, K_1), (\hl, K_2)$
and $(\hl, J_k(K_1,K_2)_{|\rm res \, \Delta})$, respectively. Set
$\boldsymbol M^{(1)}=(M_1^{(1)},\ldots,M_m^{(1)})$,
$\boldsymbol M^{(2)}=(M_1^{(2)},\ldots,M_m^{(2)})$ and
${\boldsymbol J}_k\boldsymbol {M}=(J_k M_1,\ldots,J_k M_m)$. Also, for the sake of brevity, let
$\hl_1$ and $\hl_2$ be the Hilbert spaces $(\hl, K_1)$ and $(\hl, K_2)$, respectively for the rest of this section.
\end{notation}

The following lemma is the main tool  to prove that the kernel $J_k(K_1, K_2)_{|\rm res \, \Delta}$ is sharp whenever $K_1$ and $K_2$ are sharp.


\begin{lemma}\label{lemkerneloftns}
If $K_1, K_2:\Omega\times\Omega\to \mathbb C $ are two sharp kernels, then for all $w=(w_1,\ldots,w_m)\in \Omega$,
\begin{align*}
\bigcap_{p=1}^m\ker\Big(\big((M_p^{(1)}-{w}_p)^*\otimes I\big)
_{|\mathcal A_k^\perp}\Big)&=
\bigcap_{|\boldsymbol i|=1}\ker {\big({\boldsymbol M^{(1)}}- w\big)^*}^{\boldsymbol i}\otimes\bigcap_{|\boldsymbol i|=k+1}
\ker {\big({\boldsymbol M^{(2)}}- w\big)^*}^{\boldsymbol i}
\\&=\bigvee\big\{K_1(\cdot,w)\otimes\bar{\partial}^{\boldsymbol i}K_2(\cdot,w):
 |\boldsymbol i|\leq k\big\}.
\end{align*}
\end{lemma}

\begin{proof}
Since $K_1$ and $K_2$ are sharp kernels,
by Proposition $ \ref{prop joint kernel}, $
it follows that 
\begin{equation}\label{eqn ker of tns 2}
\bigcap_{|\boldsymbol i|=1}\ker {({\boldsymbol M^{(1)}}-w)^*}^{\boldsymbol i}
\otimes\bigcap_{|\boldsymbol i|=k+1}\ker {({\boldsymbol M^{(2)}}-w)^*}^{\boldsymbol i}= 
\bigvee\{K_1(\cdot,w)\otimes \bar{\partial}^{\boldsymbol  j} K_2(\cdot,w):|\boldsymbol j|\leq k\}.
\end{equation}

Therefore, if we can show that 
\begin{equation}\label{eqn ker of tns 1}
\bigcap_{p=1}^m\ker\Big(\big((M_p^{(1)}-{w}_p)^*\otimes I\big)
_{|{\mathcal A_k}^\perp}\Big)=
\bigcap_{|\boldsymbol i|=1}\ker {({\boldsymbol M^{(1)}}-w)^*}^{\boldsymbol i}\otimes\bigcap_{|\boldsymbol i|=k+1}
\ker {(\boldsymbol M^{(2)}- w)^*}^{\boldsymbol i},
\end{equation}
then we will be done. To prove this, first note that
\begin{align*}
\bigcap_{p=1}^m\ker\Big(\big((M_p^{(1)}-{w}_p)^*\otimes I\big)
_{|\mathcal A_k^\perp}\Big)
&=\Big(\bigcap_{p=1}^m\ker \big((M_p^{(1)}-{w}_p)^*\otimes I\big)\Big) 
\bigcap\mathcal A_k^\perp\\
&=\Big(\bigcap_{p=1}^m\big(\ker(M_p^{(1)}-{w}_p)^*\otimes 
\mathcal H_2\big)\Big)\bigcap\mathcal A_k^\perp
\\
&= \Big(\Big(\bigcap_{p=1}^m \ker(M_p^{(1)}-{w}_p)^*\Big)
\otimes {\mathcal H}_2\Big)\bigcap\mathcal A_k^\perp\\
&=\Big(\ker D_{(\boldsymbol M^{(1)}-w)^*}\otimes \mathcal H_2\Big)\bigcap\mathcal A_k^\perp.
\end{align*}
Here the second equality follows from Lemma \ref{ker of tensor} and the third equality follows from Lemma \ref{intersection of tns}.
In view of the above computation, to verify \eqref{eqn ker of tns 1}, it is enough to show that 
\begin{equation}\label{eqn prop ker of tns}
\Big(\ker D_{(\boldsymbol M^{(1)}-w)^*}\otimes {\mathcal H}_2\Big)\bigcap\mathcal A_k^\perp
=\bigcap_{|\boldsymbol i|=1}\ker {({\boldsymbol M^{(1)}}- w)^*}^{\boldsymbol i}\otimes\bigcap_{|\boldsymbol i|=k+1}
\ker {(\boldsymbol M^{(2)}- w)^*}^{\boldsymbol i}.
\end{equation}
Since $K_1$ is a sharp kernel, $\ker D_{(\boldsymbol M^{(1)}-w)^*}$ is spanned by the vector $K_1(\cdot,w).$ 
It is also easy to see that the vector $K_1(\cdot,w)\otimes \bar{\partial}^{\boldsymbol j} K_2(\cdot,w)$ belongs to $\mathcal A_k^\perp$ and hence, it is in 
$\Big(\ker D_{(\boldsymbol M^{(1)}-w)^*}\otimes {\mathcal H}_2\Big)\bigcap\mathcal A_k^\perp$
for all $\boldsymbol j $ in $\mathbb Z_+^m$ with $|\boldsymbol j|\leq k$. Therefore, by \eqref{eqn ker of tns 2}, we have the inclusion 
\begin{equation}\label{onesideinequality}
\bigcap_{|\boldsymbol i|=1}\ker {({\boldsymbol M^{(1)}}- w)^*}^{\boldsymbol i}\otimes\bigcap_{|\boldsymbol i|=k+1}
\ker {(\boldsymbol M^{(2)}- w)^*}^{\boldsymbol i} \subseteq   \Big(\ker D_{(\boldsymbol M^{(1)}-w)^*}\otimes {\mathcal H}_2\Big)\bigcap\mathcal A_k^\perp.
\end{equation}
Now to prove the opposite inclusion, note that an arbitrary vector of 
$\big(\ker D_{(\boldsymbol M^{(1)}-w)^*}\otimes \hl_2\big)
\bigcap\mathcal A_k^\perp$ can be taken to be of the form $K_1(\cdot,w)\otimes g,$ where $g\in \hl_2$ is such that $K_1(\cdot,w)\otimes g\in \mathcal A_k^\perp$.  
We claim that such a vector $g$ must be in $\bigcap_{|\boldsymbol i|=k+1}
\ker {(\boldsymbol M^{(2)}-w)^*}^{\boldsymbol i}.$

As before, we realize the vectors of $\mathcal H_1\otimes \mathcal H_2$ as functions in $z=(z_1,\ldots,z_m), \zeta=(\zeta_1,\ldots,\zeta_m)$ in $\Omega$.
Fix any $\boldsymbol i \in \mathbb Z_+^m$ with $|\boldsymbol i|=k+1.$ Then 
$(\zeta-z)^{\boldsymbol i}=(\zeta_{q_1}-z_{q_1})(\zeta_{q_2}-z_{q_2})\cdots (\zeta_{q_{k+1}}-z_{q_{k+1}})$ 
for some $1\leq q_1,q_2,\ldots,q_{k+1}\leq m.$
Since $M_i^{(1)}$ and $M_i^{(2)}$ are bounded for $1\leq i\leq m ,$ for any 
$h\in \mathcal{H}_1\otimes \mathcal{H}_2$, 
we see that the function $(\zeta-z)^{\boldsymbol i} h$ belongs to $\mathcal H_1\otimes \mathcal H_2.$
Then
\begin{align*}
&\left\langle K_1(\cdot,w)\otimes g,(\zeta_{q_1}-z_{q_1})(\zeta_{q_2}-z_{q_2})\cdots (\zeta_{q_{k+1}}-z_{q_{k+1}})h\right\rangle\\
&\quad\quad\quad\quad\quad\quad\quad\quad=\left\langle M_{(\zeta_{q_1}-z_{q_1})}^* (K_1(\cdot,w)\otimes g), 
(\zeta_{q_2}-z_{q_2})\cdots (\zeta_{q_{k+1}}-z_{q_{k+1}})h\right\rangle\\
&\quad\quad\quad\quad\quad\quad\quad\quad=\left\langle (I\otimes {M^{(2)}_{q_1}}^*-{M^{(1)}_{q_1}}^*\otimes I)K_1(\cdot,w)\otimes g,
(\zeta_{q_2}-z_{q_2})\cdots (\zeta_{q_{k+1}}-z_{q_{k+1}})h\right\rangle\\
&\quad\quad\quad\quad\quad\quad\quad\quad=\left\langle K_1(\cdot,w)\otimes {M^{(2)}_{q_1}}^*g -\bar{w}_{q_1}K_1(\cdot,w)\otimes g,
(\zeta_{q_2}-z_{q_2})\cdots (\zeta_{q_{k+1}}-z_{q_{k+1}})h\right\rangle\\
&\quad\quad\quad\quad\quad\quad\quad\quad=\left\langle K_1(\cdot,w)\otimes ({M^{(2)}_{q_1}}-w_{q_1})^*g,
(\zeta_{q_2}-z_{q_2})\cdots (\zeta_{q_{k+1}}-z_{q_{k+1}})h\right \rangle.
\end{align*}
Repeating this process, we get 
$$\left\langle K_1(\cdot,w)\otimes g,(\zeta-z)^{\boldsymbol i} h\right\rangle=\left
\langle K_1(\cdot,w)\otimes {(\boldsymbol M^{(2)}-w)^*}^{\boldsymbol i} g, h\right\rangle. $$
Since $|\boldsymbol i|=k+1$, it follows that the element $(\zeta-z)^{\boldsymbol i} h$ 
belongs to $\mathcal A_k$. Furthermore, since $K_1(\cdot,w)\otimes g \in \mathcal A_k^\perp,$ from the above equality, 
we have 
$$\left\langle K_1(\cdot,w)\otimes {(\boldsymbol M^{(2)}-w)^*}^{\boldsymbol i} g, h\right\rangle=0$$
for any $h\in \mathcal H_1\otimes \mathcal H_2.$ 
Taking $h=K_1(\cdot,w)\otimes K_2(\cdot,u)$, $u\in \Omega$, we get
$K_1(w,w)\big({(\boldsymbol M^{(2)}- {w})^*}^{\boldsymbol i} g\big)(u)=0$ 
for all $u \in\Omega.$ Since $K_1(w,w)>0,$ it 
follows that
${(\boldsymbol M^{(2)}- {w})^*}^{\boldsymbol i} g=0.$  Since this is 
true for all ${\boldsymbol i} \in \mathbb Z_+^m $ with $ |{\boldsymbol i}|= k+1,$ it follows that 
$g\in \bigcap_{|{\boldsymbol i}|=k+1}\ker {(\boldsymbol M^{(2)}-w)^*}^{\boldsymbol i}.$
Hence $ K_1(\cdot,w)\otimes g $ belongs to $$\bigcap_{|\boldsymbol i|=1}\ker {(\boldsymbol M^{(1)}-w)^*}^{\boldsymbol i}\otimes\bigcap_{|\boldsymbol i|=k+1}
\ker {(\boldsymbol M^{(2)}- w)^*}^{\boldsymbol i},$$ 
proving the opposite inclusion of   $\eqref{onesideinequality}.$ This completes the proof of equality in \eqref{eqn ker of tns 1}.
\end{proof}

\begin{theorem}\label{thmsharp}
Let $\Omega \subset \mathbb C^m$ be a bounded domain.  If $K_1,K_2:\Omega\times \Omega\to
\mathbb C$ are two sharp kernels, then so is the kernel 
$J_k(K_1,K_2)_{|\rm res \, \Delta}$, $k\geq 0$.
\end{theorem}
\begin{proof}
Since the tuple $\boldsymbol M^{(1)}$ is bounded, by Corollary \ref{coroperatoronJetk}, it follows that the tuple $\boldsymbol {J}_k \boldsymbol M$ is also bounded. Now we will show that the kernel $J_k(K_1,K_2)_{|\rm res\, \Delta}$ is positive definite on $\Omega\times \Omega$. 
Since $K_2$ is positive definite, by 
Corollary \ref{cortaylor}, we obtain that the matrix $\big(\partial^{\boldsymbol i}\bar{\partial}^{\boldsymbol j} K_2(w,w)\big)_{|\boldsymbol i|,|\boldsymbol j|=0}^k$ is positive definite for $w\in \Omega$. Moreover, since $K_1$ is also positive definite, we conclude that
$J_k(K_1,K_2)_{|\rm res \, \Delta}(w,w)$ is positive definite for $w\in \Omega$. Hence, by \cite[Lemma 3.6]{Curtosalinas}, we conclude that the kernel $J_k(K_1,K_2)_{|\rm res\, \Delta}$ is positive definite.

To complete the proof, we need to show that 
$$\ker D_{(\boldsymbol J_k \boldsymbol {M}-w)^*}=\ran J_k(K_1,K_2)_{|\rm res\, \Delta}(\cdot,w),\;w\in\Omega.$$
Note that, by the definition of $R$ and $J_k$ (see the discussion before Theorem \ref{thmjetcons}), we have  
\begin{equation}\label{eqnsharpthm}
RJ_k(K_1(\cdot,w)\otimes \bar{\partial}^{\boldsymbol i} K_2(\cdot,w))
=J_k(K_1,K_2)_{|\rm res\, \Delta}(\cdot,w)e_{\boldsymbol i},~\boldsymbol i\in \mathbb Z_+^m,~
|\boldsymbol i|\leq k.
\end{equation}
In the computation below, the third equality follows from 
Lemma \ref{lemkerofunitary}, the injectivity of the map ${RJ_k}_{|{\mathcal A}_k^\perp}$ implies the fourth equality, the fifth equality follows from Lemma \ref{lemkerneloftns} and finally the last equality follows from \eqref{eqnsharpthm}:
\begin{align*}
\ker D_{(\boldsymbol {J}_k\boldsymbol M-w)^*}= &\bigcap_{p=1}^m \ker (J_kM_p-w_p)^*\\
= & \bigcap_{p=1}^m \ker\Big( (RJ_k) P_{\mathcal A_k^{\perp}}
{\big((M_p^{(1)}-{w_p})^*\otimes I\big)}_{|{\mathcal A_k^\perp}} (RJ_k)^*\Big)\\
= & \bigcap_{p=1}^m (RJ_k) \Big(\ker\big( P_{\mathcal A_k^{\perp}}
{\big((M_p^{(1)}-{w_p})^*\otimes I\big)}_{|{\mathcal A_k^\perp}}\big)\Big)\\
=& (RJ_k)\Big(\bigcap_{p=1}^m  \ker\big( P_{\mathcal A_k^{\perp}}
{\big((M_p^{(1)}-{w_p})^*\otimes I\big)}_{|{\mathcal A_k^\perp}}\big)\Big)\\
= & (RJ_k)\Big( {\bigvee}\big \{\;K_1(\cdot,w)\otimes \bar{\partial}^{\boldsymbol i}
K_2(\cdot,w):|\boldsymbol j|\leq k\big\}\;\Big)\\
= & \ran J_k(K_1,K_2)_{|\rm res \, \Delta}(\cdot,w). 
\end{align*}
This completes the proof.
\end{proof}
The lemma given below is the main tool to prove Theorem \ref{gbkjet}.
\begin{lemma}\label{lem Gleason}
Let $K_1,K_2:\Omega\times \Omega\to
\mathbb C$ be two generalized Bergman kernels, and let $w=(w_1,\ldots,w_m)$ be an arbitrary point in $\Omega.$ 
Suppose that $f$ is a 
function in $\hl_1\otimes \hl_2$ satisfying 
$\big(\big(\frac{\partial}{\partial \zeta}\big)^{\boldsymbol i}f(z,\zeta)\big)_{|z=\zeta=w}=0$ for all 
$\boldsymbol i\in  \mathbb Z_+^m,~|\boldsymbol i|\leq k.$ Then 
$$f(z,\zeta)=\sum_{j=1}^m (z_j-w_j)f_j(z,\zeta)+\sum_{|\boldsymbol q|=k+1}
(z-\zeta)^{\boldsymbol q} f^\sharp_{\boldsymbol q}(z,\zeta)$$
for some functions $f_j, f^\sharp_{\boldsymbol q}$ in $\hl_1\otimes \hl_2, j=1,\ldots,m,~
\boldsymbol q \in \mathbb Z_+^m ,|\boldsymbol q|=k+1.$ 
\end{lemma}

\begin{proof}
Since $K_1$ and $K_2$ are generalized Bergman kernels, by
Theorem $\ref{thmsalinas} $, we have that $K_1\otimes K_2$ 
is also a generalized Bergman kernel.
Therefore,  if $f$ is a function in $\hl_1\otimes\hl_2$ 
vanishing at $(w,w)$,  then using  Lemma $\ref{lemGleason1},$ we find functions $f_1,\ldots,f_m,$ and $g_1,\ldots,g_m $ in
$\hl_1\otimes\hl_2$ such that  
\begin{align*}
f(z,\zeta) &= \sum_{j=1}^m (z_j-w_j)f_j+
\sum_{j=1}^m(\zeta_j-w_j)g_j.
\end{align*}
Equivalently, we have 
\begin{align*}
f(z,\zeta) &= \sum_{j=1}^m (z_j-w_j)(f_j+g_j)+
\sum_{j=1}^m(z_j-\zeta_j)(-g_j).
\end{align*}


Thus the statement of the lemma is verified  for $k=0.$
To complete the proof by induction on $k$, 
assume that the statement is valid for some non-negative integer $k$. 
Let $f$ be a function in $\hl_1\otimes \hl_2$ such that 
$\big(\big(\frac{\partial}{\partial \zeta}\big)^{\boldsymbol i}f(z,\zeta)\big)_{|z=\zeta=w}=0$ for all $\boldsymbol i\in  \mathbb Z_+^m,~|\boldsymbol i|\leq {k+1}.$ 
By induction hypothesis, we can write
\begin{equation}\label{eqnGleason}
f(z,\zeta)=\sum_{j=1}^m (z_j-w_j) f_j(z, \zeta) +\sum_{|\boldsymbol q|=k+1}(z-\zeta
)^{\boldsymbol q} f^\sharp_{\boldsymbol q}(z, \zeta)
\end{equation}
for some $ f_j, f^\sharp_{\boldsymbol q} \in \hl_1\otimes\hl_2$, $j=1,\ldots,m,~\boldsymbol q \in \mathbb Z_+^m ,~ |\boldsymbol q|=k+1$.
Fix a $\boldsymbol i\in \mathbb Z_+^m $ with $|\boldsymbol i|= k+1.$
Applying $\big(\tfrac{\partial}{\partial \zeta
}\big)^{\boldsymbol i}$ to both sides of \eqref{eqnGleason}, we see that
\begin{align*}
\big(\tfrac{\partial}{\partial \zeta}\big)^{\boldsymbol i}f(z,\zeta)&=\sum_{j=1}^m (z_j-w_j)\big(\tfrac{\partial}{\partial \zeta
}\big)^{\boldsymbol i} f_j(z,\zeta)+\sum_{|\boldsymbol q|=k+1}\big(\tfrac{\partial}{\partial \zeta}\big)^{\boldsymbol i}\big((z-\zeta
)^{\boldsymbol q} f^\sharp_{\boldsymbol q}(z,\zeta)\big)\\
&=\sum_{j=1}^m (z_j-w_j)\big(\tfrac{\partial}{\partial \zeta
}\big)^{\boldsymbol i} f_j(z,\zeta)+\sum_{|\boldsymbol q|=k+1} \sum_{\boldsymbol p\leq \boldsymbol i}\tbinom{\boldsymbol i}{\boldsymbol p}\big(\tfrac{\partial}{\partial \zeta}\big)^{\boldsymbol p}(z-\zeta
)^{\boldsymbol q}\big(\tfrac{\partial}{\partial \zeta}\big)^{\boldsymbol {i}-\boldsymbol p} f^\sharp_{\boldsymbol q}(z,\zeta).
\end{align*}
Putting $z=\zeta=w$, we obtain 
$$\big(\big(\tfrac{\partial}{\partial \zeta}\big)^{\boldsymbol i}f(z,\zeta)\big)_{|z=\zeta=w}=(-1)^{|\boldsymbol i|}\boldsymbol i!\; f^\sharp_{\boldsymbol i}(w,w), $$
where we have used the simple identity:  $\Big(\big(\tfrac{\partial}{\partial \zeta}\big)^{\boldsymbol p}(z-\zeta
)^{\boldsymbol q}\Big)_{|z=\zeta=w}=\delta_{\boldsymbol p\boldsymbol q}(-1)^{|\boldsymbol p|} p!$.

Since $\big(\big(\tfrac{\partial}{\partial \zeta}\big)^{\boldsymbol i}f(z,\zeta)\big)_{|z=\zeta=w}=0$, we conclude that $f^\sharp_{\boldsymbol i}(w,w)=0.$ Since the statement of the lemma  has been shown to be valid for  $k=0$,  it follows that  
\begin{equation}\label{eqnGleason4}
f^\sharp_{\boldsymbol i}(z,\zeta)
=\sum_{j=1}^m (z_j-w_j)\big (f^\sharp_{\boldsymbol i}\big )_j(z,\zeta)
+\sum_{j=1}^m (z_j-\zeta_j)\big (f^\sharp_{\boldsymbol i}\big )^\sharp_j(z,\zeta)
\end{equation}
for some $\big (f^\sharp_{\boldsymbol i}\big )_j,\, \big (f^\sharp_{\boldsymbol i}\big )^\sharp_j \in \hl_1\otimes \hl_2,~ j=1,\ldots,m.$
Since \eqref{eqnGleason4} is valid for any $\boldsymbol i\in \mathbb Z_+^m,~
|\boldsymbol i|=k+1$, replacing the $f^\sharp_{\boldsymbol q}$'s in
\eqref{eqnGleason} by  $\sum_{j=1}^m (z_j-w_j)\big (f^\sharp_{\boldsymbol q}\big )_j(z,\zeta)
+\sum_{j=1}^m (z_j-\zeta_j)\big (f^\sharp_{\boldsymbol q}\big )^\sharp_j(z,\zeta)$, 
we obtain the desired conclusion after some straightforward algebraic manipulation.
\end{proof}

\begin{theorem}\label{gbkjet}
Let $\Omega\subset \mathbb C^m$ be a bounded domain. 
If $K_1,K_2:\Omega\times \Omega\to
\mathbb C$ are generalized Bergman kernels, then so is the kernel
$J_k(K_1,K_2)_{|\rm res\, \Delta}$, $k\geq 0$.
\end{theorem}

\begin{proof}
%

By Theorem $\ref{thmsharp}$, we will be done if we can show that $\ran D_{({\boldsymbol J}_k\boldsymbol {M}-w)^*}$ is closed for every $w\in\Omega.$ Fix a point $w=(w_1,\ldots,w_m)$ 
in $\Omega$. Let $\boldsymbol X :=\big(P_{\mathcal A_k^{\perp}}
{(M_1^{(1)}\otimes I)}_
{|{\mathcal A_k^\perp}},\ldots,P_{\mathcal A_k^{\perp}}
{(M_m^{(1)}\otimes I)}_{|{\mathcal A_k^\perp}}\big).$
By Corollary $\ref{coroperatoronJetk},$
we see that $\ran D_{(\boldsymbol {\boldsymbol J}_k\boldsymbol {M}-w)^*}$ is closed if and only if 
$\ran D_{(\boldsymbol X-w)^*}$ is closed.
Moreover, since $(D_{(\boldsymbol X-w)^*})^*=D^{(\boldsymbol X-w)},$ we conclude that $\ran D_{(\boldsymbol X-w)^*}$ is closed if and only if $\ran D^{(\boldsymbol X-w)}$ is closed.
Note that $X$ satisfies the following equality:
$${\ker D_{{(\boldsymbol X-w)}^*}}^\perp=\overbar {\ran (D_{{(\boldsymbol X-w)}^*})^*} 
=\overbar {\ran D^{(\boldsymbol X-w)}}.$$
Therefore, to prove $\ran D^{(\boldsymbol X-w)}$ is closed, it is enough to show that 
$
{\ker D_{{(\boldsymbol X-w)}^*}}^\perp
\subseteq \ran D^{\boldsymbol X-w}.
$
To prove this, note that 
$$D^{(\boldsymbol X-w)}(g_1\oplus\cdots\oplus g _m)=P_{\mathcal A_k^\perp}\big(\sum_{i=1}^m (z_i-w_i)g_i\big),\; g_i\in \mathcal A_k^\perp,i=1,\ldots,m.$$
Thus 
\begin{equation}\label{eq range of D_T^*}
\ran D^{(\boldsymbol X-w)}=
\Big\{P_{\mathcal A_k^\perp}\big(\sum_{i=1}^m (z_i-w_i)g_i: g_1,\ldots, g_m\in \mathcal A_k^\perp\Big \}.
\end{equation}
Now, let $f$ be an arbitrary element of ${\ker D_{(\boldsymbol X-w)^*}}^\perp.$ 
Then, by Lemma $\ref{lemkerneloftns}$ and 
Proposition $\ref{derivativeofK},$ we have 
$\big(\big(\tfrac{\partial}{\partial \zeta}\big)^{\boldsymbol i}f(z,\zeta)\big)_{|z=\zeta=w}=0$ 
for all $\boldsymbol i\in \mathbb Z_+^m$, $|\boldsymbol i|\leq k.$ 
By Lemma $\ref{lem Gleason}$,  
$$f(z,\zeta)=\sum_{j=1}^m (z_j-w_j)f_j(z,\zeta)+\sum_{|\boldsymbol q|=k+1}
(z-\zeta)^{\boldsymbol q} {f}_{\boldsymbol q}^\sharp(z,\zeta)$$
for some functions $f_j, {f}_{\boldsymbol q}^\sharp$ in $\hl_1\otimes  \hl_2,
j=1,\ldots,m$ and $\boldsymbol q \in \mathbb Z_+^m, |\boldsymbol q|=k+1.$ 
Note that the element
$\sum_{|\boldsymbol q|=k+1}(z-\zeta)^{\boldsymbol q} {f}_{\boldsymbol q}^\sharp$ belongs to 
$\mathcal A_k.$
Hence $f=P_{{\mathcal A}_k^{\perp}}(f)=P_{\mathcal A_k^{\perp}}
\big(\textstyle \sum_{j=1}^m (z_j-w_j)f_j\big)$.
Furthermore, since the subspace $\mathcal A_k$ is invariant under $(M_{j}^{(1)}-w_j)$, $j=1,\ldots,m$, we see that  
\begin{align*}
f=P_{\mathcal A_k^{\perp}}
\big(\textstyle \sum_{j=1}^m (z_j-w_j)f_j\big)&=
P_{\mathcal A_k^\perp}\Big(\textstyle \sum_{j=1}^m (z_j-w_j)\big
(P_{\mathcal A_k^\perp}f_j+P_{\mathcal A_k}f_j\big)\Big)\\
&=P_{\mathcal A_k^\perp}\big(\textstyle \sum_{j=1}^m (z_j-w_j)
(P_{\mathcal A_k^\perp}f_j)\big).
\end{align*}
Therefore, from $\eqref{eq range of D_T^*},$ 
we conclude that $f\in \ran D^{(\boldsymbol X-w)}$. 
This completes the proof. 
\end{proof}

\subsection{The class $\mathcal FB_2(\Omega)$}\label{secflag}
In this subsection, first we will use Theorem \ref{gbkjet} to prove that, if $\Omega\subset \mathbb C$, and $K^{\alpha}$, $K^\beta$, defined on $\Omega\times\Omega$, are generalized Bergman kernels, then so is the kernel $\mathbb K^{(\alpha,\beta)}$. The following proposition, which is interesting on its own right, is an essential tool in proving this theorem. The notation below is chosen to be close to that of \cite{Flagstructure}.

\begin{proposition}\label{propB_2}
Let $\Omega\subset \mathbb C$ be a bounded domain. Let $T$ be a bounded linear operator of the form $ \begin{bmatrix}
T_0 & S\\
0 &  T_1   
\end{bmatrix} $ on $H_0\bigoplus H_1$.
Suppose that $T$ belongs to $B_2(\Omega)$ and $T_0$ belongs to 
$B_1(\Omega)$. Then $T_1$ belongs to $ B_1(\Omega)$.
\end{proposition}
\begin{proof}
First, note that,
for $w\in \Omega$,
\begin{equation}\label{eqnB_1D}
(T-w)(x\oplus y) = ((T_0-w)x+Sy) \oplus (T_1-w)y.
\end{equation}
Since $T\in B_2(\mathbb D)$, $T-w$ is onto. Hence, from the above equality, it follows  that $(T_1-w)$ is onto.
 
Now we claim that $\dim \ker (T_1-w) =1$ for all $ w \in \Omega.$
From $\eqref{eqnB_1D}$, we see that $(x\oplus y)$ belongs to 
$\ker (T-w)$ if and only if $(T_0-w)x+Sy=0$ and $y\in \ker(T_1-w).$ 
Therefore, if $\dim \ker (T_1-w)$ is $0$, it must follow that 
$\ker (T-w)=\ker (T_0-w)$, which is a contradiction.
Hence $\dim \ker (T_1-w)$ is atleast $1$. Now assume that
$\dim \ker (T_1-w)> 1.$ Let $v_1(w)$ and $v_2(w)$ be two linearly
independent vectors in $\ker (T_1-w).$ Since $(T_0-w)$ is onto,
there exist $u_1(w), u_2(w)\in H_0$ such that 
$(T_0-w)u_i(w)+Sv_i(w)=0$, $i=1,2$. Hence the vectors 
$(u_1(w)\oplus v_1(w)), (u_2(w)\oplus v_2(w))$ belong to $\ker (T-w).$
Also, since $\dim \ker (T_0-w)=1,$ there exists $\gamma(w)\in  H_0,$
such that $(\gamma(w)\oplus 0)$ belongs to $\ker(T-w).$
It is easy to verify that the vectors $\{(u_1(w)\oplus v_1(w)), 
(u_2(w)\oplus v_2(w)), (\gamma(w)\oplus 0)\}$ are linearly independent.  
This is a contradiction since $\dim \ker (T-w)=2.$ Therefore $\dim \ker (T_1-w) \leq 1.$ In consequence, $\dim \ker (T_1-w)=1.$ 

Finally, to show that $\overbar \bigvee_{w\in \Omega}\ker(T_1-w)= H_1$,
let $y$ be an arbitrary vector in $H_1$ which is orthogonal to 
$\overbar \bigvee_{w\in \Omega}\ker(T_1-w)$. Then it follows that 
$(0\oplus y) $ is orthogonal to $\ker (T-w), w\in \Omega$.
Consequently, $y=0$. This completes the proof.  
\end{proof}
 
%
\begin{theorem}\label{cork^2curvgbk}
Let $\Omega \subset \mathbb{C}$ be a bounded domain and $K:\Omega\times\Omega\to \mathbb C$ be a sesqui-analytic function such that the functions $K^{\alpha}$ and $K^{\beta}$ are positive definite on $\Omega\times\Omega$ for some $\alpha,\beta>0$. Suppose that the operators ${M^{(\alpha)}}^*$ on $(\hl, K^\alpha)$ and ${M^{(\beta)}}^*$ on $(\hl, K^\beta)$ belong to $B_1(\Omega^*)$. Then the operator $~{\mathbb M^{(\alpha,\beta)}}^*$ on $(\mathcal H, \mathbb K^{(\alpha,\beta)})$ 
belongs to $B_1(\Omega^*)$. Equivalently, 
if $K^\alpha$ and $K^\beta$ are generalized Bergman kernels, then so is the kernel $\;\mathbb K^{(\alpha,\beta)}$. 
\end{theorem} 
\begin{proof}
Since the operators ${M^{(\alpha)}}^*$  and ${M^{(\beta)}}^*$ belong to $B_1(\Omega^*)$, it follows from Theorem \ref{gbkjet} that the kernel $J_1(K^{\alpha},K^\beta)_{|\rm res \, \Delta}$ is a generalized Bergman kernel. Therefore, from corollary \ref{quotient module1}, we deduce that the operator $\left(\begin{smallmatrix}
{M^{(\alpha+\beta)}}^*&\eta \;\mathsf{inc}^*\\
0 & {\mathbb M^{(\alpha,\beta)}}^*
\end{smallmatrix}\right)
$ belongs to $B_2(\Omega^*)$, 
where $\eta=\frac{\beta}{\sqrt{\alpha\beta(\alpha+\beta)}}$ and 
$\mathsf{inc}$ is the inclusion operator from $(\mathcal{H}, K^{\alpha+\beta})$ into 
$(\mathcal H, \mathbb K^{(\alpha,\beta)})$.
Also, by Theorem \ref{thmsalinas}, the operator ${M^{(\alpha+\beta)}}^*$ on $(\mathcal H, K^{\alpha+\beta})$
belongs to $B_1(\Omega^*)$.
Proposition $\ref{propB_2},$ therefore shows that the operator $~{\mathbb M^{(\alpha,\beta)}}^*$ on $(\mathcal H, \mathbb K^{(\alpha,\beta)})$
belongs to $B_1(\Omega^*).$  
\end{proof}
A smaller class of operators $\mathcal FB_n(\Omega)$ from $B_n(\Omega)$, $n\geq 2$, was introduced in \cite{Flagstructure}.  A set of tractable complete unitary invariants and concrete  models were given for operators in this class.  We give below examples of a large class of  operators in $\mathcal FB_2(\Omega)$. In case $\Omega$ is the unit disc $\mathbb D$, these examples include the homogeneous operators of rank $2$ in $B_2(\mathbb D)$ which are known to be in  $\mathcal FB_2(\mathbb D)$.
\begin{definition}
An operator $T$ on $H_0\bigoplus H_1$ is said to be in $\mathcal FB_2(\Omega)$ 
if it is of the form 
$ \begin{bmatrix}
T_0 & S\\
0 &  T_1   
\end{bmatrix}$, 
where $T_0,T_1\in B_1(\Omega)$ and $S$ is a non-zero operator satisfying
$T_0S=ST_1.$ 
\end{definition}

\begin{theorem}\label{ThmFB_2)}
Let $\Omega \subset \mathbb{C}$ be a bounded domain and $K:\Omega\times\Omega\to \mathbb C$ be a sesqui-analytic function such that the functions $K^{\alpha}$ and $K^{\beta}$ are positive definite on $\Omega\times\Omega$ for some $\alpha,\beta>0$. Suppose that the operators ${M^{(\alpha)}}^*$ on $(\hl, K^\alpha)$ and ${M^{(\beta)}}^*$ on $(\hl, K^\beta)$ belong to $B_1(\Omega^*)$. Then the operator $ (J_1M)^*$ on $(\mathcal H, J_1(K^{\alpha}, K^{\beta})_{|\rm res\, \Delta})$
belongs to $\mathcal F B_2(\Omega^*).$      \end{theorem}

\begin{proof}
By Theorem $ \ref{gbkjet}$,  
the operator $(J_1M)^*$ on $(\mathcal H, J_1(K^{\alpha}, K^{\beta})_{|\rm res\, \Delta})$ belongs to 
$B_2(\Omega^*)$, and by Corollary $\ref{quotient module1},$ it is unitarily equivalent to 
 $\left(\begin{smallmatrix}
{M^{(\alpha+\beta)}}^*&\eta \;\mathsf{inc}^*\\
0 & {\mathbb M^{(\alpha,\beta)}}^*
\end{smallmatrix}\right)
$ on $(\mathcal{H}, K^{\alpha+\beta})\bigoplus(\mathcal 
H, \mathbb K^{(\alpha,\beta)})$.  
By Theorem $\ref{thmsalinas}$, the operator ${M^{(\alpha+\beta)}}^*$ on $(\mathcal H, K^{\alpha+\beta})$ belongs to $B_1(\Omega^*)$ and by Theorem $\ref{cork^2curvgbk}$, the operator ${\mathbb M^{(\alpha,\beta)}}^*$ on $(\mathcal H, \mathbb K^{(\alpha,\beta)})$ belongs to $B_1(\Omega^*)$. The adjoint of the inclusion operator $\mathsf{inc}$ clearly intertwines ${M^{(\alpha + \beta)}}^*$ and ${\mathbb M^{(\alpha,\beta)}}^*$.   
Therefore the operator $(J_1M)^*$ on $(\mathcal H, J_1(K^{\alpha}, K^{\beta})_{|\rm res\, \Delta})$
belongs to $\mathcal FB_2(\Omega^*).$
\end{proof}
Let $\Omega \subset \mathbb{C}$ be a bounded domain and $K:\Omega\times\Omega\to \mathbb C$ be a sesqui-analytic function such that the functions $K^{\alpha_1}, K^{\alpha_2},K^{\beta_1}$ and $K^{\beta_2}$ are positive definite on $\Omega\times\Omega$ for some $\alpha_i, \beta_i >0$, $i=1,2$. Suppose that the operators ${M^{(\alpha_i)}}^*$ on $(\hl, K^{\alpha_i})$ and ${M^{(\beta_i)}}^*$ on $(\hl, K^{\beta_i})$, $i=1,2$, belong to $B_1(\Omega^*)$. Let $\mathcal A_1(\alpha_i,\beta_i)$ be the subspace $\mathcal A_1$ of the Hilbert space $(\hl ,K^{\alpha_i})\otimes (\hl, K^{\beta_i})$ for $i=1,2$. Then we have the following corollary.
\begin{corollary}
The operators $\big(M^{(\alpha_1)}\otimes I\big)^*_{|\mathcal A_1(\alpha_1,\beta_1)^\perp}$ and $\big(M^{(\alpha_2)}\otimes I\big)^*_{|\mathcal A_1(\alpha_2,\beta_2)^\perp}$ are unitarily 
equivalent if and only if $\alpha_1=\alpha_2$ and $\beta_1=\beta_2$.
\end{corollary}
\begin{proof}
If $\alpha_1=\alpha_2$ and $\beta_1=\beta_2$, then 
there is nothing to prove. For the converse, assume that the operators $\big(M^{(\alpha_1)}\otimes I\big)^*_{|\mathcal A_1(\alpha_1,\beta_1)^\perp}$ and $\big(M^{(\alpha_2)}\otimes I\big)^*_{|\mathcal A_1(\alpha_2,\beta_2)^\perp}$ are unitarily equivalent. Then, by Corollary $\ref{quotient module}$, we see that the operators $\left(\begin{smallmatrix}
{M^{(\alpha_1+\beta_1)}}^*&\eta_1 \;(\mathsf{inc})_1^*\\
0 & {\mathbb M^{(\alpha_1,\beta_1)}}^*
\end{smallmatrix}\right)
$ on $(\mathcal{H}, K^{\alpha_1+\beta_1})\bigoplus(\mathcal 
H, \mathbb K^{(\alpha_1,\beta_1)})$  
and $\left(\begin{smallmatrix}
{M^{(\alpha_2+\beta_2)}}^*&\eta_2 \;(\mathsf{inc})_2^*\\
0 & {\mathbb M^{(\alpha_2,\beta_2)}}^*
\end{smallmatrix}\right)
$ on $(\mathcal{H}, K^{\alpha_2+\beta_2})\bigoplus(\mathcal 
H, \mathbb K^{(\alpha_2,\beta_2)})$  
are unitarily equivalent, where
 $\eta_i=\frac{\beta_i}{\sqrt{\alpha_i\beta_i(\alpha_i+\beta_i)}}$ and 
$(\mathsf{inc})_i$ is the inclusion operator from $(\mathcal{H}, K^{\alpha_i+\beta_i})$ into 
$(\mathcal H, \mathbb K^{(\alpha_i,\beta_i)})$, $i=1,2$. Since  ${M^{(\alpha_i)}}^*$ on $(\hl, K^{\alpha_i})$ and ${M^{(\beta_i)}}^*$ on $(\hl, K^{\beta_i})$, $i=1,2$, belong to $B_1(\Omega^*)$, by Theorem \ref{ThmFB_2)}, we conclude that  the operator
$\left(\begin{smallmatrix}
{M^{(\alpha_i+\beta_i)}}^*&\eta_i \;(\mathsf{inc})_i^*\\
0 & {\mathbb M^{(\alpha_i,\beta_i)}}^*
\end{smallmatrix}\right)
$ 
belongs to $\mathcal FB_2(\Omega^*)$ for $i=1,2$. Therefore,
by \cite[Theorem 2.10]{Flagstructure}, we obtain that 
\begin{equation}\label{eqnFB_22}
\mathcal K_{{M^{(\alpha_1+\beta_1)}}^*}=\mathcal K_{{M^{(\alpha_2+\beta_2)}}^*} \;\;\;\;\;\mbox{and\;}\;\;\;
\frac{\eta_1 \;\|(\mathsf{inc})_1^*(t_1)\|^2}{\|t_1\|^2}=\frac{\eta_2 \;\|(\mathsf{inc})_2^*(t_2)\|^2}{\|t_2\|^2},
\end{equation}
where $\mathcal K_{{M^{(\alpha_i+\beta_i)}}^*}$, $i=1,2$, is the curvature of the operator ${M^{(\alpha_i+\beta_i)}}^*$, and $t_1$ and $t_2$ are two non-vanishing holomorphic sections of the vector bundles 
$E_{{\mathbb M^{(\alpha_1,\beta_1)}}^*}$ and $E_{{\mathbb M^{(\alpha_2,\beta_2)}}^*}$, respectively. 
Note that, for $i=1,2$, $t_i(w)=\mathbb K^{(\alpha_i,\beta_i)}(\cdot,w)$ is a holomorphic non-vanishing section of the vector bundle $E_{{\mathbb M^{(\alpha_i,\beta_i)}}^*}$, and also $(\mathsf{inc})_i^*(\mathbb K^{(\alpha_i,\beta_i)}(\cdot,w))=K^{\alpha_i+\beta_i}(\cdot,w)$, $w\in \Omega$. Therefore 
the second equality in \eqref{eqnFB_22}
implies that $$\frac{\eta_1K^{\alpha_1+\beta_1}(w,w)}{K^{\alpha_1+\beta_1}(w,w)\partial \bar{\partial}\log K(w,w)}=\frac{\eta_2K^{\alpha_2+\beta_2}(w,w)}{K^{\alpha_2+\beta_2}(w,w)\partial \bar{\partial}\log K(w,w)},\;\;w\in \Omega,$$ or equivalently $\eta_1=\eta_2$.
Furthermore, it is easy to see that $\mathcal K_{{M^{(\alpha_1+\beta_1)}}^*}=\mathcal K_{{M^{(\alpha_2+\beta_2)}}^*}$ if and only if $\alpha_1+\beta_1=\alpha_2+\beta_2$. Hence, from
\eqref{eqnFB_22}, we see that
\begin{equation}\label{eqnFB_21}
\alpha_1+\beta_1=\alpha_2+\beta_2\;\;\;\;\mbox{and}\;\;\;\;\eta_1=\eta_2.
\end{equation}
Then a simple calculation shows that \eqref{eqnFB_21} 
is equivalent to $\alpha_1=\alpha_2$ and $\beta_1=\beta_2$, completing the proof.  
\end{proof}

\section{The generalized Wallach set}
Let $\Omega$ be a bounded domain in $\mathbb C^m$. Recall that the Bergman space $A^2(\Omega)$ is the  Hilbert space of all square integrable analytic functions defined on $\Omega$.  The inner product of $A^2(\Omega)$ is given by the formula 
$$\left \langle f, g\right \rangle : = \int_{\Omega} f(z)\overbar{g(z)}~ \rm dV(z),~f,g\in A^2(\Omega),$$
where ${\rm dV(z)}$ is the area measure on $\mathbb C^m$. 
The evaluation linear functional $f\mapsto f(w)$ is bounded on $A^2(\Omega)$ for all $w\in \Omega$. Consequently, the Bergman space is a reproducing kernel Hilbert space. The reproducing kernel of the Bergman space $A^2(\Omega)$ is called the Bergman kernel of $\Omega$ and is denoted by $B_{\Omega}$.


If $\Omega\subset \mathbb C^m$ is a bounded symmetric domain, then the ordinary Wallach set $\mathcal W_{\Omega}$ is defined as $\{t >  0: B_\Omega^t \mbox{\rm \: is non-negative definite}\}$. Here $B_{\Omega}^t$, $t> 0$, makes sense since every bounded symmetric domain $\Omega$ is simply connected and the Bergman kernel on it is non-vanishing. If $\Omega$ is the Euclidean unit ball $\mathbb B_m$, then the Bergman kernel  is given by 
\begin{equation}\label{eqnballun}
B_{\mathbb B_m}(z,w)=(1-\langle  z,w \rangle)^{-(m+1)},\;\; z,w\in B_{\mathbb B_m},
\end{equation}
and the Wallach set 
$\mathcal W_{\mathbb B_m}=\{t\in \mathbb R: t > 0\}$. But, in general, there are examples of bounded symmetric domains, like the open unit ball in the space of all $m\times n$ matrices, $m,n>1$, with respect to the operator norm, where the Wallach set is a proper subset of $\{t\in \mathbb R: t > 0\}$. An explicit description of the Wallach set $\mathcal W_{\Omega}$ for a  bounded symmetric domain $\Omega$ is given in \cite{Wallachset}.

Replacing the Bergman kernel in the definition of the Wallach set by an arbitrary scalar valued non-negative definite kernel $K$, we define the ordinary Wallach set $\mathcal W(K)$  to be the set 
$$\{t > 0: K^t \mbox{\rm \: is non-negative definite}\}.$$
Here we have assumed that there exists a continuous branch of logarithm of $K$ on $\Omega\times \Omega$ and therefore $K^t$, $t>0$, makes sense. 
Clearly, every natural number belongs to the Wallach set $\mathcal W(K)$.
In \cite{infinitelydivisiblemetric}, it is 
shown that $K^t$ is non-negative definite for all $t> 0$ if and only if $\big(\partial_i\bar{\partial}_j \log K(z,w)\big)_{i,j=1}^m$ is non-negative definite. Therefore it follows from the discussion in the previous paragraph that there are
non-negative definite kernels $K$ on $\Omega\times \Omega$ for which  $\big(\partial_i\bar{\partial}_j \log K(z,w)\big)_{i,j=1}^m$ need not define a non-negative definite kernel on $\Omega\times \Omega$. 
However, it follows from Proposition \ref{k^2curv} that $K^{t_1+t_2}\big(\partial_i\bar{\partial}_j \log K(z,w)\big)_{i,j=1}^m$ is a non-negative kernel on $\Omega\times \Omega$ as soon as $t_1$ and $t_2$ are in the Wallach set $\mathcal W(K)$. Therefore it is natural to introduce the generalized Wallach set for any scalar valued kernel $K$ defined on $\Omega\times \Omega$ as follows:
\begin{equation} 
G\mathcal W(K):=\big \{t \in \mathbb R :\; K^{t-2} \mathbb K \mbox{\rm\: is non-negative definite}\big \},
\end{equation} 
where, as before, we have assumed that $K^t$ is well defined for all $t\in \mathbb R$. Clearly, we have the following inclusion 
$$\big\{t_1+t_2\,:\, t_1,t_2\in \mathcal W(K)\big\} \subseteq G\mathcal W(K).$$


\subsection
{Generalized Wallach set for the Bergman kernel of the 
Euclidean unit ball in $\mathbb C^m$}
In this section, we compute the generalized Wallach set for the Bergman kernel of the Euclidean unit ball in $\mathbb C^m$. 
In the case of the unit disc $\D$, the Bergman kernel $B_{\D}(z,w)=(1-z\bar{w})^{-2}$ and $ \partial\bar{\partial} \log B_{\D}(z,w)=2(1-z\bar{w})^{-2}$, $z,w\in \D.$ Therefore $t$ is in $G\mathcal W(B_{\D})$ if and only if $(1-z\bar{w})^{-(2t+2)}$ is non-negative definite on $\D\times\D$. Consequently, $G\mathcal W(B_{\D})=\{t\in \mathbb R:t\geq -1 \}.$
For the case of the Bergman kernel $B_{\mathbb B_m}$ of the Euclidean unit ball $\mathbb B_m$, $m\geq 2$, we have shown that $G\mathcal W(B_{\mathbb B_m})=\{t\in \mathbb R:t\geq 0 \}.$ The proof is obtained by putting together a number of lemmas which are of independent interest. 

Before computing the generalized Wallach set $G\mathcal W (B_{\mathbb B_m})$ for the Bergman kernel of the Euclidean ball $\mathbb B_m$, we point out that the result is already included in  \cite[Theorem 3.7]{Misra-Upmeier}, see  also \cite{JP,  HLZ}.  The justification for our detailed proofs in this particular case is that  it is direct and elementary in nature. 

As before, we write $K\succeq 0$ to denote that $K$ is a non-negative definite kernel. 
For two non-negative definite kernels $K_1, K_2:\Omega\times\Omega\to \mathcal M_k(\mathbb C)$, we write $K_1\preceq K_2$  if $K_2-K_1$ is a non-negative definite kernel on $\Omega\times\Omega$. Analogously, we write $K_1\succeq K_2$ if $K_1-K_2$ is non-negative definite.

\begin{lemma}\label{finitenorm}
Let $\Omega$ be a bounded domain in $\mathbb C^m$, and $\lambda_0>0$ be an arbitrary constant. 
Let $\left\{K_{\lambda}\right\}_{\lambda\geq \lambda_0}$ be a family of  non-negative definite kernels, defined on $\Omega\times\Omega$, taking values in $\mathcal M_k(\mathbb C)$ such that 
\begin{enumerate}
\item [\rm{(i)}] if $\lambda\geq {\lambda}^\prime \geq \lambda_0$, then $K_{{\lambda}^\prime} \preceq K_{\lambda}$,
\item [\rm{(ii)}] for $z,w\in \Omega$, $K_{\lambda}(z,w)$ converges to $ K_{\lambda_0}(z,w)$ entrywise as $\lambda\to \lambda_0$.
\end{enumerate}
Any $f:\Omega\to \mathbb C^k$ which is holomorphic and is in $(\mathcal H, K_{\lambda})$ for all $\lambda>\lambda_0$ belongs to 
$(\mathcal H, K_{\lambda_0})$ if and only if $\sup_{\lambda>\lambda_0}\|f\|_{(\mathcal H, K_{\lambda})}<\infty.$
\end{lemma}

\begin{proof}
Recall that if $K$ and $K^\prime$ are two non-negative definite kernels satisfying $K \preceq K^\prime$, then $(\mathcal H, K)\subseteq (\mathcal H, K^\prime)$ and $\|h\|_{(\mathcal H, K^\prime)}\leq \|h\|_{(\mathcal H, K)}$ for $h\in (\mathcal H, K)$ (see \cite[Theorem 6.25]{PaulsenRaghupati}). Therefore, by the hypothesis, we have that 
\begin{equation}\label{eqngwslem}
(\mathcal H, K_{\lambda^\prime})\subseteq (\mathcal H, K_{\lambda})\; \;\;\text{and}\;\;\; \|h\|_{(\mathcal H, K_{\lambda})}\leq \|h\|_{(\mathcal H, K_{\lambda^\prime})},
\end{equation}
whenever $\lambda\geq {\lambda}^\prime \geq \lambda_0$ and $h\in (\mathcal H, K_{\lambda^\prime})$.

Now assume that $f\in (\mathcal H, K_{\lambda_0}).$ Then, clearly $\|f\|_{(\mathcal H, K_{\lambda})}\leq \|f\|_{(\mathcal H, K_{\lambda_0})}$ for all $\lambda>\lambda_0.$ Consequently, $\sup_{\lambda>\lambda_0}\|f\|_{(\mathcal H, K_{\lambda})}\leq \|f\|_{(\mathcal H, K_{\lambda_0})}<\infty$.
For the converse, assume that $\sup_{\lambda>\lambda_0}\|f\|_{(\mathcal H, K_{\lambda})}<\infty.$ Then, from \eqref{eqngwslem}, it follows that $\lim_{\lambda\to\lambda_0}\|f\|_{(\mathcal H, K_{\lambda})}$ exists and is equal to $\sup_{\lambda>\lambda_0}\|f\|_{(\mathcal H, K_{\lambda})}.$ Since $f\in (\mathcal H, K_{\lambda})$ for all $\lambda>\lambda_0$, by \cite[Theorem 6.23]{PaulsenRaghupati}, we have that
$$f(z)f(w)^*\preceq \|f\|^2_{(\mathcal H,K_{\lambda})}K_{\lambda}(z,w).$$
Taking limit as $\lambda\to \lambda_0$ and using part $\rm(ii)$
of the hypothesis,  we obtain
$$f(z)f(w)^*\preceq \sup_{\lambda>\lambda_0}\|f\|^2_{(\mathcal H, K_{\lambda})}K_{\lambda_0}(z,w).$$
Hence, using \cite[Theorem 6.23]{PaulsenRaghupati} once again, we conclude that $f\in (\mathcal H, K_{\lambda_0}).$
\end{proof}

%

If $m\geq 2$, then from \eqref{eqnballun},  we have 
\begin{align}\label{gWSeq2}
\begin{split}
&\Big(\big(B_{\mathbb B_m}^{t}\partial_i\bar{\partial}_j\log B_{\mathbb B_m}\big)(z,w)\Big)_{i,j=1}^m\\
&\quad\quad\quad\quad=
\frac{m+1}{(1-\left\langle z, w\right \rangle)^{t (m+1)+2}} \begin{pmatrix}
1-\sum_{j\neq 1}z_j\bar{w}_j&z_2\bar{w}_1&\cdots& z_m\bar{w}_1\\
z_1\bar{w}_2&1-\sum_{j\neq 2}z_j\bar{w}_j&\cdots& z_m\bar{w}_2\\
\vdots & \vdots& \vdots & \vdots\\
z_1\bar{w}_m &z_2\bar{w}_m & \cdots & 1-\sum_{j\neq m}z_j\bar{w}_j
\end{pmatrix}.
\end{split}
\end{align}
For $m\geq 2$, $\lambda\in \mathbb R$ and $z,w\in \mathbb B_m$, set 
\begin{equation}\label{gWSeq1}
{\mathbb K}_{\lambda}(z,w):=\frac {1}{(1-\langle z,w \rangle )^{\lambda}}
\begin{pmatrix}
1-\sum_{j\neq 1}z_j\bar{w}_j & z_2\bar{w}_1&\cdots& z_m\bar{w}_1\\
z_1\bar{w}_2&1-\sum_{j\neq 2}z_j\bar{w}_j&\cdots& z_m\bar{w}_2\\
\vdots & \vdots& \vdots & \vdots\\
z_1\bar{w}_m &z_2\bar{w}_m & \cdots & 1-\sum_{j\neq m}z_j\bar{w}_j\\
\end{pmatrix}.
\end{equation}
In view \eqref{gWSeq2} and \eqref{gWSeq1}, for $\lambda>2$, we have
$$
{\mathbb K}_{\lambda}= \frac{2}{t(m+1)} \Big(\;({B_{\mathbb B_m}^{\frac{t}{2}}})^2\partial_i\bar{\partial}_j \log {B_{\mathbb B_m}^{\frac{t}{2}}}\; \Big)_{i,j=1}^m,
$$
 where $t=\frac{\lambda-2}{m+1}>0$.
Since  $B_{\mathbb B_m}^{t/2}$ is positive definite on $\mathbb B_m\times \mathbb B_m$ for $t>0$, it follows from Corollary \ref{cork^2curv} that $\mathbb K_{\lambda}$ is non-negative definite on $\mathbb B_m\times \mathbb B_m$ for $\lambda>2$. Since $\mathbb K_\lambda(z,w)\to \mathbb K_2(z,w)$, $z,w\in \mathbb B_m$, entrywise as $\lambda\to 2$, we conclude that $\mathbb K_{2}$ is also non-negative definite on $\mathbb B_m\times \mathbb B_m$. 


Let $\{e_1,\ldots,e_m\}$ be the standard basis of $\mathbb C^m$. The lemma given below finds the norm of the vector $z_2\otimes e_1$ in $(\mathcal H, \mathbb K_{\lambda})$ when  $\lambda>2$. 
 
\begin{lemma}\label{normz_2}
For each $\lambda>2$, the vector 
$z_2\otimes e_1$ belongs to $(\mathcal H, \mathbb K_{\lambda})$   and 
$\|z_2\otimes e_1\|_{(\mathcal H, \mathbb K_{\lambda})}=\sqrt{\frac {\lambda-1}{\lambda(\lambda-2)}}.$
\end{lemma}
\begin{proof}
By a straight forward computation, we obtain
$$\bar\partial_1 \mathbb K_{\lambda}(\cdot,0)e_2=z_2\otimes e_1+(\lambda-1)z_1\otimes e_2$$
and 
$$\bar\partial_2\mathbb K_{\lambda}(\cdot,0)e_1=(\lambda-1)z_2\otimes e_1+z_1\otimes e_2.$$
Thus we have
\begin{equation}\label{eqnnormofz_2}
(\lambda-1)\bar\partial_2 \mathbb K_{\lambda}(\cdot,0)e_1-\bar\partial_1 \mathbb K_{\lambda}(\cdot,0)e_2=(\lambda^2-2\lambda)z_2\otimes e_1.
\end{equation}

By Proposition \ref{derivativeofK}, the vectors $\bar\partial_2 \mathbb K_{\lambda}(\cdot,0)e_1$ and $\bar\partial_1 \mathbb K_{\lambda}(\cdot,0)e_2$ belong to $(\mathcal H, \mathbb K_{\lambda})$. Since $\lambda>2$, from \eqref{eqnnormofz_2}, it follows that the vector $z_2\otimes e_1$ belongs to $(\mathcal H, \mathbb K_{\lambda})$. Now, taking norm in both sides of 
\eqref{eqnnormofz_2} and using Proposition \ref{derivativeofK} a second time, we obtain
\begin{align}\label{eqnnormofz_2'}
\begin{split}
&(\lambda^2-2\lambda)^2\|z_2\otimes e_1\|^2\\
& \quad\quad=
(\lambda-1)^2\langle\partial_2\bar\partial_2 \mathbb K_{\lambda}(0,0)e_1,e_1\rangle -(\lambda-1)\langle\partial_1\bar\partial_2\mathbb K_{\lambda}(0,0)e_1,e_2\rangle\\
&\quad\quad\quad\quad\quad-(\lambda-1)\langle\bar\partial_1\partial_2 \mathbb K_{\lambda}(0,0)e_2,e_1\rangle
+\langle\partial_1\bar\partial_1 \mathbb K_{\lambda}(0,0)e_2,e_2\rangle
\end{split}
\end{align}
By a routine computation, we obtain 
$$\partial_i\bar \partial_j \mathbb K_{\lambda}(0,0)=(\lambda-1)\delta_{ij} I_m+E_{ji},$$
where $\delta_{ij}$ is the Kronecker delta function, $I_m$ is the identity matrix of order $m$, and $E_{ji}$ is the matrix whose $(j,i)$th entry is 1 and all other entries are 0.
Hence, from \eqref{eqnnormofz_2'}, we see that 
\begin{align*}
&(\lambda^2-2\lambda)^2||z_2\otimes e_1||^2\\&
\quad\quad \quad=(\lambda-1)^2(\lambda-1)-2(\lambda-1)+(\lambda-1)\\&\quad\quad\quad=(\lambda-1)(\lambda^2-2\lambda).
\end{align*}
Hence $||z_2\otimes e_1||=\sqrt{\frac {\lambda-1}{\lambda(\lambda-2)}}$, completing the proof of the lemma.
\end{proof}
%
%

\begin{lemma}\label{M_2 is not bdd}
The multiplication operator by the coordinate function $z_2$  on $(\mathcal H,\mathbb K_2)$ is not bounded.
\end{lemma}
\begin{proof}
Since $\mathbb K_2(\cdot,0)e_1=e_1$, we have that the constant function $ e_1$ is in $(\mathcal H, \mathbb K_2).$ Hence, to prove that $M_{z_2}$ is not bounded on $(\mathcal H, \mathbb K_2)$, it suffices to show that the vector  $z_2\otimes e_1$ does not belong to $(\mathcal H, \mathbb K_2).$ 

Consider the family of non-negative definite kernels $\left\{\mathbb K_{\lambda}\right\}_{\lambda \geq 2}$. Observe that for 
$\lambda\geq \lambda'\geq 2,$
\begin{equation}\label{eqnnormofz_2"}
\mathbb K_{\lambda}(z,w)-\mathbb K_{\lambda'}(z,w)=\left((1-\langle z,w\rangle)^{-(\lambda-\lambda')}-1\right)\mathbb K_{\lambda'}(z,w).
\end{equation}
It is easy to see that if $\lambda\geq\lambda'$, then 
$(1-\langle z,w\rangle)^{-(\lambda-\lambda')}-1 \succeq 0.$
Thus the right hand side of \eqref{eqnnormofz_2"}, being a product of a scalar valued non-negative definite kernel with a matrix valued non-negative definite kernel, is non-negative definite. 
Consequently, $K_{\lambda^\prime}\preceq K_{\lambda}$. Also  
since $\mathbb K_{\lambda}(z,w)\to \mathbb K_{2}(z,w)$ entry-wise as $\lambda\to 2$, by Lemma $\ref{finitenorm}$,
it follows that $z_2\otimes e_1\in (\mathcal H, \mathbb K_{2})$ if and only if  $\sup_{\lambda>2}\|z_2\otimes e_1\|_{(\mathcal H, \mathbb K_{\lambda})}<\infty.$ 
By lemma $\ref{normz_2},$  we have $\|z_2\otimes e_1\|_{(\mathcal H, \mathbb K_{\lambda})}=\sqrt{\frac {\lambda-1}{\lambda(\lambda-2)}}.$ Thus $\sup_{\lambda>2}||z_2\otimes e_1||_{(\hl,\mathbb K_{\lambda})}=\infty.$ Hence the vector $z_2\otimes e_1$ does not belong to $(\hl, \mathbb K_2)$ and the operator $M_{z_2}$ on $(\mathcal H, \mathbb K_{\lambda})$ is not bounded. 
\end{proof}

The following theorem describes the generalized Wallach set for the Bergman kernel of the Euclidean unit ball in $\mathbb C^m$, $m\geq 2$.

\begin{theorem}\label{thmGWall}
If $m\geq 2$, then 
$G\mathcal{W}(B_{\mathbb B_m})=\{t\in \mathbb R: t\geq 0\}$.
\end{theorem}

\begin{proof}
In view of \eqref{gWSeq2} and \eqref{gWSeq1}, we see that $t\in G\mathcal{W}(B_{\mathbb B_m})$ if and only if 
$\mathbb K_{t(m+1)+2}$ is non-negative definite on $\mathbb B_m\times \mathbb B_m$. Hence we will be done if we can show that $\mathbb K_{\lambda}$ is non-negative if and only if $\lambda\geq 2$. 

From the discussion preceding Lemma \ref{normz_2}, we have that $\mathbb K_{\lambda}$ is non-negative definite on $\mathbb B_m\times\mathbb B_m$ for $\lambda\geq 2$. 

To prove the converse, assume that $\mathbb K_{\lambda}$ is non-negative definite
for some $\lambda < 2$. Note that $\mathbb K_{2}$ can be written as the product 
\begin{equation}\label{eqnwallach}
\mathbb K_{2}(z,w)=(1-\langle z,w\rangle)^{-(2-\lambda)}\mathbb K_{\lambda}(z,w),\;z,w\in \mathbb B_m.
\end{equation} 
Also, the multiplication operator $M_{z_2}$ on  $(\mathcal H, (1-\langle z,w\rangle)^{-(2-\lambda)})$ is bounded. Hence, by Lemma \ref{lembounded}, there exists a constant $c>0$ such that $(c^2-z_2\bar{w}_2)(1-\langle z,w\rangle)^{-(2-\lambda)}$ is non-negative definite. Consequently, the product $(c^2-z_2\bar{w}_2)(1-\langle z,w\rangle)^{-(2-\lambda)}\mathbb K_{\lambda}$, which is $(c^2-z_2\bar{w}_2)\mathbb K_2$, is non-negative. Hence, again by Lemma \ref{lembounded}, it  follows that the operator $M_{z_2}$ is  bounded on $(\mathcal H, \mathbb K_2)$. This is a contradiction to the Lemma $\ref{M_2 is not bdd}$. Hence our assumption that $\mathbb K_{\lambda}$ is non-negative for some $\lambda <2$, is not valid. This completes the proof.
\end{proof}
\section{Quasi-invariant kernels}
%
In this section, we show that if $K$ a is quasi-invariant kernel with respect to some $J$, then  $K^{t-2}\mathbb K$ is also a quasi-invariant kernel with respect to  $\mathbb J:=J(\varphi,z)^{t} D\varphi(z)^{\rm tr}$, $\varphi\in \rm Aut(\Omega),~ z\in \Omega$, whenever $t$ is in the generalized Wallach set $G\mathcal W(K)$.
The lemma given below, which will be used in the proof of the Proposition \ref{proptransformation}, 
follows from applying the chain rule 
\cite[page 8]{Rudinunitball} twice.
 
\begin{lemma}\label{lemmatransformationrule}
Let $\phi=(\phi_1,\ldots,\phi_m):\Omega\to \mathbb C^m$ be a holomorphic map and $g:\ran \phi\to \mathbb C$ be a real analytic function. If $h=g\circ \phi$, then 
$$\Big (\;\big(\partial_i\bar{\partial}_j h\;\big)(z)\Big)_{i,j=1}^m 
= ({D\phi(z)})^{\rm tr} \Big (\;\big(\partial_i\bar{\partial}_j g\;\big)(\varphi(z))\Big)_{i,j=1}^m \overbar{({D\phi(z)})},
$$
where $(D\phi)(z)^{\rm tr}$ is the transpose of the derivative of $\phi$ at $z$.
\end{lemma}



\begin{proposition}\label{proptransformation}
Let $\Omega\subset \mathbb C^m$ be a bounded domain. Let $K:\Omega\times\Omega\to \mathbb C$ be a non-negative definite kernel
and $J:\rm Aut(\Omega)\times \Omega\to \mathbb C\setminus\{0\}$ be a function such that $J(\varphi,\cdot)$ is holomorphic for each $\varphi$ in $\rm Aut(\Omega)$. Suppose that $K$ is quasi-invariant with respect to $J$.  Then the kernel $ K^{t-2}\mathbb{K}$
is also quasi-invariant with respect to
$\mathbb J$ whenever $t\in G\mathcal W_\Omega(K)$, where $ \mathbb J(\varphi,z)=J(\varphi,z)^{t} D\varphi(z)^{\rm tr}$, $\varphi\in \rm Aut(\Omega),~ z\in \Omega.$ 
\end{proposition}

\begin{proof}
Since $K$ is quasi-invariant with respect to $J$, we have
$$\log K(z,z) = \log |J(\varphi,z)|^2 + \log K(\varphi(z),\varphi(z)),\;\varphi\in {\rm Aut}(\Omega),~ z\in \Omega.$$
Also, $J(\varphi,\cdot)$ is a non-vanishing holomorphic function on $\Omega,$ therefore 
${\partial_i}{\bar{\partial_j}} \log |J(\varphi,z)|^2=0$. Hence
\begin{equation}\label{eqntransrule1}
{\partial_i}{\bar{\partial}_j}\log K(z,z)=
 {\partial_i}{\bar{\partial}_j}\log K(\varphi(z),\varphi(z)),\;\varphi\in {\rm Aut}(\Omega),~ z\in \Omega.
\end{equation}  

Any biholomorphic automorphism $\varphi$ of  $\Omega$ is of the form $(\varphi_1,\ldots,\varphi_m)$, where $\varphi_i:\Omega\to \mathbb C$ is holomorphic, $i=1,\ldots,m$. 
By setting  $g(z)=\log K(z,z),~z\in \Omega$, and using Lemma \ref{lemmatransformationrule},  we obtain
\begin{align*}
 \big({\partial_i}{\bar{\partial}_j} \log K( \varphi(z),\varphi(z))\big)_{i,j=1}^m
=   D\varphi(z)^{\rm tr}\big(\big({\partial_l}{\bar{\partial}_p}
\log K \big)(\varphi(z), \varphi(z))\big)_{l,p=1}^m \overbar{D\varphi (z)}.
\end{align*}

Combining this with \eqref{eqntransrule1}, we obtain
\begin{equation} \label{transrulecurv}
\big({\partial_i}{\bar{\partial}_j}\log K(z,z)\big)_{i,j=1}^m=
D\varphi(z)^{\rm tr}\big(\big({\partial_l}{\bar{\partial}_p}
\log K \big)(\varphi(z), \varphi(z))\big)_{l,p=1}^m \overbar{D\varphi (z)}.
\end{equation}

Multiplying $K(z,z)^t$ both sides and using the quasi-invariance  of $K$, a second time, we obtain
\begin{align*}
& \big( \;K(z,z)^t {\partial_i}{\bar{\partial}_j}\log K (z,z)\;\big)_{i,j=1}^m\\
&\quad\quad\quad=J(\varphi, z)^t D\varphi(z)^{\rm tr} K(\varphi(z),\varphi(z))^t 
\big(\big({\partial_l}{\bar{\partial}_p}
\log K \big)(\varphi(z), \varphi(z)\;\big)_{l,p=1}^m
\overbar{J(\varphi, z)^t D\varphi(z)}.
\end{align*}
Equivalently, we have
\begin{equation}
K^{t-2}(z,z)\mathbb K(z,z)= \mathbb J (\varphi, z) K^{t-2}(\varphi(z),\varphi(z))\mathbb K(\varphi(z),\varphi(z)) \mathbb J(\varphi, z)^*,
\end{equation}
where $\mathbb J(\varphi, z)=J(\varphi, z)^t D\varphi(z)^{\rm tr}$, $\varphi\in {\rm Aut}(\Omega),~ z\in \Omega$.
Therefore, polarizing both sides of the above equation, we have the desired conclusion.
\end{proof}

%

\begin{remark}
The function $J$ in the definition of quasi-invariant kernel is said to be a projective cocycle if it is a Borel map satisfying
\begin{equation}
J(\varphi\psi, z)=m(\varphi,\psi)J(\psi,z)J(\varphi, \psi z), 
~\varphi, \psi\in {\rm Aut}(\Omega),z\in \Omega,
\end{equation}
where $m:{\rm Aut}(\Omega)\times {\rm Aut}(\Omega)\to \mathbb T$ is a multiplier, that is, $m$ is Borel and satisfies the following properties:
\begin{itemize}
\item[\rm (i)]$m(e,\varphi)=m(\varphi,e)=1$, where $\varphi \in {\rm Aut}(\Omega)$ and $e$ is the identity in ${\rm Aut}(\Omega)$
\item[\rm (ii)]$m(\varphi_1,\varphi_2)m(\varphi_1\varphi_2,\varphi_3)=m(\varphi_1,\varphi_2\varphi_3)m(\varphi_2,\varphi_3)$, $\varphi_1,\varphi_2, \varphi_3\in {\rm Aut}(\Omega)$.
\end{itemize}
$J$ is said to be a cocycle if it is a projective cocycle with $m(\varphi,\psi)=1$ for all $\varphi,\psi$
in ${\rm Aut}(\Omega)$. 

If  $J:\rm Aut(\Omega)\times \Omega\to \mathbb C\setminus\{0\}$ in the Proposition \ref{proptransformation} is a cocycle, then it is not hard to verify that the function $\mathbb J$ is a projective co-cycle. Moreover, if $t$ is a positive integer, then $\mathbb J$ is also a cocycle.
\end{remark}

For the preceding to be useful, one must exhibit  non-negative definite kernels which are quasi-invariant.
%
%
It is known that the Bergman kernel $B_\Omega$ of any  bounded domain $\Omega$ is quasi-invariant with respect to $J$, where $J(\varphi,z)=\det D \varphi (z)$, $\varphi\in \rm{Aut}(\Omega), z\in \Omega$ . 

\begin{lemma}\label{lemBergmtrans}{\rm(\cite[Proposition 1.4.12]{Krantz})}
Let $\Omega\subset \mathbb C^m$ be a bounded domain and $\varphi:\Omega\to \Omega$ be a biholomorphic map. Then 
$$B_{\Omega}(z,w) = \det D \varphi(z) B_{\Omega}(\varphi(z),\varphi(w))\overbar {\det D \varphi (w)},~z, w \in \Omega.$$ 
\end{lemma}

The following proposition follows from combining  Proposition \ref{proptransformation}
and Lemma \ref{lemBergmtrans}, and therefore the proof is omitted.
\begin{proposition}\label{bergqua}
Let $\Omega$ be a bounded domain $\mathbb C^m$. If $t$ is in $G\mathcal W (B_{\Omega})$, then the kernel 
$$\boldsymbol B_\Omega^{(t)}(z,w):= \big(\; B_{\Omega}^t(z,w)\partial_i\bar{\partial}_j\log B_{\Omega}(z,w)\; \big)_{i,j=1}$$ 
is quasi-invariant with respect to 
$(\det D\varphi (z))^t D\varphi (z)^{\rm tr},~\varphi\in \rm Aut(\Omega),~z\in \Omega$.
\end{proposition}

 For a fixed but arbitrary $\varphi \in {\rm Aut}(\Omega)$, let $U_{\varphi}$ be the linear map on ${\rm Hol}(\Omega,\mathbb C^k)$ defined by
 \begin{equation}\label{eqnuphi}
 U_{\varphi}(f)=J\big(\varphi^{-1}, \cdot \big) f  \circ \varphi^{-1},\;\;f\in {\rm Hol}(\Omega,\mathbb C^k).
 \end{equation}
The following proposition is a basic tool in defining unitary representations of the automorphism group ${\rm Aut}(\Omega)$. The straightforward proof for the case of unit disc $\D$ appears in \cite{MisraKoranyihomog}. The proof for the general domain $\Omega$ follows in exactly the same way. 

\begin{proposition}
The linear map $U_{\varphi}$ is unitary on $(\hl, K)$ 
for all $\varphi$ in ${\rm Aut}(\Omega)$ if and only if the kernel $K$ is quasi-invariant with respect to $J$.
\end{proposition}

Let $Q: \Omega \to \mathcal M_k(\mathbb C)$ be a real analytic function such that $Q(w)$ is  positive definite for $w\in \Omega$.  Let $\mathcal H$ be the Hilbert space of $\mathbb C^k$ valued holomorphic functions on $\Omega$ which are square integrable with respect to $Q(w) dV(w)$, that is, 
$$
\hl=\big\{f\in \rm Hol(\Omega, \mathbb C^k):\|f\|^2:= \int_\Omega \langle Q(w)f(w), f(w) \rangle_{\mathbb C^k} dV(w) <\infty\big\},
$$
where $dV$ is the normalized volume measure on $\mathbb C^m$. Assume  that the constant functions are in $\mathcal H$. 
The operator $U_\varphi$, defined in \eqref{eqnuphi} is unitary if and only if 
\begin{eqnarray*}
\|U_\varphi f\|^2 &=& \int_{\Omega} \langle Q(w) (U_\varphi f)(w), (U_\varphi f)(w) \rangle dV(w)\\   
&=& \int_{\Omega} \langle \overline{J(\varphi^{-1},w)}^{\rm \,tr} Q(w) J(\varphi^{-1},w) f(\varphi^{-1}(w)), f(\varphi^{-1}(w))  \rangle dV(w)\\
&=&\int_{\Omega} \langle Q(w) f(w), f(w) \rangle dV(w),
\end{eqnarray*}
that is, if and only if $Q$ transforms according to the rule
\begin{equation}\label{Q}
\overline{J(\varphi^{-1},w)}^{\rm \,tr} Q(w) J(\varphi^{-1},w) = Q(\varphi^{-1}(w)) | {\det} (D\varphi^{-1})(w)|^2.
\end{equation}
Set $J(\varphi^{-1},w)= \det (D\varphi^{-1}(w))^t D\varphi^{-1}(w)^{\rm tr}$ and $Q^{(t)}(w):={B_\Omega(w,w)^{1-t}} \mathcal K(w,w)^{-1}$, where $\mathcal K(z,w) := \big (\partial_i \bar{\partial}_j \log B_\Omega(z,w) \big )_{i,j=1}^m$, $t > 0$. Then $Q^{(t)}$ transforms according to the rule \eqref{Q} since $\mathcal K$ transforms  according  to \eqref{transrulecurv} and $B_\Omega$ transfomrs as in Lemma  \ref{lemBergmtrans}. If for some $t> 0$, the Hilbert space $L_{\rm hol}^2(\Omega, Q^{(t)}\,dV)$ determined by the measure is nontrivial, then the corresponding reproducing kernel is of the form $B_\Omega^t(z,w) \mathcal K(z,w)$. 

Let $\Omega$ be a bounded symmetric domain in $\mathbb C^m$.  Note that if  $K:\Omega\times\Omega\to \mathcal M_k(\mathbb C)$ is a quasi-invariant kernel with respect to some $J$ and the commuting tuple ${\boldsymbol M}_z=(M_{z_1},\ldots,M_{z_m})$ on $(\hl, K)$ is bounded, then the commuting tuple 
$\boldsymbol M_{\varphi}:=(M_{\varphi_1},\ldots, M_{\varphi_m})$ is unitarily equivalent to $\boldsymbol M_z$ via the unitary map $U_{\varphi}$, where $\varphi=(\varphi_1,\ldots,\varphi_m)$ is in $~{\rm Aut}(\Omega)$.
If $t$ is in $G\mathcal W(B_\Omega)$ and the operator of multiplication $M_{z_i}$ by the coordinate function $z_i$ is bounded on the Hilbert space $(\mathcal H, B_\Omega^{t/2} )$, then it follows from Corollary \ref{thmboundedness} that  the operator $M_{z_i}$ on the Hilbert space $\big (\mathcal H, \boldsymbol B_\Omega^{(t)})$ is bounded as well.  Therefore, in the language of \cite{Misra-Sastry}, we conclude that the multiplication tuple $\boldsymbol {M}_z$ on $(\hl, \boldsymbol B_\Omega^{(t)})$ is homogeneous with respect to the group ${\rm Aut}(\Omega)$. In particular, if $\Omega$ is the Euclidean unit ball in $\mathbb C^m$, and $t$ is any positive real number, then the multiplication tuple $\boldsymbol{M}_z$  on $(\hl, B_{\mathbb B_m}^{t/2})$ is bounded.  Also, from Theorem \ref{thmGWall}, it follows that $\boldsymbol B_{\mathbb B_m}^{(t)}$ is non-negative definite.  Consequently, the commuting $m$ - tuple of operators $\boldsymbol{M}_z$ must be homogeneous with respect to the group ${\rm Aut}(\mathbb B_m)$.


\begin{thebibliography}{10}





\bibitem{Aro}
N.~Aronszajn, \emph{Theory of reproducing kernels}, Trans. Amer. Math. Soc.
  \textbf{68} (1950), 337--404. 
%


%

\bibitem{constantchar}
B.~Bagchi and G.~Misra, \emph{Constant characteristic functions and homogeneous
  operators}, J. Operator Theory \textbf{37} (1997),51--65.
  
 
  
  \bibitem{homoshift}
  \bysame, \emph{The homogeneous shifts},
    J. Funct. Anal. \textbf{204}(2003),  293--319.
     
    


\bibitem{infinitelydivisiblemetric}
S.~Biswas, D.~K. Keshari, and G.~Misra, \emph{Infinitely divisible metrics and
  curvature inequalities for operators in the {C}owen-{D}ouglas class}, J.
  Lond. Math. Soc. (2) \textbf{88} (2013), 941--956. 

 

\bibitem{onhomocontraction}
\bysame, \emph{On homogeneous contractions and unitary representations of
  {${\rm SU}(1,1)$}}, J. Operator Theory \textbf{30} (1993), 109--122.




\bibitem{CD}
M.~J. Cowen and R.~G. Douglas, \emph{Complex geometry and operator theory},
  Acta Math. \textbf{141} (1978), no.~3-4, 187--261. 

\bibitem{CDopen}
\bysame, \emph{Operators possessing an open set of eigenvalues}, Functions,
  series, operators, {V}ol. {I}, {II} ({B}udapest, 1980), Colloq. Math. Soc.
  J\'anos Bolyai, vol.~35, North-Holland, Amsterdam, (1983), 323--341.
  



\bibitem{Curtosalinas}
R.~E. Curto and N.~Salinas, \emph{Generalized bergman kernels and the
  cowen-douglas theory}, American Journal of Mathematics \textbf{106} (1984), 447--488.

\bibitem{equivofquotient}
R.~G. Douglas and G.~Misra, \emph{Equivalence of quotient {H}ilbert modules},
  Proc. Indian Acad. Sci. Math. Sci. \textbf{113} (2003), 281--291.


\bibitem{D-M-V}
R.~G. Douglas, G.~Misra, and C.~Varughese, \emph{On quotient modules---the case
  of arbitrary multiplicity}, J. Funct. Anal. \textbf{174} (2000), 364--398. 

\bibitem{Douglasmodule}
R.~G. Douglas and V.~I. Paulsen, \emph{Hilbert modules over function algebras}, Longman Sc \& Tech, 1989.

\bibitem{Wallachset}
J.~Faraut and A.~Kor\'anyi, \emph{Function spaces and reproducing kernels on
  bounded symmetric domains}, J. Funct. Anal. \textbf{88} (1990), 64--89. 

\bibitem{Ferguson-Rochberg}
S.~H. Ferguson and R.~Rochberg, \emph{Higher order {H}ilbert-{S}chmidt {H}ankel forms and tensors of analytic kernels}, Math. Scand. \textbf{96} (2005),  117--146.

\bibitem{SGhara}
S. ~Ghara,  {\em  Decomposition of the tensor product of Hilbert modules via the jet construction and weakly homogeneous operators}, PhD thesis, Indian Institute of  Science, 2018.


\bibitem{HLZ}
S. ~Hwang, Y. ~Liu, and G. ~Zhang, \emph{Hilbert spaces of tensor-valued holomorphic functions on  the  unit ball of  $\mathbb C^n$}, 
Pac. J. Math. \textbf{214} (2004), 303  - 322. 

\bibitem{Flagstructure}
K.~Ji, C.~Jiang, D.~K. Keshari, and G.~Misra, \emph{Rigidity of the flag
  structure for a class of {C}owen-{D}ouglas operators}, J. Funct. Anal.
  \textbf{272} (2017), 2899--2932. 


\bibitem{MisraKoranyihomog}
A.~Kor\'anyi and G.~Misra, \emph{Homogeneous operators on {H}ilbert spaces of
  holomorphic functions}, J. Funct. Anal. \textbf{254} (2008), 2419--2436. 

\bibitem{Krantz}
S.~G. Krantz, \emph{Function theory of several complex variables}, second ed.,
  The Wadsworth \& Brooks/Cole Mathematics Series, Wadsworth \& Brooks/Cole
  Advanced Books \& Software, Pacific Grove, CA, 1992. 

\bibitem{JP}
J. ~Peetre, \emph{Reproducing formulae for holomorphic tensor fields}, 
Boll. Un. Mat. Ital. B \textbf{(7)  2}  (1988),  345–359. 

\bibitem{curvandbackwardshift}
G.~Misra, \emph{Curvature and the backward shift operators}, Proc. Amer. Math.
  Soc. \textbf{91} (1984), 105--107. 

\bibitem{GMCI}
\bysame, \emph{Curvature inequalities and extremal properties of bundle
  shifts}, J. Operator Theory \textbf{11} (1984),  305--317. 


\bibitem{Misra-Sastry}
G.~Misra and N.~S.~N. Sastry, \emph{Homogeneous tuples of operators and
  representations of some classical groups}, J. Operator Theory \textbf{24}
  (1990), 23--32.


\bibitem{Misra-Upmeier}
G.~Misra and H. Upmeier, \emph{Homogeneous vector bundles and intertwining operators for symmetric domains}, Adv. Math. \textbf{303} (2016), 1077--1121.

\bibitem{PaulsenRaghupati}
V.~I. Paulsen and M.~Raghupathi, \emph{An introduction to the theory of
  reproducing kernel {H}ilbert spaces}, Cambridge Studies in Advanced
  Mathematics \textbf{152}, Cambridge University Press, Cambridge, (2016).
  

\bibitem{Rudinunitball}
W.~Rudin, \emph{Function theory in the unit ball of {${\bf C}^{n}$}},
  Grundlehren der Mathematischen Wissenschaften [Fundamental Principles of
  Mathematical Science] \textbf{241}, Springer-Verlag, New York-Berlin (1980).

\bibitem{Sal}
N.~Salinas, \emph{Products of kernel functions and module tensor products},
  Topics in operator theory, Oper. Theory Adv. Appl. \textbf{32}, Birkh\"auser,   Basel (1988), 219--241. 






\end{thebibliography}
\end{document}